\documentclass[11pt,reqno]{amsart}
\usepackage[margin=1in]{geometry}
\usepackage{algorithm,algpseudocode}
\usepackage[fleqn,tbtags]{mathtools}
\usepackage{pifont,enumerate}
\usepackage{verbatim}
\usepackage{tikz}
\usepackage{filecontents}
\usepackage{placeins}
\usepackage{mlmodern}
\usepackage{extarrows}
\usepackage{graphicx}
\usepackage{calc}
\usepackage{adjustbox}
\usepackage{amssymb}
\usepackage{tabu}
\usepackage{cases}
\usepackage{subcaption}
\usepackage{bbm, dsfont, mathrsfs}
\usepackage[normalem]{ulem} 
\usepackage{hyperref,xcolor}
\usepackage{makecell}
\usepackage[shortlabels]{enumitem}
\hypersetup{
    colorlinks,
    linkcolor={red!50!black},
    citecolor={blue!50!black},
    urlcolor={blue!80!black}
}

\DeclarePairedDelimiterX\innerp[2]{\langle}{\rangle}{#1,#2}

\newcommand\restr[2]{{
  \left.\kern-\nulldelimiterspace 
  #1 
  \vphantom{\big|} 
  \right|_{#2} 
  }}

\newcommand{\tp}{{\scriptscriptstyle\mathsf{T}}}

\DeclareMathOperator{\rank}{rank}
\DeclareMathOperator{\tr}{tr}

\DeclareMathOperator*{\argmin}{argmin}

\DeclareMathOperator{\supp}{supp}
\newcommand{\E}{\mathbb{E}}
\DeclareMathOperator{\Var}{Var}

\theoremstyle{definition}
\newtheorem{theorem}{Theorem}[section]

\newtheorem{lemma}[theorem]{Lemma}
\newtheorem{corollary}[theorem]{Corollary}
\newtheorem{proposition}[theorem]{Proposition}
\newtheorem{example}[theorem]{Example}
\newtheorem{remark}[theorem]{Remark}
\numberwithin{equation}{section}

\newcommand{\1}{\mathbbm{1}}

\newcommand{\C}{\mathbb{C}}

\newcommand{\N}{\mathbb{N}}
\newcommand{\Z}{\mathbb{Z}}
\newcommand{\R}{\mathbb{R}}

\newcommand{\G}{\mathsf{G}}
\newcommand{\g}{\mathfrak{g}}

\newcommand*{\method}[1]{#1}
\newcommand*{\overeqU}[1]{\ensuremath{\mathrel{\overset{\method{#1}}{=}}}}
\newcommand*{\overeq}[1]{\mathrel{\overset{\method{#1}}{\resizebox{\widthof{\kern1.25pt\overeqU{\method{#1}}}}{\heightof{$=$}}{$=$}}}}

\newcommand{\range}{\mathcal{R}}
\newcommand{\dR}{\mathsf{dR}}
\renewcommand{\O}{\mathsf{O}}
\newcommand{\U}{\mathsf{U}}

\setlist[description]{font=\normalfont\itshape,labelindent=2em}

\allowdisplaybreaks

\begin{document}

\title{Singular value decomposition of unbounded operators}

\author{Rongbiao Thomas Wang}
\address{Computational and Applied Mathematics, University of Chicago, Chicago, IL 60637}
\email{rbwang@uchicago.edu}

\author{Haoming Wang}
\address{Department of Statistics, Columbia University, New York, NY 10027}
\email{hw3129@columbia.edu}

\author{Lek-Heng Lim}
\address{Computational and Applied Mathematics Initiative, University of Chicago, Chicago, IL 60637}
\email{lekheng@uchicago.edu}

\begin{abstract}
The singular value decomposition has been established for matrices, Hilbert--Schmidt operators, trace-class operator, compact operators, and bounded operators, but surprisingly not for unbounded operators. Unfortunately, most interesting operators in applied math are unbounded, as any operators involving some form of derivatives —-- gradient, exterior derivatives, Laplacians, Fourier and other transforms of derivatives, Hamiltonians, etc. --— are likely unbounded. In this article, we fill in this last missing piece by establishing the existence of singular value decompositions for unbounded operators in three natural forms: multiplication-operator, direct-integral, and operator-valued-measure. We show it inherits classical properties of finite-dimensional singular value decomposition including approximation results, relationships with fundamental subspaces, and the Moore--Penrose inverse. This discovery opens the door
to the singular value decompositions of a myriad of well-known unbounded operators in mathematics, physics, statistics, and finnance --- gradients on Euclidean spaces and manifolds, Petrov--Galerkin method, finite-difference operators, Hilbert--Schmidt operators, Hilbert complexes, supersymmetric quantum mechanics, Sturm--Liouville theory, nonparametric density estimation, and the Black--Scholes equation. The resulting decompositions reveal a number of novel insights, among many others: bosonic and fermionic states in supersymmetric quantum mechanics arise as left and right singular vectors of generalized ladder operators; the Riesz transform appears as the left singular operator of the Euclidean gradient; and the Hodge decomposition follows directly from the singular value decompositions of the exterior derivatives.
\end{abstract}



\maketitle

\section{Introduction}\label{sec:intro}

The singular value decomposition is one of the fundamental tools of applied mathematics. Over the past century and a half, its development has progressed through increasingly general classes of operators. The first pioneers were Beltrami, Jordan, and Sylvester in the late nineteenth century \cite{Beltrami1873,Jordan1874,Sylvester1889a,Sylvester1889b,Sylvester1889c}, who studied the reduction of bilinear forms by orthogonal substitutions. At the beginning of the twentieth century, interest in singular value decomposition for infinite-dimensional operators grew out of the theory of integral equations and integral operators \cite{Fredholm1903,Hilbert1904}, culminating in Schmidt's singular value expansion for Hilbert--Schmidt operators \cite{Schmidt1907}. This development led to the singular value decomposition of compact operators, sometimes known as the Schmidt expansion \cite{GK1969}, and subsequently to operator classes defined through their singular values, such as trace-class operators, whose singular values are summable \cite{vonNeumannSchatten1946,vonNeumannSchatten1948,Grothendieck1955,Schatten1960,Lidskii1959}. The next natural step was to pass from compact to general bounded operators, where the singular values need no longer be discrete. Two principal approaches emerged to address this difficulty: operator ideals and $s$-numbers \cite{Pietsch1980}, and measure-theoretic singular value expansions \cite{CraneGockenbach2020}.

The natural next step is to pass from bounded to unbounded operators. Yet, to the best of our knowledge, a comparably general theory of singular value decomposition for unbounded operators has remained missing beyond special cases \cite{KR2025}. This gap is particularly important because many of the fundamental operators arising in applied mathematics are naturally unbounded. Differential operators, including gradients, Laplacians, Sturm--Liouville operators, Hamiltonians in quantum mechanics, and operators governing partial differential equations, inverse problems, and mathematical finance, are among the most basic examples. In this article, we fill in this gap. 

Taking the measure-theoretic approach in the same spirit as \cite{CraneGockenbach2020}, we show that, analogous to the spectral decomposition, an unbounded operator $A: \mathcal{D}(A) \subseteq \mathcal{H} \to \mathcal{K}$ has its singular value decomposition in three different forms:
\begin{description}
\item [Multiplication operator] There exists a measurable space $(M,\mu)$ such that $A = U T_\sigma V^*$ where $U$ and $V$ are partial isometries from $L^2_\mu$ to $\mathcal{K}$ and $\mathcal{H}$ respectively and $T_\sigma$ is the multiplication operator by $\sigma \geq 0$ almost everywhere. 
\item [Direct integral] There exists a direct integral $\int_{\mathsf{\Sigma}(A)}^\oplus \mathcal{H}_s \, d\eta(s)$ over the set of singular values of $A$ such that $A = U \Sigma V^*$ where $U$ and $V$ are partial isometries from the direct integral to $\mathcal{K}$ and $\mathcal{H}$ and $\Sigma$ is a multiplication operator of measurable sections.
\item [Operator-valued measure] There exists an operator-valued measure $\nu$ such that $A$ decomposes as $A = \int_{[0,\infty)} \lambda^{1/2}  \, d\nu(\lambda)$.
\end{description}
Moreover, the singular value decomposition of unbounded operators retains many familiar properties of the matrix singular value decomposition. In Section~\ref{sec:properties}, we show that it admits truncation approximations in the spirit of the Eckart--Young theorem \cite{Eck}, is preserved under taking adjoints, recovers the fundamental subspaces via orthogonal projections, and yields the Moore--Penrose inverse.

The significance of the extension becomes apparent when it is applied to familiar unbounded operators arising in analysis, numerical computation, physics, statistics, and mathematical finance. These decompositions not only illustrate the general theory, but also uncover a spectrum of novel insights:
\begin{enumerate}[\normalfont(i)]
\item On $\R^n$, the set of singular values of the gradient $\nabla$ is entire half-line $[0,\infty)$. The right singular operator is the Fourier transform, while the left singular operator is a composition of the inverse Fourier transform and the Riesz transform.
\item On a bounded domain $\Omega\subseteq\R^n$, the singular values of $\nabla$ become discrete, and the corresponding singular vectors reflect the geometry of the domain.
\item On the unit sphere $S^2$, the spherical harmonics arise naturally as the singular vectors of the gradient $\nabla_{S^2}$.
\item On a compact Lie group $\mathsf{G}\subseteq\mathsf{U}(n)$, the singular value decomposition of the gradient $\nabla_\mathsf{G}$ is determined by the Casimir operator, with matrix coefficients of irreducible representations as right singular vectors and, for class functions, irreducible characters.
\item The singular value decompositions of the differentials in a closed Hilbert complex recover the Hodge decomposition. For the de Rham complex, the singular operators naturally encode the Betti numbers and harmonic forms, while the singular values yield a new formula for the analytic torsion.
\item The right and left singular vectors of a differential operator provide natural subspaces in the Petrov--Galerkin method.
\item For the finite-difference operator $\nabla_h$ of $\nabla$, the singular value decomposition converges to that of $\nabla$: The right singular operator is the same Fourier transform, the left singular operators converge strongly, and the singular value functions converge pointwise.
\item The singular value decomposition of the two-sided operator $X \mapsto AXC^*$ for a Hilbert--Schmidt operator $X$ is given by the singular value decompositions of $A$ and $C$.
\item In supersymmetric quantum mechanics, the Hamiltonian is factorized as $\mathsf H_1=A^*A$, while reversing the factors produces the partner Hamiltonian $\mathsf H_2=AA^*$. Physically, the bosonic and fermionic energy states are exactly the left and right singular vectors of $A$. In particular, they are described by the Hermite functions in the special case of the quantum harmonic oscillator, where $A$ is the annihilation operator. \label{item:ham}
\item Similarly, the Sturm--Liouville operator is factorized into $L = B^*B$, which is closely related to the factorization of Hamiltonian in \ref{item:ham}. Special functions such as Laguerre and Jacobi polynomials arise as singular vectors of $B$.
\item The singular value decomposition of the integral operator over $L^2(0,\infty)$ provides a natural nonparametric density estimator.
\item The risk-free rate is the smallest singular value of the Black--Scholes operator.
\end{enumerate}

\section{Notations and conventions}\label{sec:notations}
We denote the positive integers by $\mathbb{N}$. In this article, we will assume all Hilbert spaces are separable and complex unless otherwise specified. Let $\mathcal{H}$ and $\mathcal{K}$ be Hilbert spaces. We denote by $\mathcal{L}(\mathcal{H},\mathcal{K})$ the Banach space of bounded operators from $\mathcal{H}$ to $\mathcal{K}$ and shorten it to $\mathcal{L}(\mathcal{H})$ if $\mathcal{K} = \mathcal{H}$. The identity operator in $\mathcal{L}(\mathcal{H})$ is denoted by $I_\mathcal{H}$, where the subscript will be dropped if the context is clear.

Our goal is to study the singular value decomposition of a closed and densely defined unbounded operator $A \colon \mathcal{D}(A) \subseteq \mathcal{H} \to \mathcal{K}$, i.e., the domain $\mathcal{D}(A)$ is a dense subspace of $\mathcal{H}$ and if $\varphi_n \in \mathcal{D}(A)$ for all $n \in \mathbb{N}$, $\varphi_n \to \varphi$, and $A\varphi_n \to \psi$, then $\varphi \in \mathcal{D}(A)$ and $A\varphi = \psi$. We maintain the assumption that all operators under consideration are
closed and densely defined throughout the article. Bounded operators are included as the special case $\mathcal D(A)=\mathcal H$. The domain $\mathcal{D}(A)$ is complete with respect to the norm $\lVert \,\cdot \,\rVert_\mathcal{H}$ only if $A$ is a bounded operator. However, $\mathcal{D}(A)$ is always complete with respect to the graph norm
\[ 
\lVert \varphi \rVert_A^2 \coloneqq \lVert \varphi \rVert_\mathcal{H}^2 + \lVert A\varphi  \rVert_\mathcal{K}^2,
\]
so $\mathcal{D}(A)$ is a Hilbert space with the inner product $\langle \varphi , \psi \rangle_A = \langle \varphi, \psi\rangle_\mathcal{H}+\langle A\varphi, A\psi\rangle_\mathcal{K}$ for $\varphi,\psi \in \mathcal{D}(A)$.

We denote the range and kernel of $A$ by $\mathcal{R}(A)$ and $\mathcal{N}(A)$, respectively. When $\mathcal{H}=\mathcal{K}$, we denote the spectrum of $A$ by $\mathsf{\Lambda}(A)$. The adjoint operator $A^*$ is defined by $\langle \psi, A\varphi\rangle_\mathcal{K} = \langle A^*\psi, \varphi \rangle_{\mathcal{H}}$ for all $\varphi \in \mathcal{D}(A)$ with domain
\[
\mathcal{D}(A^*) \coloneqq \{\psi \in \mathcal{K}: \varphi \mapsto \langle \psi, A\varphi \rangle_{\mathcal{K}} \text{ is bounded on } \mathcal{D}(A)\text{ in }\lVert\,\cdot\,\rVert_{\mathcal H}\}.
\]

Let $B: \mathcal{D}(B) \subseteq \mathcal{H} \to \mathcal{K}$ be another unbounded operator. We say $A \subseteq B$ if both $\mathcal{D}(A) \subseteq \mathcal{D}(B)$ and $A\varphi = B\varphi$ for all $\varphi \in \mathcal{D}(A)$. If the set inclusion is a set equality, we write $A = B$. The sum of the two unbounded operators is defined by
\[A+B: \mathcal{D}(A) \cap \mathcal{D}(B) \subseteq \mathcal{H} \to \mathcal{K}, \quad \varphi \mapsto A\varphi + B\varphi.\]

A (linear) isometry is an operator $U \in \mathcal{L}(\mathcal{H},\mathcal{K})$ such that $\langle U\varphi, U\psi \rangle_{\mathcal{K}} = \langle \varphi, \psi \rangle_{\mathcal{H}}$ for all $\varphi, \psi \in \mathcal{H}$. In particular, $U$ is injective and satisfies $U^*U = I_{\mathcal{H}}$. If $U$ is furthermore surjective, it is a unitary operator and satisfies $UU^* = I_\mathcal{K}$.

Let $(M, \mu)$ be a $\sigma$-finite measure space. The multiplication operator for a measurable function $g \colon M \to \mathbb{C}$ is 
\[
T_g: \bigl\{ f \in L^2_\mu  : g f \in L^2_\mu  \bigr\} \subseteq L^2_\mu  \to L^2_\mu , \quad f \mapsto g f.
\]
All multiplication operators are closed and if $g$ is real-valued, $T_g$ is self-adjoint. To avoid confusion with the identity operator, we will denote by $\1_E$ the indicator function of a measurable subset $E \subseteq M$.

\section{Singular value decomposition of unbounded operators}\label{sec:svd}

Analogous to the spectral decomposition for unbounded self-adjoint operators, the singular value decomposition of unbounded operators takes three forms: multiplication operator, direct integral, and operator-valued measure, which we establish in this section. As a reminder, we assume $A \colon \mathcal{D}(A) \subseteq \mathcal{H} \to \mathcal{K}$ is a closed, densely defined linear operator throughout this section. 

\subsection{Singular value decomposition in multiplication operator form}\label{sec:svd-multiply}

We first prove the singular value decomposition in multiplication operator form, which has two versions: condensed and full. This is the most direct generalization of the matrix singular value decomposition, which also admits both versions.

\begin{theorem}[Condensed singular value decomposition]\label{thm:condensed-mult}
There exist a $\sigma$-finite Borel measure space $(M, \mu)$, isometries $U \in \mathcal{L} (L^2_\mu, \mathcal{K})$, $V  \in \mathcal{L}( L^2_\mu, \mathcal{H})$, and a measurable function $\sigma \colon M \to [0, \infty)$ that is positive almost everywhere such that 
\[\mathcal{D}(A) = \{x \in \mathcal{H} : V^*x \in \mathcal{D}(T_\sigma)\} \quad \text{and } Ax = U T_\sigma V^* x \quad \text{for all } x \in \mathcal{D}(A).\]
\end{theorem}

\begin{theorem}[Full singular value decomposition]\label{thm:full-mult}
There exist measure spaces $(M_1,\mu_1)$ and $(M_2,\mu_2)$ with a common measurable subset $M$ such that $\mu_1\vert_{M} = \mu_2\vert_{M} \eqqcolon \mu$, unitary operators  $U \in \mathcal{L}(L^2_{\mu_2}, \mathcal{K})$, $V \in \mathcal{L}(L^2_{\mu_1}, \mathcal{H})$, and a measurable function $\sigma \colon M \to [0, \infty)$ that is positive almost everywhere such that for $\Sigma \colon \mathcal{D}(\Sigma) \subseteq L^2_{\mu_1} \to L^2_{\mu_2}$ defined by
\[
  \mathcal{D}(\Sigma) = \{f \in L^2_{\mu_1} : \sigma (f\vert_M) \in L^2_\mu\}, \quad
  \Sigma f =
  \begin{cases}
  \sigma (f\vert_M) & \text{on } M, \\
  0 & \text{on } M_2 \setminus M,
  \end{cases}
\]
we have
\[
\mathcal{D}(A) = \{x \in \mathcal{H} : V^*x \in \mathcal{D}(\Sigma)\} \quad \text{and } Ax = U \Sigma V^*x \quad \text{for all } x \in \mathcal{D}(A).
\]
\end{theorem}

We start with an instrumental lemma.
\begin{lemma}\label{lem:dense}
The following holds:
\begin{enumerate}[\normalfont(i)]
   \item \label{item:dense} The domain $\mathcal{D}(A^*A)$ is dense in $\mathcal{D}(A)$ with respect to $\lVert\,\cdot\,\rVert_A$ and dense in $\mathcal{H}$ with respect to $\lVert\,\cdot\,\rVert_\mathcal{H}$.
    \item \label{item:dense2} If $(M, \mu)$ is a measure space, and $\theta \colon M \to [0, \infty)$ is a measurable function that is positive almost everywhere, then $\{ f \in L^2_\mu : \theta f \in L^2_\mu \}$ is dense in $L^2_\mu$.
\end{enumerate}
\end{lemma}
\begin{proof}
\ref{item:dense} follows immediately from \cite[p.~195]{Pedersen1989} and \ref{item:dense2} is a direct consequence of \cite[Lemma~2.1]{Gockenbach2016} and \cite[Theorem~10.10]{Hall2013}.
\end{proof}

To prove the condensed singular value decomposition, we first establish the following case. 

\begin{theorem}[Condensed singular value decomposition with trivial kernel]\label{thm:svd1}
Let $A \colon \mathcal{D}(A) \subseteq \mathcal{H} \to \mathcal{K}$ be a closed, densely defined linear operator with $\mathcal{N}(A) = \{0\}$. Then there exist a $\sigma$-finite Borel measure space $(M, \mu)$, isometry $U \in \mathcal{L} (L^2_\mu, \mathcal{K})$, unitary operator $V  \in \mathcal{L}( L^2_\mu, \mathcal{H})$, and a measurable function $\sigma \colon M \to [0, \infty)$ that is positive almost everywhere such that 
\[\mathcal{D}(A) = \{x \in \mathcal{H} : V^*x \in \mathcal{D}(T_\sigma)\} \quad \text{and } Ax = U T_\sigma V^* x \quad \text{for all } x \in \mathcal{D}(A).\]
\end{theorem}

\begin{proof}
Since $A$ is closed and densely defined, the operator $A^*A$ is self-adjoint
with domain
\[\mathcal{D}(A^*A) = \{x \in \mathcal{D}(A): Ax \in \mathcal{D}(A^*)\} \subseteq \mathcal{D}(A).\]
By Lemma~\ref{lem:dense}, $\mathcal{D}(A^*A)$ is dense in $\mathcal{H}$ with respect to $\lVert\,\cdot\, \rVert_\mathcal{H}$. By spectral theorem in multiplication operator form \cite[Theorem~10.10]{Hall2013}, there exist a $\sigma$-finite measure space $(M, \mu)$, a unitary operator $V \in \mathcal{L}(L^2_\mu, \mathcal{H})$, and a measurable function $\theta \colon M \to \R$ such that
\[
A^*Ax = V T_\theta V^*x \quad \text{ for all } x \in \mathcal{D}(A^*A)
\] 
and
\begin{equation}\label{eq:domain-A*A-equiv}
  V^*\bigl(\mathcal{D}(A^*A)\bigr) = \bigl\{f \in L^2_\mu  : \theta f \in L^2_\mu \bigr\} = \mathcal{D}(T_\theta).
\end{equation}

\underline{\textsc{Claim}~(i)}: $\theta > 0$ almost everywhere.

Since $(M,\mu)$ is $\sigma$-finite, there exist measureable sets $E_k\subseteq M$ with $E_k\uparrow M$ and $\mu(E_k)<\infty$ for all $k \in \N$. Suppose there is a measurable subset $F \subseteq M$ such that $\mu(F)>0$ and $\theta<0$ on $F$. Then there exists an $k$ such that  $0<\mu(F\cap E_k)<\infty$. Let $F_j=\{a \in M: -j \le \theta(a) \le -1/j \}$, then 
\[
0<\mu(F\cap E_k) \leq \mu \Big(\bigcup_j F_j\cap E_k\Big) \leq \sum_j \mu(F_j\cap E_k).
\]
Consequently, there exists $j$ such that $\mu(F_j\cap E_k)>0$, and 
\[
\int (\theta \1_{F_j\cap E_k})^2 \, d\mu =\int_{F_j\cap E_k} \theta^2 \, d\mu \leq j^2 \mu(F_j\cap E_k) <\infty,
\]
so $\1_{F_j\cap E_k} \in V^* \big(\mathcal{D}(A^*A)\big)$. We then obtain a contradiction
\begin{align*}
\langle A V \1_{F_j\cap E_k} &, A V \1_{F_j\cap E_k} \rangle_\mathcal{K} = \langle V^* A^*A V \1_{F_j\cap E_k}, \1_{F_j\cap E_k} \rangle_{L^2_\mu } \\
&= \langle T_\theta \1_{F_j\cap E_k}, \1_{F_j\cap E_k} \rangle_{L^2_\mu }  =\int_{F_j\cap E_k} \theta \,d\mu \leq -\frac{1}{j} \mu(F_j\cap E_k) < 0.
\end{align*}
Therefore, $\theta \ge 0$ almost everywhere. 

Let $E \coloneqq \{a \in M : \theta(a) = 0 \}$. If $\mu(E) > 0$, then similarly, there is $k$ such that $0<\mu(E\cap E_k)<\infty$. The indicator function $\1_{E\cap E_k}$ is nonzero in $L^2_\mu$. Since
\[\int(\theta \1_{E\cap E_k})^2 \,d\mu= \int_{E\cap E_k} \theta^2 \,d\mu = 0,\]
we obtain $\1_{E\cap E_k} \in V^*(\mathcal{D}(A^*A))$. Thus, $V\1_{E\cap E_k} \neq 0$ in $\mathcal{D}(A^* A)$. However,
\[
A^*A V \1_{E\cap E_k} = V T_\theta V^* V \1_{E\cap E_k} = V T_\theta \1_{E\cap E_k} = V (\theta \1_{E\cap E_k}) = 0,
\]
contradicting the injectivity of $A^*A$ given by $\mathcal{{N}}(A)=\{0\}$. Thus, $\theta > 0$ almost everywhere.

\underline{\textsc{Claim}~(ii)}: Let $\sigma \coloneqq \sqrt{\theta}$. There exists an isometry $U \in \mathcal{L}(L^2_\mu,\mathcal{K})$ such that $Ax=U T_\sigma V^*x$ for all $x\in\mathcal{D}(A^*A)$.

Consider
\[
\mathcal{D}_0 \coloneqq \{f\in L^2_\mu: f/\sigma\in\mathcal{D}(T_\theta)\} = \{ f\in L^2_\mu: f/\sigma\in L^2_\mu, \sigma f\in L^2_\mu\}.
\]
Since $\sigma>0$ almost everywhere, consider $G_k=\{1/k\le\sigma\le k\}$ for $k=1,2,\dots$. Since $G_k\uparrow M$, $f\1_{G_k}\to f$ in $L^2_\mu$ for any $f\in L^2_\mu$ by monotone convergence. Moreover, for each $k$,
\[
\lvert f\1_{G_k}/\sigma\vert \le k\1_{G_k} \lvert f\rvert \in L^2_\mu,\quad \lvert f\1_{G_k}\sigma \rvert\le k\1_{G_k} \lvert f \rvert\in L^2_\mu,
\]
so $\mathcal{D}_0$ is dense in $L^2_\mu$.

For $f\in\mathcal{D}_0$, $V(f/\sigma)\in \mathcal{D}(A^*A)\subseteq \mathcal{D}(A)$ by construction. Define $\widetilde{U} \colon \mathcal{D}_0 \to \mathcal{K}$ by
\[
  \widetilde{U} f \coloneqq  A V (f/\sigma), \quad f \in \mathcal{D}_0.
\]
Since
\begin{align*}
    \lVert \widetilde{U} f\rVert_{\mathcal{K}}^2
    &= \lVert A V (f/\sigma)\rVert_{\mathcal{K}}^2
    = \langle A^*A V (f/\sigma), V(f/\sigma) \rangle_{\mathcal{H}} \\
    &= \langle V T_\theta (f/\sigma), V(f/\sigma) \rangle_{\mathcal{H}}
    = \langle T_\theta (f/\sigma), (f/\sigma) \rangle_{L^2_\mu}
    = \int \theta\,\lvert f/\sigma\rvert^2\,d\mu
    = \lVert f\rVert^2_{L^2_\mu},
\end{align*}

$\widetilde{U}$ is isometric on the dense subspace $\mathcal{D}_0 \subseteq L^2_\mu $ and therefore extends uniquely to an isometry $U$ in $\mathcal{L}(L^2_\mu , \mathcal{K})$.  

For $x\in\mathcal{D}(A^*A)$, let $g=V^*x\in \mathcal{D}(T_\theta)$ and $f\coloneqq\sigma g$. Since 
\[
\int \sigma^2\lvert V^*x\rvert^2\,d\mu
= \int \theta\lvert V^*x\rvert^2\,d\mu
\le \int \lvert V^*x\rvert^2\,d\mu
+ \int \theta^2\lvert V^*x\rvert^2\,d\mu
= \lVert V^*x\rVert_{L^2_\mu}^2
+ \lVert T_\theta V^*x\rVert_{L^2_\mu}^2 < \infty.
\]
$T_\sigma V^*x\in\mathcal{D}_0$. Thus, $UT_\sigma V^*x = AVV^*x=Ax$.

\underline{\textsc{Claim}~(iii)}: $V\vert_{\mathcal{D}(T_\sigma)}$ is a norm preserving operator between Banach spaces $(\mathcal{D}(T_\sigma), \lVert \, \cdot \, \rVert_{T_\sigma})$ and the domain $(\mathcal{D}(A), \lVert \,\cdot\,\rVert_A)$. In particular, $V\mathcal{D}(T_\sigma) = \mathcal{D}(A)$.

Let $f\in \mathcal{D}(T_\sigma)$. By dominated convergence theorem, $f\1_{\{\sigma <k\}}\to f$ as $k \to \infty$.  Since 
\[
\lVert\sigma^2 f\1_{\{\sigma <k\}}\rVert_{L^2_\mu} \le k^2 \lVert f\1_{\{\sigma <k\}}\rVert_{L^2_\mu}<\infty,
\] 
$f\1_{\{\sigma <k\}}\in \mathcal{D}(T_\sigma^2)$ for all $k \in \N$. We obtain $Vf\1_{\{\sigma <k\}} \in \mathcal{D}(A^*A) \subseteq \mathcal{D}(A)$ and combining with Claim~(ii),
\[
AVf\1_{\{\sigma <k\}} = U T_\sigma V^* Vf\1_{\{\sigma <k\}} = U T_\sigma f\1_{\{\sigma <k\}}.
\]
By unitarity of $V$, $Vf\1_{\{\sigma <k\}}\to Vf$. Since $\sigma f \in L^2_\mu $,
\[
\lVert T_\sigma f\1_{\{\sigma <k\}} - T_\sigma f\rVert_{L^2_\mu}^2 = \int \vert\sigma f\1_{\{\sigma <k\}} - \sigma f|^2\,d\mu =\int_{\{\sigma \geq k\}} |\sigma f|^2\,d\mu\to 0,
\]
and hence $AVf\1_{\{\sigma <k\}} = U T_\sigma f\1_{\{\sigma <k\}} \to U T_\sigma f$ as $U$ is an isometry. By closedness of $A$, $Vf \in \mathcal{D}(A)$ and 
\begin{equation}\label{eq:closed-limit}
AVf = U T_\sigma f.
\end{equation}

By \eqref{eq:closed-limit}, and because $V$ is unitary and $U$ is an isometry,
\[
\lVert Vf\rVert_A^2
=\lVert Vf\rVert_{\mathcal H}^2+\lVert AVf\rVert_{\mathcal K}^2
=\lVert f\rVert_{L^2_\mu}^2+\lVert T_\sigma f\rVert_{L^2_\mu}^2
=\lVert f\rVert_{T_\sigma}^2.
\]
It follows that $V\vert_{\mathcal{D}(T_\sigma)}$ is a norm-preserving injection.

Let $x \in \mathcal{D}(A)$. By Lemma~\ref{lem:dense}, let $(x_k)_{k=0}^\infty$ be a sequence in $\mathcal{D}(A^*A)$ such that $x_k \to x$ in the graph norm. By \eqref{eq:domain-A*A-equiv}, there exists $g_k \in \mathcal{D}(T_\theta) = \mathcal{D}(T_\sigma^2)$ such that $Vg_k = x_k$ for all $k\in \mathbb{N}$. Since $V$ is norm-preserving, for sufficiently large $k$ and $j$,
\[\lVert g_k - g_j \rVert_{T_\sigma} = \lVert Vg_k - Vg_j \rVert_A = \lVert x_k - x_j \rVert_A \to 0. \]
Therefore, there exists $g \in \mathcal{D}(T_\sigma)$ such that $g_k \to g$ in $\lVert\,\cdot\,\rVert_{T_\sigma}$. By continuity of $V\vert_{\mathcal{D}(T_\sigma)}$, $Vg = x$, which completes the claim.

Finally, for any $x \in \mathcal{D}(A)$, $V^*x \in \mathcal{D}(T_\sigma)$. Substituting $f = V^*x$ into \eqref{eq:closed-limit} yields
\[
Ax = A(V V^* x) = U T_\sigma V^* x. \qedhere
\]
\end{proof}

We would like to emphasize that this result cannot be obtained by simply considering $A$ as a bounded operator on $(\mathcal{D}(A), \lVert\,\cdot\,\rVert_A)$ and apply singular value decomposition for bounded operator \cite{CraneGockenbach2020}. In that case, the operator $V: L^2_\mu  \to \mathcal{D}(A)$ obtained will be unitary only with respect to $\lVert\,\cdot\,\rVert_A$ instead of the Hilbert space norm $\lVert\,\cdot\,\rVert_\mathcal{H}$. 

Now we proceed to prove the complete version of condensed singular value decomposition.

\begin{proof}[Proof of Theorem~\ref{thm:condensed-mult}]
Let $A_0\coloneqq A|_{\mathcal{D}(A)\cap \mathcal{N}(A)^{\perp}}$. Then
\[
A_0 : \mathcal{D}(A)\cap \mathcal{N}(A)^{\perp} \subseteq \mathcal{N}(A)^\perp \to \mathcal{K}
\]
is a densely defined, closed operator with trivial kernel. Applying Theorem~\ref{thm:svd1} to $A_0$ yields an isometry $U \in \mathcal{L}(L^2_\mu, \mathcal{K})$, a unitary operator $V_0 \in \mathcal{L}(L^2_\mu, \mathcal{N}(A)^\perp)$, and a measurable function $\sigma \colon M \to [0, \infty)$ that is positive almost everywhere such that 
\[
A_0x = U T_\sigma V_0^*x \quad \text{for all } x \in \mathcal{D}(A)\cap \mathcal{N}(A)^{\perp}.
\]

Let $V:L^2_\mu \to \mathcal{H}$ be the isometry obtained by letting $Vf=V_0f$ for all $f\in L^2_\mu$. It is straightforward to verify that 
\begin{equation}\label{eq:V}
  V^* x=
\begin{cases}
    V_0^{*}x  & x\in \mathcal{N}(A)^{\perp},\\
    0          & x\in \mathcal{N}(A).
\end{cases}
\end{equation}
For any $x\in\mathcal{H}$, we can uniquely write $x=x_0+x_1$ with $x_0\in\mathcal{N}(A)^\perp$ and $x_1\in \mathcal{N}(A)$. Since $\mathcal{N}(A) \subseteq \mathcal{D}(A)$, $x \in \mathcal{D}(A)$ if and only if $x_0 \in \mathcal{D}(A_0)$. We obtain that $x \in \mathcal{D}(A)$ if and only if $V^*x \in \mathcal{D}(T_\sigma)$ as $V^*x = V_0^*x_0$ and $x_0 \in \mathcal{D}(A_0)$ if and only if $V_0^*x_0 \in \mathcal{D}(T_\sigma)$.

At the same time, 
\[
U T_\sigma V^*x_0 = UT_\sigma V_0^{*}x_0 =A_0x_0 =Ax_0.
\]
Since $V^* x_1=0$, we have $U T_\sigma V^* x_1=0=Ax_1$, which completes the proof. 
\end{proof}

Next, we construct the full singular value decomposition by augmenting the operators and spaces in the condensed singular value decomposition.

\begin{proof}[Proof of Theorem~\ref{thm:full-mult}]
Apply Theorem~\ref{thm:condensed-mult} to $A$ to obtain the $\sigma$-finite Borel measure space $(M, \mu)$, isometries $V_0  \in \mathcal{L}( L^2_{\mu}, \mathcal{H})$ and $U_0 \in \mathcal{L} (L^2_{\mu}, \mathcal{K})$, and the measurable function $\sigma \colon M \to [0, \infty)$ that is positive almost everywhere. As every Hilbert space is isometrically isomorphic to an $L^2$ space, there exist measure spaces $(N_1,\nu_1)$, $(N_2,\nu_2)$ and unitary maps $P_1 \colon \range(V_0)^\perp \to L^2_{\nu_1}$ and $P_2 \colon \range(U_0)^\perp \to L^2_{\nu_2}$ \cite[Theorem~5.30]{Folland1984}. Let $M_1 \coloneqq M \sqcup N_1$ and $\mu_1(X \sqcup Y) \coloneqq \mu(X) + \nu_1(Y)$ for any measurable subsets $X \subseteq M$ and $Y \subseteq N_1$. It is straightforward to verify that $(M_1,\mu_1)$ is a well-defined measure space. Consider the map
\[Q_1 \colon L^2_{\mu_1} \to L^2_{\mu}  \oplus L^2_{\nu_1}, \quad f \mapsto (f\rvert_M, f\rvert_{N_1}).\]
Since
\[
\lVert f \rVert^2_{L^2_{\mu_1}} = \int_{M} \lvert f\rvert^2 \,d\mu + \int_{N_1} \lvert f\rvert^2 \,d\nu_1 = \bigl\lVert f\rvert_{M} \bigr\rVert_{L^2_{\mu} }^2 + \bigl\lVert f\rvert_{N_1} \bigr\rVert_{L^2_{\nu_1}}^2 = \lVert Q_1f \rVert_{L^2_{\mu}  \oplus L^2_{\nu_1}}^2,
\]
$Q_1$ is an isometry. For any $(g,h) \in  L^2_{\mu}  \oplus L^2_{\nu_1}$, let $f$ be defined by $f=g$ on $M$ and $f=h$ on $N_1$. The same computation shows that $f \in L^2_{\mu_1}$ and $Q_1f = (g,h)$. Therefore, $Q_1$ is surjective and hence unitary. Let $M_2 \coloneqq M \sqcup N_2$ with measure $\mu_2$ defined analogously. By the same argument, the map $Q_2 \colon L^2_{\mu_2} \to L^2_\mu  \oplus L^2_{\nu_2}$ defined by $f \mapsto (f\rvert_M, f\rvert_{N_2})$ is unitary. 

Now, define the unitary operators $V \in \mathcal{L}(L^2_{\mu_1}, \mathcal{H})$ and $U \in \mathcal{L}(L^2_{\mu_2}, \mathcal{K})$ via the compositions
\[
V \coloneqq (V_0 \oplus P_1^*) Q_1, \quad U \coloneqq (U_0 \oplus P_2^*) Q_2
\]
and define the block operator $\Sigma_0 \colon \mathcal{D}(T_\sigma) \oplus L^2_{\nu_1} \to L^2_\mu \oplus L^2_{\nu_2}$ by $\Sigma_0 (g, h) = (T_\sigma g, 0)$. Recall that $\Sigma \colon \mathcal{D}(\Sigma) \subseteq L^2_{\mu_1} \to L^2_{\mu_2}$ is defined by
\[
  \mathcal{D}(\Sigma) = \{f \in L^2_{\mu_1} : \sigma (f\rvert_M) \in L^2_\mu\}, \quad
  \Sigma f =
  \begin{cases}
  \sigma (f\rvert_M) & \text{on } M, \\
  0 & \text{on } M_2 \setminus M.
  \end{cases}
\] 
It follows immediately that $\mathcal{D}(\Sigma) = Q_1^*(\mathcal{D}(T_\sigma) \oplus L^2_{\nu_1})$ and $\Sigma = Q_2^* \Sigma_0  Q_1$.

For any $x \in \mathcal{H}$, we can uniquely write $x = x_0 + x_1$ where $x_0 \in \range(V_0)$ and $x_1 \in \range(V_0)^\perp = \mathcal{N}(A)$. 
Observe that $V^*x = Q_1^* (V_0^* x_0, P_1 x_1)$. Consequently, $V^*x \in \mathcal{D}(\Sigma)$ if and only if $V_0^* x_0 \in \mathcal{D}(T_\sigma)$. By Theorem~\ref{thm:condensed-mult}, this holds if and only if $x_0 \in \mathcal{D}(A)$. Since $\mathcal{N}(A) \subseteq \mathcal{D}(A)$, this is equivalent to $x \in \mathcal{D}(A)$, establishing the domain condition.

Finally, for $x = x_0 + x_1 \in \mathcal{D}(A)$, $Ax_1 = 0$, so $Ax = Ax_0$. Thus,
\begin{align*}
  U \Sigma V^*x &= U (Q_2^* \Sigma_0  Q_1) V^*x = (U_0 \oplus P_2^*) Q_2 Q_2^* \Sigma_0  Q_1 Q_1^* (V_0^* x_0, P_1 x_1) \\
  &= (U_0 \oplus P_2^*) \Sigma_0  (V_0^* x_0, P_1 x_1) = (U_0 \oplus P_2^*) (T_\sigma V_0^* x_0, 0) = U_0 T_\sigma V_0^* x_0 = Ax_0 = Ax. \qedhere
\end{align*}
\end{proof}

Let $A = U T_\sigma  V^*$ be the condensed singular value decomposition of $A$. We say that $U$ and $V$ are respectively the left and right singular operators of $A$. For a scalar $s \geq 0$, we say that $s$ is a \emph{singular value} of $A$ if one of the following holds:
\begin{enumerate}[\normalfont(i)]
  \item $s$ is in the essential range of $\sigma$,
  \item $s = 0$ and $\mathcal{N}(A) \neq \{0\}$.
\end{enumerate}
Let $\mathsf{\Sigma}(A)$ denote the set of singular values of $A$. Because the essential range of a measurable function is closed, $\mathsf{\Sigma}(A)$ is a closed subset of $[0, \infty)$, and is therefore a Borel measurable set. For a singular value $s > 0$, we call $m_s \coloneqq \dim \mathcal{N}(T_\sigma - s I)$ the \emph{multiplicity} of $s$. If $m_s > 0$, the non-zero elements in the spaces $U(\mathcal{N}(T_\sigma - s I))$ and $V(\mathcal{N}(T_\sigma - s I))$ are called the left and right singular vectors corresponding to $s$, respectively. Moreover concretely,
\[
\mathsf{\Sigma}(A)
=
\bigl\{\sqrt{\lambda}:\lambda\in\mathsf{\Lambda}(A^*A)\bigr\}.
\]

\subsection{Singular value decomposition in direct integral form}\label{sec:svd-direct}

We first recall the definition of a direct integral \cite[Section~7.3]{Hall2013}. Let $(M,\mu)$ be a $\sigma$-finite measure space and suppose for each $x \in M$, we associate a separable Hilbert space $\mathcal{H}_x$ with inner product $\langle \,\cdot , \cdot \,\rangle_x$  such that $x \mapsto \dim \mathcal{H}_x$ is a measurable function from $M$ to $[0,\infty]$. A section $\tau$ is a function from $M$ to the disjoint union of the $\mathcal{H}_x$'s, such that $\tau(x) \in \mathcal{H}_x$ for all $x \in M$. A measurability structure on $\{\mathcal{H}_x\}_{x \in M}$ is a collection of sections $\{e_k(\cdot)\}_{k=1}^\infty$ such that 
\begin{enumerate}[\normalfont(i)]
  \item for all $x\in M$, $\langle e_k(x),e_j(x)\rangle_x = 0$ for all $k \neq j$,
  \item for all $x\in M$ and $k\in \mathbb{N}$, the norm of each $e_k(x)$ is either $0$ or $1$,
  \item for all $x\in M$, $\overline{\operatorname{span}(\{e_k(x)\}_{k=1}^\infty)} = \mathcal{H}_x$,
  \item and the function $x \mapsto \langle e_k(x), e_j(x) \rangle_x$ is measurable for all $j,k \in \mathbb{N}$.
\end{enumerate}
 A section $\tau$ is measurable if $x \mapsto \langle e_k(x), \tau(x) \rangle_x$ is measurable for all $k \in \mathbb{N}$. Given the measurability structure, the direct integral of $\{\mathcal{H}_x\}_{x \in M}$, denoted 
\[
\int_M^\oplus \mathcal{H}_x \, d\mu(x),
\]
is the Hilbert space of equivalence classes of almost-everywhere-equal measurable sections with finite norm, equipped with the inner product 
\begin{equation}\label{eq:innerp-direct}
\langle \tau, \rho \rangle = \int_M \langle \tau(x), \rho(x) \rangle_x \, d\mu(x).
\end{equation}
 
\begin{theorem}[Singular value decomposition in direct integral form]\label{thm:di}
Let $A: \mathcal{D}(A) \subseteq \mathcal{H} \to \mathcal{K}$ be a closed, densely defined linear operator. Then there exists a $\sigma$-finite measure $\eta$ on $\mathsf{\Sigma}(A)$, a direct integral 
\[
\int_{\mathsf{\Sigma}(A)}^\oplus \mathcal{H}_s \, d\eta(s),
\]
partial isometries 
\[U:\int_{\mathsf{\Sigma}(A)}^\oplus \mathcal{H}_s \, d\eta(s) \to \mathcal{K}, \quad V: \int_{\mathsf{\Sigma}(A)}^\oplus \mathcal{H}_s \, d\eta(s) \to \mathcal{H},\] 
and 
\begin{equation}\label{eq:dom-direct}
  \Sigma: \Bigl\{\tau \in \int_{\mathsf{\Sigma}(A)}^\oplus \mathcal{H}_s \, d\eta(s): \int_{\mathsf{\Sigma}(A)} \lVert s \tau(s) \rVert_s^2 \,d\eta(s) < \infty \Bigr\} \to \int_{\mathsf{\Sigma}(A)}^\oplus \mathcal{H}_s \, d\eta(s), \quad \tau (s) \mapsto s\tau(s)
\end{equation}
such that $\mathcal{D}(A) = \{x \in \mathcal{H} : V^* x \in \mathcal{D}(\Sigma)\}$ and $Ax = U\Sigma V^*x$ for all $x \in \mathcal{D}(A)$.
\end{theorem}
\begin{proof}
If $A=0$, the conclusion is immediate by taking $\eta=0$ and $\mathcal{H}_s=\{0\}$ for every
$s\in\mathsf{\Sigma}(A)$. Suppose $A \neq 0$. Let $(M,\mu)$, $U_0$, $V_0$, and $\sigma$ be given by Theorem~\ref{thm:condensed-mult}. By \cite[Sections~10.2 and 10.3]{Hall2013}, we may assume that $M$ is standard Borel and $\sigma:M\to \mathsf{\Sigma}(A)$ is Borel measurable. 

Since $\mu$ is nonzero and $\sigma$-finite, choose a positive Borel function $w:M\to(0,\infty)$ such that $\int_M w\,d\mu=1$, and let $d\nu \coloneqq w\,d\mu$. The map
\[
J:L^2_\mu\to L^2_\nu,\qquad f \mapsto w^{-1/2}f,
\]
is a unitary that intertwines $T_\sigma$. Thus, $A= (U_0J^*)T_\sigma (V_0J^*)^*$ is also a singular value decomposition of $A$. Therefore, we may without loss of generality assume that $\mu$ is a probability measure.

Consider the probability measure $\eta$ on $\mathsf{\Sigma}(A)$ defined by $\eta(B)\coloneqq\mu\bigl(\sigma^{-1}(B)\bigr)$ for all Borel subsets $B \subseteq \mathsf{\Sigma}(A)$. Since $M$ is a standard Borel space and $\mu$ is a probability measure, $\mu$ is compact by
\cite[Proposition~451M]{Fremlin2006}. By \cite[Theorem~452I]{Fremlin2006}, there exists a disintegration $\{\mu_s\}_{s\in \mathsf{\Sigma}(A)}$ of $\mu$ over $\eta$ consistent with $\sigma$. In particular, $s\mapsto\mu_s(E)$ is $\eta$-measurable for every Borel set $E\subseteq M$ and
\[
\mu(E)=\int_{\mathsf{\Sigma}(A)}\mu_s(E)\,d\eta(s).
\]
Since $\mathsf{\Sigma}(A) \subseteq [0, \infty)$, the
disintegration is strongly consistent with $\sigma$ by \cite[Proposition~452G(c)]{Fremlin2006}.
Therefore, $\mu_s\bigl(\sigma^{-1}(\{s\})\bigr)=1$ for $\eta$-almost every $s\in\mathsf{\Sigma}(A)$. Consequently, for every nonnegative Borel function $q$,
\[
\int_M q\,d\mu
=\int_{\mathsf{\Sigma}(A)}\biggl(\int_{\sigma^{-1}(s)}q(x)\,d\mu_s(x)\biggr)d\eta(s).
\]

Let $\mathcal{H}_s\coloneqq L^2(\sigma^{-1}(s),\mu_s\vert_{\sigma^{-1}(s)})$. Let$\{B_j\}_{j\ge1}$ be a countable algebra generating the Borel
$\sigma$-algebra of $M$, and let 
$g_k(s)\coloneqq\1_{B_k}\vert_{\sigma^{-1}(s)}$. Since $\mu_s$ is finite, the span of the $g_k(s)$ is
dense in $\mathcal{H}_s$ and $\langle g_j(s),g_k(s)\rangle_s=\mu_s(B_j\cap B_k)$ is measurable in $s$ for all $j$, $k$.
Performing the Gram--Schmidt fiberwise gives a measurability structure $\{e_k(\cdot)\}_{k\ge1}$ and we obtain the direct integral
\[
\int_{\mathsf{\Sigma}(A)}^\oplus\mathcal{H}_s\,d\eta(s).
\]

For $f\in L^2_\mu$, choose a Borel representative and let $(Wf)(s)\coloneqq f\vert_{\sigma^{-1}(\{s\})}$ for $\eta$-almost every $s$. Then $W$ is an isometry from $L^2_\mu$ to the direct integral since for any $f,g\in L^2_\mu$,
\begin{align*}
    \langle Wf, Wg \rangle_{\int \mathcal{H}_s \, d\eta(s)} &= \int_{\mathsf{\Sigma}(A)} \langle Wf(s),Wg(s)\rangle_s d\eta(s)\\
    &= \int_{\mathsf{\Sigma}(A)}\Bigl( \int_{\sigma^{-1}(s)}f(x)\overline{g(x)}\,d\mu_s(x)\Bigr)d\eta(s)=\int_{M}f(x)\overline{g(x)}\,d\mu(x) = \langle f, g\rangle_{L^2_\mu}.
\end{align*}

It remains to prove that $W$ is surjective. It suffices to show $\range(W)$ is dense. Let $\tau\in \range(W)^\perp$. For every bounded Borel
function $b:\mathsf{\Sigma}(A)\to\C$ and every $k\ge1$, $(b\circ\sigma)\1_{B_k}\in L^2_\mu$. Moreover, for $\eta$-almost every $s$, $(W(b\circ\sigma)\1_{B_k})(s)=b(s)g_k(s)$.
Therefore,
\[
0
=
\langle W(b\circ\sigma)\1_{B_k},\tau\rangle
=
\int_{\mathsf{\Sigma}(A)}
b(s)\langle g_k(s),\tau(s)\rangle_s\,d\eta(s).
\]
Therefore, $\langle g_k(s),\tau(s)\rangle_s=0$ for $\eta$-almost every $s$ and hence $\tau=0$. Thus, $\range(W)$ is dense and hence $W$ is unitary.

Consider isometries $U = U_0W^*$ and $V = V_0W^*$. Since $(WT_\sigma f)(s) = s \cdot f\rvert_{\sigma^{-1}(s)} = s (Wf)(s)$, $W\mathcal{D}(T_\sigma) = \mathcal{D}(\Sigma)$ and $W T_\sigma W^* = \Sigma$ where $\Sigma \tau(s) = s \tau(s)$. Thus, by Theorem~\ref{thm:condensed-mult}, $x \in \mathcal{D}(A)$ if and only if $V^* x \in \mathcal{D}(\Sigma)$. For each $x \in \mathcal{D}(A)$, 
\[
Ax = U_0T_\sigma V_0^* x = U_0W^* WT_\sigma W^* WV_0^* x = U \Sigma V^* x. \qedhere 
\]
\end{proof}

\begin{corollary}
Let $0 < s \in \mathsf{\Sigma}(A)$ such that $\eta(\{s\}) > 0$, and define the subspace
\[
\mathcal{E}_s \coloneqq \biggl\{ \tau \in \int_{\mathsf{\Sigma}(A)}^\oplus \mathcal{H}_t \, d\eta(t) : \tau = \mathbbm{1}_{\{s\}} \tau \biggr\}.
\] 
Then the nonzero elements of $U\mathcal{E}_s$ and $V\mathcal{E}_s$ are the spaces of left and right singular vectors of $A$ corresponding to the singular value $s$.
\end{corollary}
\begin{proof}
For any section $\tau$, $\tau \in \mathcal{N}(\Sigma - sI)$ if and only if $t\tau(t) = s\tau(t)$ for almost all $t \in \mathsf{\Sigma}(A)$, i.e, $(t-s)\tau(t) = 0$ almost everywhere. Therefore, $\mathcal{N}(\Sigma - sI) = \mathcal{E}_s$, and the result follows immediately from the definition of singular vectors.
\end{proof}

\subsection{Singular value decomposition in operator-valued measure form}\label{sec:svd-ovm}

Again, we first recall the definitions and tools of spectral measure and spectral integrals. Let $\mathcal{H}$ be a Hilbert space. A \emph{spectral measure}, also known as a \emph{projection-valued measure}, $\mu$ on $\R$ is a mapping from the Borel $\sigma$-algebra on $\R$ to the set of orthogonal projections on $\mathcal{H}$ such that $\mu(\R)=I$ and for any disjoint Borel subsets $\{E_k\}_{k = 1}^\infty $,
  \begin{equation}\label{eq:count-add}
    \mu\biggl(\bigcup_{k=1}^\infty E_k \biggr)x =\sum_{k=1}^\infty\mu(E_k)x, \quad \text{for all } x \in \mathcal{H}
\end{equation}
For any spectral measure $\mu$ and $x,y\in\mathcal{H}$, define $\mu_{x,y}(E) \coloneqq \langle \mu(E)x,y\rangle$ for Borel subset $E$ and write $\mu_x \coloneqq \mu_{x,x}$. The measure $\mu_x$ is a finite positive Borel measure on $\R$ with $\mu_x(\R)=\lVert x\rVert^2$.

Let $f:\R\to \mathbb{C}$ be measurable and define
\[
\mathcal{D}_f = \Big\{ x\in\mathcal{H}: \int_\R |f(\lambda)|^2 d\mu_x(\lambda) <\infty\Big\},
\] 
which is dense in $\mathcal{H}$. There is a unique unbounded operator on $\mathcal{H}$, denoted $\int_\R f \,d\mu$, with domain $\mathcal{D}_f$ such that for all $x\in\mathcal{D}_f$ and $y\in\mathcal{H}$,
\[
\biggl\langle  \Big(\int_\R f \,d\mu \Big)x, y \biggr\rangle = \int_\R f(\lambda) \, d\mu_{x,y}(\lambda).
\]
In particular, $\int_\R f \,d\mu$ is normal (hence closed), and if $f$ is real-valued, then $\int_\R f \,d\mu$ is self-adjoint \cite[Section~10]{Hall2013}. We call $\int_\R f \,d\mu$ the \emph{spectral integral} of $f$.

Let $S$ be a self-adjoint operator on $\mathcal{H}$. By spectral theorem for self-adjoint operators in spectral measure form \cite[Theorem~5.7]{Konrad2012}, there exists a unique spectral measure $\mu$ on $\R$ associated with $S$ such that $S$ is the spectral integral of the identity function. For a Borel function $f:\R\to\C$, define $f(S)$ by the spectral integral of $f(\lambda)$ with respect to $\mu$, i.e., 
\[
f(S) \coloneqq \int_\R f(\lambda)\,d\mu(\lambda), \quad \mathcal{D}(f(S))=\Big\{ x\in\mathcal{H}: \int_{\R} |f(\lambda)|^2\,d\mu_x(\lambda) < \infty\Big\}.
\]
The assignment $f\mapsto f(S)$ is called the \emph{functional calculus} of the self-adjoint operator $S$. Throughout this section, let $P$ denote the orthogonal projection onto $\mathcal{N} (S)$, so $(I-P)$ is the orthogonal projection onto $\mathcal{N} (S)^\perp$.

The following lemma lists the fundamental properties for functional calculus \cite[Theorem~3.1 and Theorem~3.2]{Teschl2014} and \cite[Proposition~4.23 and Theorem~5.9]{Konrad2012}. 
\begin{lemma}\label{lem:fc}
Let $S$ be a self-adjoint operator on $\mathcal{H}$ with associated spectral measure $\mu$. For Borel functions $f,g:\R\to\C$, the following holds.
\begin{enumerate}[(i)]
\item For $x\in \mathcal{D}(f(S))$, $ \lVert f(S)x\rVert^2 = \int_{\R} |f(\lambda)|^2 \, d  \mu_x(\lambda)$.
\item\label{it:fc-iii} $f(S)g(S) \subseteq (fg)(S) =\int_{\R}f(\lambda)g(\lambda)\, d \mu(\lambda)$ with domain $\mathcal{D}(f(S)g(S)) = \mathcal{D}(g(S)) \cap \mathcal{D}((fg)(S))$.
\item\label{it:fc-iv} If $E$ is a Borel subset, then $\mu(E)f(S)\subseteq f(S)\mu(E)$. Furthermore, if $f$ is bounded, then equality holds.
\item\label{it:fc-cm}  For $x \in \mathcal{D}(f(S))$, $d\mu_{f(S)x} (\lambda) = |f(\lambda)|^2d\mu_x(\lambda)$.
\item\label{it:fc-sa} $\overline{f}(S)=f(S)^*$. In particular, $f(S)$ is self-adjoint if $f$ is real-valued.
\item\label{it:fc-approx} If $\{ f_k \}_{k=1}^\infty$ is a sequence of measurable simple functions such that $f_k\to f$ pointwise for $k \to \infty$ and $|f_k(\lambda)|\le |f(\lambda)|$ for all $k \in \N$ and $\lambda \in \R$, then $f_k(S)x\to f(S)x$ for all $x\in \mathcal{D}(f(S))$. 
\end{enumerate}
\end{lemma}

\begin{lemma}\label{lem:partial-iso}
Let $S = A^*A$ and $\mu$ its associated spectral measure. Take $f(\lambda)=\lambda^{1/2}\mathbbm{1}_{[0,\infty)}$, $g(\lambda)=\lambda^{-1/2}\mathbbm{1}_{(0,\infty)}$, and let $S^{1/2}=f(S)$, $S^{-1/2}=g(S)$ via functional calculus.  Then
\begin{enumerate}[(i)]
\item\label{it:pi-domain1} $\mathcal{D}(A)=\mathcal{D}(S^{1/2})$ and $\lVert S^{1/2}x\rVert =\lVert Ax\rVert$ for all $x\in \mathcal{D}(A)$;
\item\label{it:pi-domain2} $\mathcal{D}(AS^{-1/2})=\mathcal{D}(S^{-1/2})$ and $\lVert AS^{-1/2}x\rVert =\lVert (I-P) x\rVert$ for all $x\in \mathcal{D}(AS^{-1/2})$;
\item\label{it:pi-ker} $\mathcal{N} (S) = \mathcal{N} (S^{1/2})$.
\end{enumerate}
\end{lemma}

\begin{proof}
Item \ref{it:pi-domain1} follows from \cite[Theorem~2.23]{Kato1976}. Since $(fg)(\lambda) = \mathbbm{1}_{(0,\infty)}(\lambda)$, $(fg)(S)=I-\mu(\{0\})=I-P$ and $\mathcal D(S^{1/2}S^{-1/2})=\mathcal D(S^{-1/2})$ by Lemma~\ref{lem:fc}\ref{it:fc-iii}. Hence, $S^{1/2}S^{-1/2}=I-P$ on $\mathcal D(S^{-1/2})$. By~\ref{it:pi-domain1}, 
\[
\mathcal D(AS^{-1/2})
=\{x\in \mathcal D(S^{-1/2}): S^{-1/2}x\in \mathcal D(A)\}=\{x\in \mathcal D(S^{-1/2}): S^{-1/2}x\in \mathcal D(S^{1/2})\}.
\]
If $x \in \mathcal{D}(S^{-1/2})$, then $S^{-1/2}x \in \mathcal{D}(S^{1/2})$ by Lemma~\ref{lem:fc}~\ref{it:fc-cm} since 
\[ 
\int \lambda \,d\mu_{S^{-1/2}x} (\lambda)= \int \lambda (\lambda^{-1})\,d\mu_x(\lambda) = \int_{(0,\infty)} 1\,d\mu_x(\lambda) \le \lVert x\rVert^2 < \infty.
\]
Thus, $S^{-1/2}x \in \mathcal{D}(S^{1/2})$ for all $x \in \mathcal{D}(S^{-1/2})$. Hence, $\mathcal{D}(AS^{-1/2}) = \mathcal{D}(S^{-1/2})$. 

For $x\in \mathcal D(AS^{-1/2})$, by \ref{it:pi-domain1},
\[
\lVert AS^{-1/2}x\rVert =\lVert S^{1/2}S^{-1/2}x\rVert =\lVert (I-P) x\rVert. 
\]
By \ref{it:fc-sa} and \ref{it:fc-iii} of Lemma~\ref{lem:fc}, $S^{1/2}$ is self-adjoint, and for $x\in \mathcal{N}(S)$ we have  $\lVert S^{1/2}x\rVert^2 = \langle S^{1/2}x, S^{1/2}x\rangle = \langle S x,x\rangle=0$, therefore $S^{1/2}x=0$, i.e. $x\in\mathcal{N}(S^{1/2})$. If $x\in \mathcal{N} (S^{1/2})$, by Lemma~\ref{lem:fc}\ref{it:fc-iii} again, $Sx=S^{1/2}S^{1/2}x=0$, therefore $\mathcal{N} (S)=\mathcal{N} (S^{1/2})$.
\end{proof}

\begin{proposition}\label{prop:ovm}
Let $\mu$ be a spectral measure on $\R$ with Hilbert space $\mathcal{H}$, and let $Q: \mathcal{H} \to \mathcal{K}$ be a fixed partial isometry. Then
\begin{equation}\label{eq:ovm}
    \nu(E) \coloneqq  Q\mu(E), \quad \text{for Borel subsets } E \subseteq \R 
\end{equation}
defines an operator-valued measure, i.e.
\begin{enumerate}[(i)]
    \item\label{it:ovm-i} $\nu(\varnothing)=0$ and $\nu(\R)=Q$;
    \item\label{it:ovm-ii} For any disjoint Borel subsets $\{E_k\}_{k=1}^\infty$,
    \[
    \nu \biggl(\bigcup_{k=1}^\infty E_k \biggr)x =\sum_{k=1}^\infty\nu(E_k)x, \quad \text{for all } x \in \mathcal{H}.
    \]
\end{enumerate}
\end{proposition}

\begin{proof}
   For~\ref{it:ovm-i}, we have $\nu(\varnothing)=Q\mu(\varnothing)=Q0=0$ and $\nu(\R)=Q\mu(\R)=QI=Q$. For~\ref{it:ovm-ii}, since $Q$ is continuous, we have
   \[
   \nu \biggl(\bigcup_{k=1}^\infty E_k \biggr)x = Q\mu \biggl(\bigcup_{k=1}^\infty E_k \biggr)x=Q\sum_{k=1}^\infty\mu(E_k)x=\sum_{k=1}^\infty Q\mu(E_k)x=\sum_{k=1}^\infty\nu(E_k)x.\qedhere
   \]
\end{proof}

Let $S = A^*A$ and $\mu$ its associated spectral measure. By functional calculus, $S^{-1/2}$ is a densely defined operator. Since $\mathcal{N} (S)^\perp$ is a closed subspace, $\mathcal D(S^{-1/2})\cap\mathcal N(S)^\perp=(I-P)\mathcal D(S^{-1/2})$ is dense in $\mathcal N(S)^\perp$. By Lemma~\ref{lem:partial-iso}~\ref{it:pi-domain2}, $AS^{-1/2}$ is an isometry on $\mathcal{D}(S^{-1/2})\cap \mathcal{N} (S)^\perp$. Therefore, it extends to a unique isometry $\widetilde{Q}$ on $\mathcal{N} (S)^\perp$. Consequently, we may extend $AS^{-1/2}$ to a partial isometry 
\begin{equation}\label{eq:q}
  Q:\mathcal{H}\to \mathcal{K}, \quad x \mapsto \widetilde{Q} (I-P) x.
\end{equation}
We define its associated operator-valued measure by \eqref{eq:ovm} as in Proposition~\ref{prop:ovm}.

Let $E_1,\dots,E_n \subseteq \mathsf{\Lambda}(S)$ be pairwise disjoint and $g\coloneqq \sum_{k=1}^n c_k \1_{E_k}:\mathsf{\Lambda}(S)\to \R$ with $c_k \in \R$. Define the integral for simple function $g$ with respect to $\nu$ as
\[
\int_{\R} g(\lambda)\,d\nu(\lambda) \coloneqq \sum_{k=1}^n c_k \nu(E_k) = \sum_{k=1}^n c_k Q\mu(E_k) = Q\sum_{k=1}^n c_k \mu(E_k) = Qg(S).
\]
Therefore, $\int_{\R} g(\lambda)\,d\nu(\lambda)$ is independent of the particular representation of $g$ and hence is well-defined. Moreover, since $g(S)$ is a bounded operator on $\mathcal{H}$ \cite[Lemma~4.11]{Konrad2012}, so is $\int_{\R} g(\lambda)\,d\nu(\lambda)$. 

\begin{proposition}\label{prop:Dnuf}
Let $f:\R\to \R$ be a nonnegative measurable function and $\{f_k\}$ a sequence of simple functions approximating $f$ from below. Define
\[
\mathcal{D}(\nu f)\coloneqq \Big\{x\in\mathcal{H}: \int_{\R} f(\lambda)^2 \, d  \mu_{(I-P)x} (\lambda)<\infty\Big\}.
\]
Then $x\in\mathcal{D}(\nu f)$ if and only if $\{Qf_k(S)x\}$ converges in $\mathcal{K}$. Moreover, for another sequence $\{g_k\}$ of simple functions approximate $f$ from below,
\[
  \lim_{k\to \infty} Qf_k(S)x=\lim_{k\to \infty} Qg_k(S)x  \quad \text{ for all } x \in \mathcal{D}(\nu f).
\]
\end{proposition}

\begin{proof}
Let $x \in \mathcal{D}(\nu f)$. Then 
\[
Qf_k(S)Px=Qf_k(0)Px = f_k(0)QPx = 0,
\]
and therefore by Lemma~\ref{lem:fc}\ref{it:fc-iii}, for $j$, $k \in \N$,
\begin{align*}
 \lVert Q f_k(S)x-Q f_j(S)x\rVert^2 &= \lVert Q f_k(S)(I-P)x-Q f_j(S)(I-P)x\rVert^2 \\ &= \lVert Q (f_k(S)-f_j(S))(I-P)x\rVert^2 \le \lVert Q\rVert^2 \int_{\R} \lvert f_k(\lambda)-f_j(\lambda)\rvert^2 \, d\mu_{(I-P)x}(\lambda).
\end{align*}
Since $\int_{\R} f(\lambda)^2\,d\mu_{(I-P)x}(\lambda) <\infty$, we have $\{Qf_k(S)x\}$ is a Cauchy sequence in $\mathcal{K}$ by Lebesgue's dominated convergence theorem, and hence converges.

Since $Q$ is partial isometry with $\supp Q =\mathcal{N} (S)^\perp$, we have $Q^*Q=(I-P)=\mu(\mathsf{\Lambda}(S)\backslash \{0\})$, hence 
\[
\lVert Qf_k(S)x\rVert^2 =\langle (I-P) f_k(S)x, f_k(S)x \rangle = \lVert (I-P) f_k(S)x\rVert^2 = \lVert f_k(S)(I-P)x\rVert^2.
\]
Suppose now $\{Q f_k(S)x\}$ converges. Then $\{\lVert f_k(S)(I-P)x\rVert^2\}$ converges. Since $\{f_k(\lambda) ^2\}$ converges monotonically to $f(\lambda)^2 $, by Lebesgue's monotone convergence theorem,
\[
\int_{\R} f(\lambda) ^2\,d\mu_{(I-P) x} (\lambda)  = \lim_{k\to\infty} \int_{\R} f_k(\lambda) ^2 \,d \mu_{(I-P) x}(\lambda)  = \lim_{k\to\infty}  \lVert f_k(S)(I-P)x\rVert^2  <\infty.
\]
Therefore $x\in \mathcal{D}(\nu f)$.     

Next, suppose $\{g_k\}$ is another sequence of simple function approximate to $f$ from below, then
\begin{align*}
\lVert Qf_k(S)x-Qg_k(S)x\rVert^2 &= \lVert Q(f_k(S)-g_k(S))(I-P)x\rVert^2 \le \lVert Q\rVert^2 \int_{\R} \lvert f_k(\lambda)-g_k(\lambda)\rvert^2\,d\mu_{(I-P)x}(\lambda)	\\&\le 2\int_{\R} \lvert f_k(\lambda)-f(\lambda) \rvert^2\,d\mu_{(I-P)x}(\lambda)+2\int_{\R} \lvert f(\lambda)-g_k(\lambda)\rvert^2\,d\mu_{(I-P)x}(\lambda)\to 0. \qedhere
\end{align*}
\end{proof}

As a consequence of Proposition~\ref{prop:Dnuf},
\begin{equation}\label{eq:nu-f}
  \int_{\R} f(\lambda)\,d\nu(\lambda)x \coloneqq \lim_{k\to \infty} Qf_k(S)x=\lim_{k\to \infty} Qg_k(S)x   
\end{equation}
defines a linear operator on $\mathcal{D}(\int f(\lambda)\,d\nu)=\mathcal{D}(\nu f)$. Furthermore, if $x\in \mathcal{D}(f(S)) \subseteq  \mathcal{D}(\nu f)$, we have $\lVert f_k(S)x-f(S)x\rVert^2 \to 0$ by Lemma~\ref{lem:fc}\ref{it:fc-approx}, and hence 
\[
\int_{\R} f(\lambda)\,d\nu(\lambda)x=Qf(S)x.
\]

\begin{corollary}[Singular value decomposition in operator-valued measure form]\label{thm:svd-pvm}
Let $A : \mathcal{D}(A) \subseteq  \mathcal{H} \to \mathcal{K}$ be a closed, densely defined linear operator and $\nu$ defined by \eqref{eq:ovm}. Then
\[
A = \int_{[0,\infty)} \lambda^{1/2}  \, d\nu(\lambda).
\]
\end{corollary}

\begin{proof}
Let $S^{1/2}=\lambda^{1/2}\mathbbm{1}_{[0,\infty)}(S)$ and $S^{-1/2}=\lambda^{-1/2}\mathbbm{1}_{(0,\infty)}(S)$ as in Lemma~\ref{lem:partial-iso}. Then

\begin{align*}
\mathcal{D}\Big(\int_{[0,\infty)} \lambda^{1 / 2} \,d \nu\Big)
&=\Big\{x\in\mathcal{H}: \int_{[0,\infty)} \lambda \, d  \mu_{(I-P)x} (\lambda)<\infty\Big\}=\Big\{x\in\mathcal{H}: \int_{(0,\infty)} \lambda \, d  \mu_{x} (\lambda)<\infty\Big\}\\
&=\Big\{x\in\mathcal{H}: \int_{[0,\infty)} \lambda \, d  \mu_{x} (\lambda) < \infty\Big\}
= \mathcal{D}(S^{1/2})=\mathcal{D}(A).
\end{align*}

For $x\in\mathcal{D}(S^{1/2})=\mathcal{D}(A)\subseteq \mathcal{D}(\int_{[0,\infty)} \lambda^{1/2}\,d\nu(\lambda))$, we have $S^{1/2}x \in \mathcal{D}(S^{-1/2})$ and
\[
\int_{(0,\infty)} \lambda^{-1} \,d\mu_{S^{1/2}x} (\lambda)= \int_{(0,\infty)} \lambda (\lambda^{-1})\,d\mu_x(\lambda) = \int_{(0,\infty)} 1\,d\mu_x(\lambda) \le \lVert x\rVert^2 < \infty
\]
by Lemma~\ref{lem:fc}\ref{it:fc-cm}. By Lemma~\ref{lem:partial-iso}\ref{it:pi-ker}, we have $S^{1/2}w=0$ for all $w\in \mathcal{N} (S)=\mathcal{N}(S^{1/2})$. Therefore, Lemma~\ref{lem:fc}~\ref{it:fc-sa} gives that
\[
\langle S^{1/2}x,w\rangle = \langle x,S^{1/2}w\rangle = \langle x,0\rangle = 0 \Longrightarrow S^{1/2}x\in \mathcal{N} (S)^\perp.
\]
Thus, $S^{1/2} x\in \mathcal{D}(S^{-1/2})\cap \mathcal{N} (S)^\perp$ for all $x\in \mathcal{D}(A)$. Let $Q$ be defined as in \eqref{eq:q}. By Lemma~\ref{lem:fc}
\ref{it:fc-iii} and Proposition~\ref{prop:Dnuf},
\[\int_{[0,\infty)} \lambda^{1/2} \,d\nu(\lambda) x =QS^{1/2}x= A S^{-1/2} S^{1/2} x = A(I-P) x = Ax.
\] 
The last equality holds because $\mathcal{N}(S^{1/2})=\mathcal{N}(A)$ by Lemma~\ref{lem:partial-iso}\ref{it:pi-domain1} and~\ref{it:pi-ker}.
\end{proof}

\begin{remark}
The polar decomposition can be recovered from the singular value decomposition in operator-valued measure form. Let $\{f_k\}_{k=1}^\infty$ be a sequence of simple functions approximating $f(\lambda)=\lambda^{1/2}$ from below, then we obtain
\[
Ax=\int_{[0,\infty)}\lambda^{1/2}\,d\nu(\lambda)x = \lim_{k\to \infty} \int_{\R}f_k(\lambda)\,d\nu(\lambda)x=\lim_{k\to \infty}  Qf_k(S)x=Q\lim_{k\to \infty}f_k(S)x = QS^{1/2}x,
\]

for all $x\in\mathcal{D}(A)$. The equation $A=QS^{1/2}$ is exactly the polar decomposition of $A$ \cite[Theorem~7.2]{Konrad2012}.
\end{remark}

\subsection{Discrete singular values}\label{sec:svd-discrete}

Suppose now that $\mathsf{\Sigma}(A)$ is discrete, which is true if and only if $\mathsf{\Lambda}(A^*A)$ is discrete. Observant readers have probably noticed that the decomposition simplifies massively in all three forms of singular value decomposition if $A$ has countably many singular values. 

\begin{corollary}\label{cor:discrete}
If $\mathsf{\Sigma}(A)$ is discrete, then 
\[
  Ax = \sum_{s \in \mathsf{\Sigma}(A) \setminus \{0\}} \sum_{k=1}^{m_s} s \langle x, v_{s,k} \rangle u_{s,k}
\]
for all $x \in \mathcal{D}(A)$, where $\{v_{s,k}\}_{k=1}^{m_s}$ and $\{u_{s,k}\}_{k=1}^{m_s}$ are orthonormal bases of the right and left singular spaces of $s$, respectively.
\end{corollary}
\begin{proof}
It follows from Theorem~\ref{thm:di} as
\[
\int_{\mathsf{\Sigma}(A)}^\oplus \mathcal{H}_s \, d\eta(s) \cong
\bigoplus_{\substack{s \in \mathsf{\Sigma}(A)\\ \eta(\{s\})>0}} \mathcal{H}_s. \qedhere
\]
\end{proof}
 It is the straightforward generalization of the case that for a matrix $A \in \C^{m \times n}$, its singular value decomposition can be written as the sum of rank-$1$ matrices 
\[A = \sum_{k=1}^{\rank(A)} \sigma_{k} u_k v_k^*.\]

This representation also gives us a convenient way to calculate the left singular vectors from the right singular vectors. Namely, given right singular vectors $\{v_{s,k}\}_{k=1}^{\dim \mathcal{H}_s}$, one may compute $\{u_{s,k}\}_{k=1}^{\dim \mathcal{H}_s}$ by
\begin{equation}\label{eq:basis}
    u_{s, k} = \frac{1}{s}Av_{s,k}, \quad s \in \mathsf{\Sigma}(A)\setminus\{0\},\, k = 1, \dots, m_s.
\end{equation}

\section{Properties of singular value decomposition}\label{sec:properties}
Given a matrix $M \in \C^{m \times n}$, the following entities could all be computed immediately from the (condensed) singular value decomposition of $M = U \Sigma V^*$:
\begin{description}
  \item [rank-$r$ approximation] by Eckart--Young theorem, the best rank-$r$ approximation of $M$ under the operator norm is the matrix $U \Sigma_r V^*$ where $\Sigma_r$ replaces the last $\rank(M)-r$ diagonal entries of $\Sigma$ by zeros;
  \item [singular value decomposition of adjoint] the singular value decomposition of $M^*$ is $M^* = V \Sigma U^*$;
\item [projection onto fundamental subspaces] the projection onto $\mathcal R(M)$ is $UU^*$ and the projection onto the kernel of $M$ is $I-VV^*$;
\item [Moore--Penrose inverse] $M^\dagger = V \Sigma^{-1} U^*$.
\end{description}
In this section, we illustrate how these fundamental properties of the matrix singular value decomposition carries over to singular value decomposition. 

\subsection{Approximation}

Unlike the finite-dimensional case, there are two ways to approximate an unbounded operator: using a bounded operator and using another unbounded operator. Since the sum of an unbounded operator and a bounded operator is always unbounded with the same domain, the approximation error cannot be measured using the operator norm. Nevertheless, we may still construct a sequence of bounded operators that approximates $A$ in the strong sense.

\begin{proposition}[Approximation by bounded operators]\label{prop:approx-bounded}
  Let $A: \mathcal{D}(A) \subseteq \mathcal{H} \to \mathcal{K}$ be a closed, densely defined linear operator such that $A = UT_\sigma V^*$ is the condensed singular value decomposition of $A$. Let $\{c_k\}_{k=1}^\infty$ be a sequence in $[0,\infty)$ such that $c_k\to \infty$. Then for all $k \in \mathbb{N}$, $A_k \coloneqq UT_{\1_{[0,c_k]}\sigma} V^* \in \mathcal{B}(\mathcal{H},
  \mathcal{K})$ and for all $x \in \mathcal{D}(A)$,
  \[
  \lim_{k \to \infty} A_k x = Ax.
  \]
\end{proposition}
\begin{proof}
Since $\1_{[0,c_k]}\sigma$ is bounded, $A_k$ is bounded as composition of bounded operators. By dominated convergence theorem, for $x \in \mathcal{D}(A)$,
\[
\lVert Ax - A_k x \rVert_\mathcal{K}^2 = \lVert T_\sigma V^*x - T_{\1_{[0,c_k]}\sigma} V^*x \rVert_{L^2_\mu }^2 = \int_M \lvert(\sigma - \1_{[0,c_k]}\sigma) V^*x\rvert^2 \,d\mu \to 0.\qedhere
\]
\end{proof}

Consider the closed subspace
\[
\mathcal{H}_c \coloneqq \{ x \in \mathcal{H} : T_{\mathbbm{1}_{(c,\infty)}\sigma} V^*x = 0 \} \subseteq \mathcal{D}(A) \quad \text{where } c>0. 
\]
For any $x\in\mathcal{D}(A)$, let $x=x_0+x_1$ with $x_0\in\mathcal{N}(A)^\perp$, $x_1\in\mathcal{N}(A)$. We obtain that $x_1+VT_{\1_{\{\sigma\le c\}}}V^*x \to x$ in the graph norm $\lVert \,\cdot \, \rVert_A$ as $c\to\infty$ since
\begin{align*}
  \lVert x-(x_1+V \1_{\{\sigma\le c\}} V^*x)\rVert_{\mathcal{H}}^2 &= \int_{\{\sigma>c\}} \lvert V^*x \rvert^2\,d\mu\to 0 \quad \text{and}\\
  \lVert Ax-A(x_1+V \1_{\{\sigma\le c\}} V^*x)\rVert_{\mathcal{K}}^2 &= \int_{\{\sigma>c\}} \sigma^2 \lvert V^*x \rvert^2\,d\mu\to 0.
\end{align*}

Hence, in the topology induced by the graph norm,
\[
\overline{\bigcup_{c>0}\mathcal{H}_c} = \mathcal{D}(A).
\]
Furthermore, if $x \in \mathcal{H}_c$, $Ax = U T_\sigma V^*x = U T_{\mathbbm{1}_{[0, c]}\sigma} V^*x$. Thus, the bounded operator $U T_{\mathbbm{1}_{[0, c]}\sigma} V^*$ recovers $A$ on $\mathcal{H}_c$. 

On the other hand, the singular value decomposition gives us a way to approximate $A$ with another unbounded operator within any desired precision.
\begin{proposition}[Approximation by unbounded operators]\label{prop:approx-unbounded}
  Let $A: \mathcal{D}(A) \subseteq \mathcal{H} \to \mathcal{K}$ be a closed, densely defined linear operator and $A = UT_\sigma V^*$ its condensed singular value decomposition. For $\varepsilon \in [0,\infty)$, $A-UT_{\1_{[\varepsilon,\infty)}\sigma} V^*$ extends uniquely to a bounded operator $E$ such that $\lVert E \rVert \leq \varepsilon$.
\end{proposition}
\begin{proof}
Since $\sigma \geq \1_{[\varepsilon,\infty)}\sigma$, we have $\mathcal{D}(A-UT_{\1_{[\varepsilon,\infty)}\sigma} V^*)= \mathcal{D}(A)$. In particular,
\[
\sup_{x\in \mathcal{D}(A) \setminus \{0\}}\frac{\lVert (A-UT_{\1_{[\varepsilon,\infty)}\sigma} V^*)x \rVert_\mathcal{K}}{\lVert x \rVert_\mathcal H} = \sup_{x\in \mathcal{D}(A)\setminus \{0\}} \frac{\lVert T_{\1_{[0,\varepsilon)}\sigma}V^*x  \rVert_{L^2_\mu}}{\lVert x \rVert_\mathcal{H}
} \leq \varepsilon.
\]
Therefore, $A-UT_{\1_{[\varepsilon,\infty)}\sigma} V^*$ is bounded on the dense subspace $\mathcal{D}(A) \subseteq \mathcal{H}$ and hence extends uniquely to a bounded operator $E$ such that 
\[
  \lVert E \rVert = \lVert UT_{\1_{[0,\varepsilon)}\sigma} V^* \rVert = \lVert T_{\1_{[0,\varepsilon)}\sigma} \rVert  \leq \varepsilon . \qedhere
\]
\end{proof}

\subsection{Adjoint operator}

The way to compute the singular value decomposition of $A^*$ is most conveniently written in the condensed multiplication operator form and direct integral form as they are already compositions of simple operators. For unbounded operators $A$ and $B$, the complication lies in the fact that $(AB)^*$ and $B^*A^*$ are not necessarily equal. Fortunately, the following lemma from \cite[Theorem~1.7, Theorem~1.8]{Konrad2012} and \cite[Theorem~2.3.4]{seifert2020} covers cases of our concern.

\begin{lemma}\label{lem:adjoint}
Let $\mathcal{H}$, $\mathcal{K}$ and $\mathcal{J}$ be Hilbert spaces, $A \colon \mathcal{D}(A) \subseteq \mathcal{H} \to \mathcal{K}$ be closed and densely defined, $B\in \mathcal{L} (\mathcal{K}, \mathcal{J})$, and $C\in \mathcal{L} (\mathcal{J}, \mathcal{H})$ surjective. Then $(BA)^*=A^*B^*$ and $(AC)^* = C^*A^*$ if $AC$ is densely defined. 
\end{lemma}

The singular value decomposition of the adjoint operator follows as a corollary.

\begin{proposition}[Singular value decomposition of adjoint operator]\label{prop:adjoint}
Let $A: \mathcal{D}(A) \subseteq \mathcal{H} \to \mathcal{K}$ be a closed, densely defined linear operator. 
\begin{enumerate}[\normalfont(i)]
  \item \label{it:ad1} If $A= UT_\sigma V^*$ is the condensed singular value decomposition of $A$, then $U^*\mathcal{D}(A^*) = \mathcal{D}(T_\sigma)$, and $A^* = VT_\sigma U^*$.
  \item \label{it:ad2} If $A = U \Sigma V^*$ is the singular value decomposition of $A$ in direct integral form, then $U^*\mathcal{D}(A^*) = \mathcal{D}(\Sigma^*)$, and $A^* = V\Sigma^* U^*$.
\end{enumerate}
\end{proposition}
\begin{proof}
Let $A = U T_\sigma V^*$ be the condensed singular value decomposition of $A$. By Theorem~\ref{thm:condensed-mult}, $\mathcal{D}(T_\sigma V^*) = \mathcal{D}(A)$. Therefore, $T_\sigma V^*$ is also densely defined. By Lemma~\ref{lem:adjoint},
\[ (T_\sigma V^*)^* = (V^*)^* T_\sigma^* = V T_\sigma.\]
Applying Lemma~\ref{lem:adjoint} again,
\[(U (T_\sigma V^*))^* = (T_\sigma V^*)^* U^* = V T_\sigma U^*. \]
A verbatim argument proves \ref{it:ad2}.
\end{proof}

\subsection{Orthogonal projections}
Comparing with the finite-dimensional case, the projection onto range and kernel involves one subtlety: The range of an unbounded operator is not necessarily closed. Consequently, the orthogonal projection we obtain is onto the closure of $\mathcal{R}(A)$.
\begin{proposition}[Orthogonal projection onto fundamental subspaces]\label{prop:proj}
Let $A: \mathcal{D}(A) \subseteq \mathcal{H} \to \mathcal{K}$ be a closed, densely defined unbounded operator and $A =U T_\sigma V^*$ be the condensed singular value decomposition of $A$. Then
\[
\range(U) = \overline{\range(A)}, \quad \mathcal{N}(V^*) = \mathcal{N}(A), \quad VV^* = P_{\mathcal{N}(A)^\perp}, \quad UU^* = P_{\overline{\range(A)}}.
\]
\end{proposition}
\begin{proof}
Since an isometry between Hilbert spaces has closed range, we have $\overline{\range(A)}\subseteq \range(U)$ as $Ax = UT_\sigma V^*x \in \range(U)$ for all $ x \in \mathcal{D}(A)$. Conversely, let $y \in \range(U)$. Since $U$ is an isometry from $L^2_\mu$ to a subspace of $\mathcal{K}$, $ U^*y \in L^2_\mu$. Because $\sigma > 0$ almost everywhere, $T_\sigma$ is injective and self-adjoint, so $\range(T_\sigma) \subseteq L^2_\mu$ is dense. Therefore, there exists a sequence $\{f_n\}_{n=1}^\infty \subseteq \mathcal{D}(T_\sigma)$ such that $T_\sigma f_n \to U^*y$. For each $n$, 
\[A(Vf_n) = U T_\sigma V^* (V f_n) = U T_\sigma f_n, \quad 
\text{and } Vf_n \in \mathcal{D}(A).\] 
Since $U$ is bounded, $A(Vf_n) = U(T_\sigma f_n) \to UU^*y = y$. Therefore, $\range(U) = \overline{\range(A)}$. 

By \eqref{eq:V}, $V$ is an isometry onto $\mathcal{N}(A)^\perp$, so $\range(V) = \mathcal{N}(A)^\perp$. Because $A$ is a closed operator, $\mathcal{N}(A)$ is closed, so
\[
  \mathcal{N}(V^*) = \range(V)^\perp = (\mathcal{N}(A)^\perp)^\perp = \mathcal{N}(A).
\] 

Lastly, since $V$ is an isometry, $VV^*$ is the orthogonal projection onto its range, $\range(V)$. Substituting $\range(V) = \mathcal{N}(A)^\perp$ yields $VV^* = P_{\mathcal{N}(A)^\perp}$. Similarly, $UU^* = P_{\range(U)} = P_{\overline{\range(A)}}$.
\end{proof}

\subsection{Moore--Penrose Inverse}
The theory of generalized inverse has been extended to unbounded operators \cite{Tseng1949,BIC1963} and is characterized by the following result \cite[Theorem~6]{BIC1963}.

\begin{proposition}\label{prop:pinv}
Let $A: \mathcal{D}(A) \subseteq \mathcal{H} \to \mathcal{K}$ be densely defined. Then there exists a unique operator $A^\dagger : \range(A) \oplus \range(A)^\perp \subseteq \mathcal{K} \to \mathcal{H}$ satisfying
\begin{enumerate}[\normalfont(i)]
    \item $\range(A^\dagger) \subseteq \mathcal{D}(A)$, 
    \label{cond:1}
    \item $\mathcal{N}(A^\dagger) = \range(A)^\perp$, \label{cond:2}
    \item $AA^\dagger = P_{\overline {\range(A)}}\vert_{\mathcal D(A^\dagger)}$, \label{cond:3}
    \item $A^\dagger A =P_{\overline{\range(A^\dagger)}}\lvert_{\mathcal D(A)}$, \label{cond:4}
\end{enumerate}
if and only if 
\begin{equation}\label{eq:existence}
P_{\overline{\mathcal{N}(A)}}\mathcal{D}(A)  \subseteq \mathcal{N}(A). 
\end{equation}
\end{proposition}

Since equation~\eqref{eq:existence} is necessarily true for closed operators, the singular value decomposition computes $A^\dagger$ by the exact same formula as the finite-dimensional case.
\begin{proposition}[Moore--Penrose inverse]\label{prs:mp-inv}
Let $A: \mathcal{D}(A) \subseteq \mathcal{H} \to \mathcal{K}$ be a closed, densely defined linear operator. If $A = UT_\sigma V^*$ is the condensed singular value decomposition of $A$, then $A^\dagger = V T_{\sigma^{-1}}  U^*$.
\end{proposition}

\begin{proof}
Let $y = y_1 + y_2 \in \range(A) \oplus \range(A)^\perp$ where $y_1\in \range(A)$, $y_2 \in \range(A)^\perp$ and $x\in\mathcal{D}(A)$ such that $y_1 = Ax$. By Proposition~\ref{prop:proj}, $y_2 \in \range(A)^\perp = \mathcal{N}(U^*)$. Since $U^* y_1 \in \range(T_\sigma)$,
\begin{equation}\label{eq:range}
    V T_{\sigma^{-1}}  U^*(y_1+y_2) =V T_{\sigma^{-1}} (U^*y_1) + V T_{\sigma^{-1}} (U^*y_2) = V T_{\sigma^{-1}}  U^* y_1
\end{equation}
is well-defined. Therefore, we obtain the unbounded operator $ V T_{\sigma^{-1}}  U^*: \range(A) \oplus \range(A)^\perp \subseteq \mathcal{K} \to \mathcal{H}$.

It remains to check that $V T_{\sigma^{-1}}  U^*$ satisfies conditions~\ref{cond:1}---\ref{cond:4}.  By \eqref{eq:range} and Proposition~\ref{prop:proj}, 
\begin{equation}\label{eq:proj-kernel-perp}
V T_{\sigma^{-1}}  U^*(y_1+y_2) = V T_{\sigma^{-1}}  U^*Ax = V T_{\sigma^{-1}}  U^* UT_\sigma V^* x = VV^*x = P_{\mathcal{N}(A)^\perp}x \in \mathcal{D}(A),
\end{equation}
implying \ref{cond:1}. Condition~\ref{cond:2} holds as $A^\dagger(Ax+y_2) = 0$ if and only if $x \in \mathcal{N}(A)$. Condition~\ref{cond:3} holds by Proposition~\ref{prop:proj} since
\[
A (V T_{\sigma^{-1}}U^*(y_1+y_2)) = A(P_{\mathcal{N}(A)^\perp} x) = A(x - P_{\mathcal{N}(A)}x) = Ax = y_1.
\]
Last, by condition~\ref{cond:1} and equation~\eqref{eq:proj-kernel-perp}, $\range(A^\dagger) \subseteq \mathcal{D}(A)\cap \mathcal{N}(A)^\perp$. For any $x \in \mathcal{D}(A)\cap \mathcal{N}(A)^\perp$, $A^\dagger A x = P_{\mathcal{N}(A)^\perp} x = x\in \range(A^\dagger)$, therefore $\range(A^\dagger) = \mathcal{D}(A)\cap \mathcal{N}(A)^\perp$. Since $\mathcal{D}(A)$ is dense in $\mathcal{H}$ and $\mathcal{N}(A)$ is closed, we have $\overline{\range(A^\dagger)} = \overline{\mathcal{D}(A)\cap \mathcal{N}(A)^\perp}= \mathcal{N}(A)^\perp$, hence equation~\eqref{eq:proj-kernel-perp} gives condition~\ref{cond:4}.
\end{proof}

\section{Singular value decomposition of gradient operators}\label{sec:eg-grad}

We now switch our focus to computing the singular value decomposition of some fundamental operators arising from various scenarios. We begin with arguably the most commonly encountered unbounded operator: the gradient operator $\nabla$. Unlike most Laplacian operators, the gradient operators are usually not self-adjoint. In this section, we compute the singular value decomposition of $\nabla$ over $\R^n$, bounded domains, the manifold $S^2$ and compact Lie groups.
\subsection{Gradient on $\R^n$}\label{sec:grad-Rn} 
We begin with the canonical gradient on Euclidean space:
\[
  \nabla: H^1(\R^n) \subseteq L^2(\R^n) \to L^2(\R^n, \C^n), \quad 
 f \mapsto \begin{bmatrix}
    \partial_1 f\\ \vdots \\ \partial_n f
\end{bmatrix},
\]
where the partial derivatives are defined in the weak sense. It is standard to check that $\nabla$ is a closed operator. Integrating by parts, its adjoint is the negative divergence
\[
 - \operatorname{div} \begin{bmatrix}
    f_1\\ \vdots \\ f_n
\end{bmatrix} = -\nabla \cdot \begin{bmatrix}
    f_1\\ \vdots \\ f_n
\end{bmatrix} = - \sum_{k=1}^n \partial_k f_k, \quad \mathcal{D}(- \operatorname{div}) = \Biggl\{\begin{bmatrix}
    f_1\\ \vdots \\ f_n
\end{bmatrix} \in  L^2(\R^n, \C^n) : - \sum_{k=1}^n \partial_k f_k \in L^2(\R^n) \Biggr\}.\]
Thus, $\nabla^* \nabla = -\Delta: H^2 \subseteq L^2(\R^n) \to L^2(\R^n)$, the negative Laplacian. 

Let $\mathcal{F}: L^2_x(\R^n) \to L^2_\xi(\R^n)$ denote the Fourier transform
\[
\mathcal{F}f(\xi) = \int_{\R^n} f(x)e^{-i2\pi \xi \cdot x} \, dx,
\]
whose adjoint is the inverse Fourier transform $\mathcal{F}^* = \mathcal{F}^{-1}$. For all $f \in H^2(\R^n)$,
\[
-\Delta f = \mathcal{F}^* T_{(2\pi \lvert \xi \rvert)^2} \mathcal{F}f.
\]
Let $\sigma(\xi) = 2\pi \lvert \xi \rvert$, which is nonzero
almost everywhere. Let $\sigma^{-1}(0)=0$. For
$g \in \mathcal{D}(T_{\sigma^{-1}})$,
\[
\nabla\mathcal{F}^* T_{\sigma^{-1}}g
=
\nabla_x
\int_{\R^n}
e^{i2\pi x\cdot\xi}
\frac{g(\xi)}{2\pi\lvert\xi\rvert}\,d\xi.
\]
For $k=1,\dots,n$, the $k$th component is related to the Riesz transform $\mathcal{R}_k$ \cite[Section~5.1]{Grafakos2014} as
\[
(\nabla\mathcal{F}^*T_{\sigma^{-1}}g)_k(x)
=
\int_{\R^n}
e^{i2\pi x\cdot\xi}
\frac{i\xi_k g(\xi)}{\lvert\xi\rvert}\,d\xi
=
-\mathcal R_k\mathcal F^*g(x).
\]
Since $\mathcal{D}(T_{\sigma^{-1}})$ is dense in
$L^2_\xi(\R^n)$ and the multiplier $i\xi/\lvert\xi\rvert$ is unit norm almost everywhere, the operator extends uniquely to the isometry
\[
\mathcal R:
L^2_\xi(\R^n)
\to
L^2_x(\R^n,\C^n),
\quad
g\mapsto
-
\begin{bmatrix}
\mathcal R_1\mathcal F^*g\\
\vdots\\
\mathcal R_n\mathcal F^*g
\end{bmatrix}.
\]
Since $\mathcal{R}_k^* = - \mathcal{R}_k$, 
\[
\mathcal{R}^* \begin{bmatrix}
    f_1 \\
    \vdots\\
    f_n
\end{bmatrix}= \mathcal{F} \sum_{k=1}^n \mathcal{R}_k f_k.
\]
Therefore, for all $f \in H^1(\R^n)$ and $[f_1,
\cdots,f_n]^\tp \in \mathcal{D}(-\operatorname{div})$, 
\[\nabla f = \mathcal{R}T_\sigma \mathcal{F} f, \quad -\operatorname{div}\begin{bmatrix}
    f_1 \\
    \vdots\\
    f_n
\end{bmatrix} = \mathcal{F}^*T_\sigma \mathcal{R}^* \begin{bmatrix}
    f_1 \\
    \vdots\\
    f_n
\end{bmatrix}. \]
When $n=1$, the gradient is simply the first derivative $d/dx$, whose spectrum is the entire imaginary axis $i\R$. On the other hand, its singular values are $\mathsf{\Sigma}(d/dx) = [0,\infty)$. In this case, the Riesz transform $\mathcal{R}_1$ is known as the Hilbert transform \cite[Section~5.1]{Grafakos2014}.

\subsection{Gradient on bounded domain}\label{sec:grad-bounded}
Now, let $\Omega \subseteq \R^n$ be a bounded domain with smooth boundary. The gradient operator $\nabla: H^1_0(\Omega) \subseteq L^2(\Omega) \to L^2(\Omega, \C^n)$ is again the closed operator defined by the formula
\[\nabla f(x) = \begin{bmatrix}
    \partial_1 f(x) \\
    \vdots \\
    \partial_n f(x)
\end{bmatrix},
\] 
with the negative divergence as its adjoint. Hence 
\[
\nabla^* \nabla = -\Delta: H^2(\Omega) \cap H^1_0(\Omega) \subseteq L^2(\Omega) \to L^2(\Omega)
\]
is the Laplacian operator. By the standard result in PDE \cite[Section~6.5.1, Theorem~1]{Evans2010}, the spectrum of the Laplacian is discrete, i.e, $\mathsf{\Lambda}(-\Delta) = \{\lambda_k\}_{k =1}^\infty$, counting multiplicity. Let $\{f_k\}_{k=1}^\infty$ be an eigenbasis corresponding normalized eigenbasis, then by Corollary~\ref{cor:discrete}, 
\[
\nabla = \sum_{k=1}^\infty \sqrt{\lambda_k} \langle \, \cdot, f_k \rangle \frac{1}{\sqrt{\lambda_k}}\nabla f_k.
\]

\subsection{Gradient on manifolds}
Despite the ubiquity of the flat space $\R^n$, the geometric nature of various applications require one to consider gradients over more complicated manifolds. We briefly recall some basic notations and definitions here. Let $\mathcal{M}$ be a compact oriented manifold. A Riemannian metric smoothly assigns an inner product $\langle\, \cdot,\cdot \,\rangle_x$ to the tangent space $T_x \mathcal{M}$ at each point $x \in \mathcal{M}$. The Riemannian metric uniquely determines a volume form $dV$ that induces the measure 
\[
\mu(E) \coloneqq \int_E dV
\]
for any Borel subset $E \subseteq \mathcal{M}$. Consequently, $L^2(\mathcal{M})$, and $H^k(\mathcal{M})$ are defined with respect to $\mu$ for all $k \in \mathbb{N}$. The Riemannian metric also induces on the set of smooth vector fields $\Gamma(T\mathcal{M})$ the inner product
\[
    \langle \mathbf{u} , \mathbf{v} \rangle
    =\int_\mathcal{M} \langle \mathbf{u}(x),\mathbf{v}(x) \rangle_x \,dV, \quad \mathbf{u} , \mathbf{v} \in \Gamma(T\mathcal{M}).
\]
We denote by $L^2(T\mathcal{M})$ the completion of $\Gamma(T\mathcal{M})$ in the topology induced by the inner product. The subspaces $C^k(T\mathcal{M})$ and $H^k(T\mathcal{M})$ are defined similarly. We now consider the two cases when $\mathcal{M}$ is a sphere and a compact Lie group.

\subsection{Gradient on sphere}
Functions over the sphere are ubiquitous in applications, including in acoustics \cite{signal1,signal2}, computer graphics \cite{graphics}, geophysics \cite{geo1,geo2}, and medical imaging \cite{med}. Let $S^2 \subseteq \R^3$ denote the $2$-dimensional unit sphere equipped with the round metric, i.e., the metric induced from the Euclidean metric. For $f \in C^\infty(S^2)$, let $\widetilde{f}$ be any smooth extension of $f$ to a neighborhood of $S^2$ and define
\[
\nabla_{S^2} f(x) =
(I-xx^\tp)\nabla_{\R^3}\widetilde{f}(x),
\quad x \in S^2.
\]
It extends to the closed operator
\[
\nabla_{S^2}: H^1(S^2) \subseteq L^2(S^2) \to L^2(TS^2).
\]
Its adjoint is $- \operatorname{div}_{S^2}$ where for $f \in H^1(S^2)$ and $\mathbf{v}\in\mathcal{D}(\operatorname{div}_{S^2})$, 
\[
\langle\nabla_{S^2}f,\mathbf{v}\rangle_{L^2(TS^2)}
=
-\langle f, \operatorname{div}_{S^2}\mathbf{v} \rangle_{L^2(S^2)}
\]
and
\[
\Delta_{S^2} \coloneqq -\operatorname{div}_{S^2}\nabla_{S^2}
= \nabla_{S^2}^*\nabla_{S^2}: H^2(S^2) \subseteq L^2(S^2) \to L^2(S^2)
\]
is the positive Laplace--Beltrami operator on $S^2$. Its spectrum \cite[Section~3.3]{AH2012} is
$\mathsf{\Lambda}(\Delta_{S^2})= \{ \ell(\ell+1): \ell= 0,1,2,\dots \}$, with eigenspace for $\ell(\ell+1)$
\[
\mathcal{H}_{\ell(\ell+1)}
=\{p|_{S^2}: p(x,y,z) \text{ is a harmonic homogeneous polynomial of degree }\ell \},
\]
known as the spherical harmonics. They are often written in the spherical coordinates $(\theta,\varphi)$. For $m=0,1,\dots,\ell$, consider
\[
p_{\ell,m}(\theta,\varphi)
=\sqrt{\frac{2\ell+1}{4\pi}\frac{(\ell-m)!}{(\ell+m)!}}
P_\ell^m(\cos\theta)e^{im\varphi},
\quad \ell=0,1,2,\dots,
\]
where
\[
P_\ell^m(x)=(-1)^m(1-x^2)^{m/2}\frac{d^m}{dx^m}P_\ell(x),
\quad
P_\ell(x)
=\frac{1}{2^\ell\ell!}\frac{d^\ell}{dx^\ell}(x^2-1)^\ell.
\]
For $m=1,\dots,\ell$, let $p_{\ell,-m}=(-1)^m\overline{p_{\ell,m}}$. By \cite[Theorem~2.9]{AH2012}, $\{p_{\ell,m}\}_{m=-\ell}^\ell$ forms an orthonormal basis of
$\mathcal{H}_{\ell(\ell+1)}$. In other words, the spectral
decomposition of $\Delta_{S^2}$ can be written in the discrete form as
\[
\Delta_{S^2}f
=
\sum_{\ell=0}^\infty
\sum_{m=-\ell}^{\ell}
\ell(\ell+1)
\langle f,p_{\ell,m}\rangle p_{\ell,m},
\quad
f\in H^2(S^2).
\]
Consequently, $\mathsf{\Sigma}(\nabla_{S^2})= \{ \sqrt{\ell(\ell+1)} : \ell=0,1,2,\dots \}$. The function $p_{0,0}$ spans $\mathcal{N}(\nabla_{S^2})$. For $\ell\ge1$ and $m=-\ell,\dots,\ell$, the right singular vectors are $p_{\ell,m}$ and, by \eqref{eq:basis}, the corresponding left singular vectors are
\[
q_{\ell,m}=\frac{1}{\sqrt{\ell(\ell+1)}}
\nabla_{S^2}p_{\ell,m}=
\frac{1}{\sqrt{\ell(\ell+1)}}
\begin{bmatrix}
\partial_\theta p_{\ell,m}\\
\frac{1}{\sin\theta}\partial_\varphi p_{\ell,m}
\end{bmatrix}
=
\frac{1}{\sqrt{\ell(\ell+1)}}
\begin{bmatrix}
\partial_\theta p_{\ell,m}\\
\frac{im}{\sin\theta}p_{\ell,m}
\end{bmatrix}.
\]
Therefore, the singular value decomposition for the spherical gradient and divergence are
respectively
\[
\nabla_{S^2}f
=
\sum_{\ell=1}^\infty \sum_{m=-\ell}^{\ell}
\sqrt{\ell(\ell+1)} \langle f,p_{\ell,m}\rangle_{L^2(S^2)} q_{\ell,m},
\quad
f\in H^1(S^2),
\]
and
\[
\operatorname{div}_{S^2}\mathbf{v}
=
-\sum_{\ell=1}^\infty \sum_{m=-\ell}^{\ell} \sqrt{\ell(\ell+1)} \langle\mathbf{v},q_{\ell,m}\rangle_{L^2(TS^2)} p_{\ell,m},
\quad
\mathbf{v}\in\mathcal{D}(\operatorname{div}_{S^2}).
\]

\subsection{Gradient on compact Lie groups}
Interests in Lie groups, i.e., groups equipped with manifold structure, from the applied point-of-view have increased significantly as they naturally characterize symmetries in the model space or data in optimization and machine learning \cite{TO2020,KT2024,CNN2018}. Among them, compact Lie groups, such as the real orthogonal group $\O(n,\R)$ and the complex unitary group $\U(n)$, are of particular importance as they characterize rotational symmetries. For the sake of concreteness, we focus on subgroups of $\U(n)$, but the same argument extends to other compact Lie groups. 

Let $\G \subseteq \U(n)$ be a compact Lie group. Throughout this subsection, $L^2(\G)$ is taken with respect to the
normalized Haar probability measure on $\G$. Its Lie algebra $\g$ is defined by
\[
\g = \{X \in \C^{n \times n}: \exp(tX) \in \G 
\text{ for all } t \in \R\} \subseteq \{X \in \C^{n \times n}: X^* = -X\}
\]
where $\exp(\cdot)$ is the matrix exponential function. The dimension $d\coloneqq \dim\g$ is the dimension of $\G$. For each $g \in \G$, $g\g = T_g \G$ is the tangent space of $\G$ at $g$. The natural inner product $\langle X,Y\rangle=\tr(Y^*X)$, induces a natural Riemannian metric on $\G$. On the other hand, the Lie algebra elements $X$ acts on $C^\infty(\G)$ as the first-order left-invariant differential operators on $\G$:
\[
(X\cdot f)(g) \coloneqq \frac{d}{dt} f(g\exp(tX))\rvert_{t=0}, \quad f \in C^\infty(\G).
\]
Each $X\in\mathfrak g$ extends continuously from $H^1(\G)$ to $L^2(\G)$. Let $\{X_k\}_{k=1}^d$ be an orthonormal basis of
$\mathfrak g$. Then the operator
\[
\nabla_\G : H^1(\G) \subseteq L^2(\G) \to L^2(\G,\C^d), \quad f \mapsto \begin{bmatrix}
  X_1\cdot f\\
  \vdots\\
  X_d \cdot f
\end{bmatrix}
\]
is closed.

Similar to the previous two examples, the adjoint of $\nabla_\G$ is the negative divergence 
\[
  \nabla_\G^* \begin{bmatrix}
    f_1\\
    \vdots\\
    f_d
  \end{bmatrix} = - \sum_{k=1}^d X_k f_k, \quad \text{with } \mathcal{D}(\nabla_\G^*) = \Biggl\{\begin{bmatrix}
    f_1\\
    \vdots\\
    f_d
  \end{bmatrix}\in L^2(\G,\C^d):\sum_{k=1}^d X_k f_k\in L^2(\G)\Biggr\}.
\]
Therefore, 
\[
  \nabla_\G^* \nabla_\G = \Delta_\G: H^2(\G) \to L^2(\G), \quad f \mapsto -\sum_{k=1}^d X_k^2f,
\]
the Casimir element, which is also the Laplace--Beltrami operator on $\G$.

The spectral decomposition $\Delta_\G$ is obtained via representation theory. Let $\widehat{\G}$ denote the (countable) set of equivalence classes of irreducible representations of $\G$. Abusing the notation slightly, we will simply write $\rho \in \widehat{\G}$ when $\rho$ is a representative from its equivalence class in $\widehat{\G}$.

Given an irreducible representation $\rho: \G \to \U(\mathcal{H}_\rho)$ and $v,w \in \mathcal{H}_\rho$, define the matrix coefficient function
\[
\rho_{v,w}: \G \to \C \quad g \mapsto \langle v, \rho(g) w \rangle \in L^2(\G).
\] 
By \cite[Proposition~10.6]{Hall2015}, there exists a quadratic function $c(\rho)$ on highest weights of $\rho$ such that 
\[
\Delta_\G \rho_{v,w} = c(\rho) \rho_{v,w} \quad \text{ for all } v,w \in \mathcal{H}_\rho.
\]
By Peter--Weyl theorem \cite[Section~5.2]{Folland2016}, $\{\sqrt{d_\rho}\rho_{ij}: i,j = 1,\dots,d_\rho\}$ is an orthonormal basis of $L^2(\G)$ and 
\[
\Delta_\G f = \sum_{\rho \in \widehat{\G}} \sum_{i,j=1}^{d_\rho} c(\rho) d_\rho \langle f , \rho_{ij} \rangle \rho_{ij}.
\]
By Corollary~\ref{cor:discrete}, the right singular vectors of $\nabla_\G$ are $\{\sqrt{d_\rho}\rho_{i,j}: i,j = 1,\dots,d_\rho\, , \rho \in \widehat{\G},\, c(\rho)>0\}$ with singular values $\{\sqrt{c(\rho)}: \rho \in \widehat{\G},\, c(\rho)>0\}$, and left singular vectors 
\[
  \sqrt{\frac{d_\pi}{c(\pi)}} \nabla_\G \pi_{ij} = \sqrt{\frac{d_\pi}{c(\pi)}} \begin{bmatrix}
    X_1 \pi_{ij}\\
    \vdots\\
    X_d \pi_{ij}
  \end{bmatrix}.
\]
The singular value decomposition of $\nabla_\G$ is given by 
\[
\nabla_\G f =  \sum_{\substack{\rho\in\widehat{\G}\\c(\rho)>0}}   \sum_{i,j=1}^{d_\rho} \sqrt{c(\rho)} \langle f , \sqrt{d_\rho}\rho_{ij} \rangle \sqrt{\frac{d_\rho}{c(\rho)}} \nabla_\G \rho_{ij} = \sum_{\substack{\rho\in\widehat{\G}\\c(\rho)>0}}   \sum_{i,j=1}^{d_\rho} d_\rho \langle f , \rho_{ij} \rangle \nabla_\G \rho_{ij}.
\]
The formula further simplifies if we restrict our attention to functions invariant under conjugation, i.e., the spaces
\[
ZL^2(\G) = \{f \in L^2(\G): f(y^{-1}xy) = f(x) \text{ for all } x, y \in \G\}, \quad ZH^1(\G) = H^1(\G) \cap ZL^2(\G).
\]
Such functions naturally arise in various applications \cite{Hielscher2013,Qiu2014}. The characters
$\{\chi_\rho\}_{\rho\in\widehat{\G}}$, where
$\chi_\rho\coloneqq\tr[\rho(\cdot)]$, form an orthonormal basis of
$ZL^2(\G)$, and each $\chi_\rho$ is an eigenfunction of
$\Delta_\G$ with eigenvalue $c(\rho)$
\cite[Proposition~5.23]{Folland2016}. Hence
\[
\Delta_\G\vert_{ZH^2(\G)} f
=
\sum_{\rho \in \widehat{\G}}
c(\rho) \langle f, \chi_\rho \rangle \chi_\rho,
\quad
f \in ZH^2(\G).
\]
Consequently, the singular value decomposition of $\nabla_\G$ is 
\begin{equation}\label{eq:character}
  \nabla_\G \vert_{ZH^1(\G)} f = \sum_{\substack{\rho\in\widehat{\G}\\c(\rho)>0}} \sqrt{c(\rho)} \langle f, \chi_\rho \rangle \frac{\nabla \chi_\rho}{\sqrt{c(\rho)}} = \sum_{\substack{\rho\in\widehat{\G}\\c(\rho)>0}} \langle f, \chi_\rho \rangle \nabla_\G \chi_\rho.
\end{equation}

As a concrete example, we compute the singular value decomposition of the gradient on periodic functions. Consider the $n$-torus 
\[ \G = \U(1)^n \subseteq \U(n) = \Biggl\{e^{i\theta} \coloneqq \begin{bmatrix}
  e^{i\theta_1}\\ \vdots\\ e^{i\theta_n}
  \end{bmatrix}\cong\begin{bmatrix}
e^{i\theta_1}\\&\ddots\\ & &e^{i\theta_n}
\end{bmatrix} \in \U(n):  \theta = \begin{bmatrix}
\theta_1\\ \vdots\\ \theta_n
\end{bmatrix} \in (\R/2\pi\Z)^n\Biggr\},
\] 
Functions on $\G$ are equivalent to $n$-variant periodic functions as
\[
L^2(\G)=ZL^2(\G)
\cong
L^2\biggl([0,2\pi]^n,\frac{d\theta}{(2\pi)^n}\biggr),
\]
where $\widetilde f(\theta)\coloneqq f(e^{i\theta_1},\dots,e^{i\theta_n})$. The Lie algebra is $\mathfrak{g} = \mathfrak{u}(1)^n = i\R^n$. We can hence choose the orthonormal basis $\{X_k = ie_k\}_{k=1}^n$. As differential operators,
\[
(X_k \cdot f)(e^{i\theta}) = \frac{d}{dt} f(e^{i\theta_1},\dots, e^{i\theta_k+ it}, \dots,e^{i\theta_n}) = \frac{d}{dt} \widetilde{f}(\theta_1,\dots, \theta_k + t,\dots,\theta_n) = \frac{\partial}{\partial \theta_k} \widetilde{f}(\theta).\]
Therefore, the gradient is exactly the canonical gradient
\[
\nabla_\G f(e^{i\theta_1},\dots,e^{i\theta_n}) = \begin{bmatrix}
  \frac{\partial\widetilde{f}}{\partial \theta_1}(\theta)\\
  \vdots\\
  \frac{\partial\widetilde{f}}{\partial \theta_n}(\theta)
\end{bmatrix}.
\]
Since $\G$ is Abelian, $\widehat{\G}
=
\{\rho_k:k\in\Z^n\}$ where $\rho_k(e^{i\theta})=e^{ik\cdot\theta}$. Since $\Delta_{\U(1)^n}\rho_k
=
\lvert k \rvert^2\rho_k
$ and $\nabla_\G\rho_k
=
ik\,e^{ik\cdot\theta}$, we obtain from \eqref{eq:character},
\[
\nabla_\G f(e^{i\theta})
=
\sum_{k\in\Z^n\setminus\{0\}}
ik\,\frac{1}{(2\pi)^n}
e^{ik\cdot\theta} \int_{[0,2\pi]^n}
\widetilde f(\tau)e^{-ik\cdot\tau}\,d\tau . 
\]
\subsection{Hilbert complex}\label{sec:complex}
Hilbert complexes arise naturally in geometry and analysis, particularly in the study of differential complexes on Riemannian manifolds \cite{hilbertcomplex,complex2}, and in numerical analysis through finite element exterior calculus \cite{FEEC,Hu2026}.
Let
\[
0 \longrightarrow W^0
\xrightarrow{d^0} W^1
\xrightarrow{d^1} \cdots
\xrightarrow{d^{n-1}} W^n
\longrightarrow 0
\]
be a Hilbert complex, where each $W^k$ is a Hilbert space and, for $k=0,\dots,n-1$, $d^k:\mathcal D(d^k)\subseteq W^k\longrightarrow W^{k+1}$ is a closed and densely defined operator satisfying $\range(d^{k-1})\subseteq\mathcal D(d^k)$ and $d^kd^{k-1}=0$ with $d^{-1} = d^n=0$. For $k=0,\dots,n-1$, let $\delta^k\coloneqq(d^k)^*:\mathcal D(\delta^k)\subseteq W^{k+1}\to W^k$ and let $\delta^{-1} = \delta^n =0$.

For $k=0,\dots,n-1$, let $d^k=U_kT_{\sigma_k}V_k^*$ be the condensed singular value decomposition of $d^k$ in multiplication operator form. Since $d^kd^{k-1}=0$, applying Proposition~\ref{prop:proj} gives, for $k=1,\dots,n-1$,
\begin{equation}\label{eq:proj0}
U_kT_{\sigma_k}V^*_k U_{k-1}T_{\sigma_{k-1}}V^*_{k-1}=0
\Longrightarrow
V_k^*U_{k-1}=0,
\quad\text{and}\quad
\range(d^{k-1})\subseteq\mathcal N(d^k).
\end{equation}

The $k$th cohomology of the Hilbert complex is the quotient space
\[H^k(W)\coloneqq\mathcal N(d^k)/\range(d^{k-1}).\]
Suppose that the Hilbert complex is closed, i.e., $\range(d^k)$ is closed for every $k$. By the closed range theorem, $\range(\delta^k)$ is also closed for every $k$. Then every cohomology class $[x]\in H^k(W)$ contains a unique element $\widehat{x}\in
\bigl(\mathcal N(d^k)\cap\mathcal N(\delta^{k-1})\bigr)\cap[x]$ and hence
\[H^k(W)
\cong
\mathcal N(d^k)\cap\mathcal N(\delta^{k-1})=
\mathcal N(d^k)\cap\range(d^{k-1})^\perp \eqqcolon \mathfrak H^k.\]

By Proposition~\ref{prop:adjoint}, $\delta^k=V_kT_{\sigma_k}U_k^*$, and by Proposition~\ref{prop:proj},
\[
  P_{\range(d^k)}=U_kU_k^*,
  \quad
  P_{\mathcal N(\delta^{k-1})}=I-U_{k-1}U_{k-1}^*,
  \quad
  P_{\range(\delta^k)}=V_kV_k^*,
  \quad
  P_{\mathcal N(d^k)}=I-V_kV_k^*.
\]
Thus, given $[x]\in H^k(W)$ and a representative $x\in\mathcal N(d^k)$, $\widehat{x}
=
x-U_{k-1}U_{k-1}^*x
\in
\mathfrak H^k$. Since $I-U_{k-1}U_{k-1}^*$ is an orthogonal projection, we have
\[
  \widehat{x}
  =
  \argmin_{y\in[x]}\lVert y\rVert_{W^k}.
\]

By \eqref{eq:proj0},
\[
  (I-V_kV_k^*)(I-U_{k-1}U_{k-1}^*)
  =
  I-V_kV_k^*-U_{k-1}U_{k-1}^*
  =
  P_{\mathfrak H^k}.
\]
Using the conventions $U_{-1}U_{-1}^*=0$ and $V_nV_n^*=0$, we recover the Hodge decomposition of the closed Hilbert complex
\[
  W^k
  =
  \range(d^{k-1})
  \oplus
  \mathfrak H^k
  \oplus
  \range(\delta^k).
\]

\subsection{Exterior derivative}\label{sec:derham}

Consider the example of the de Rham complex. Let $\mathcal{M}$ be a $n$-dimensional closed Riemannian oriented manifold. For $k=1,\dots,n$, the Riemannian metric on $\mathcal{M}$
naturally defines a dual metric on $T_x^*\mathcal{M}$ and induces an inner product on $\wedge^k T_x^*\mathcal{M}$, which we continue to denote by $\langle\, \cdot,\cdot \,\rangle_x$. The space of smooth $k$-forms $\Omega^k(\mathcal{M})$ is thus equipped with the inner product
\[
\langle \alpha, \beta \rangle
=
\int_\mathcal{M}
\langle \alpha(x),\beta(x) \rangle_x\,dV,
\quad
\alpha,\beta\in\Omega^{k}(\mathcal M).
\]
Denote the closure of $\Omega^k(\mathcal{M})$ with respect to this inner
product by $L^2(\wedge^k T^*\mathcal{M})$. For $k=0,\dots,n-1$, define the maximal exterior derivative by $d^k:
\mathcal D(d^k)\subseteq L^2(\wedge^kT^*\mathcal M)
\longrightarrow
L^2(\wedge^{k+1}T^*\mathcal M)$ with domain
\[
\mathcal D(d^k)
=
\{
\alpha\in L^2(\wedge^kT^*\mathcal M):
d\alpha\in L^2(\wedge^{k+1}T^*\mathcal M)
\text{ distributionally}
\}.
\]
With $d^n=0$, these operators are closed and densely defined and form a Hilbert complex. For a smooth form, $d^k$ is
given in local coordinates by
\[
d\alpha
=\sum_{l=1}^n
\sum_{j_1<\cdots<j_k}
\frac{\partial\alpha_{j_1\cdots j_k}}{\partial x_l}\;
dx_l\wedge dx_{j_1}\wedge\cdots\wedge dx_{j_k} \quad \text{ for }\alpha
=\sum_{1\le j_1<\cdots<j_k\le n}
\alpha_{j_1\cdots j_k}\;
dx_{j_1}\wedge\cdots\wedge dx_{j_k} .\]
For $k=0,\dots,n-1$, let $\delta^k\coloneqq(d^k)^*:
\mathcal D(\delta^k)\subseteq
L^2(\wedge^{k+1}T^*\mathcal M)
\longrightarrow
L^2(\wedge^kT^*\mathcal M)$ with domain
\[
\mathcal D(\delta^k)
=
\{
\beta\in L^2(\wedge^{k+1}T^*\mathcal M):
\delta\beta\in L^2(\wedge^kT^*\mathcal M)
\text{ distributionally}
\}
\]
and let $\delta^{-1}=0$.

Let $d^k=U_kT_{\sigma_k}V_k^*$ denote the condensed singular value decomposition of $d^k$. By G\aa rding's inequality \cite[Theorem~1.34]{Rosenberg1997} and the Hodge theorem \cite[Theorem~1.30]{Rosenberg1997}, the de Rham Hilbert complex is closed. Moreover, by \cite[Corollary~1.36 and Theorem~1.45]{Rosenberg1997}, the $k$th de Rham cohomology is the quotient
\[
  H^k_{\dR}(\mathcal{M})
  \coloneqq
  (\mathcal N(d^k)\cap\Omega^k(\mathcal{M}))/d\Omega^{k-1}(\mathcal{M}) \cong \mathcal N(d^k)/\range(d^{k-1}).
\]

For a compact $\mathcal{M}$, $\dim H^k_{\dR}(\mathcal{M})<\infty$ is the $k$th Betti number. By Section~\ref{sec:complex}, it can be computed as
\[
  \tr(I-V_kV_k^*-U_{k-1}U_{k-1}^*)
  =
  \rank(I-V_kV_k^*-U_{k-1}U_{k-1}^*)
  =
  \dim H^k_{\dR}(\mathcal M).
\]

By Section~\ref{sec:complex}, for each $[\alpha]\in H^k_{\dR}(\mathcal{M})$, there exists a unique element
\[
  \widehat{\alpha} = \alpha-U_{k-1}U_{k-1}^*\alpha = \argmin_{\beta\in[\alpha]}
  \lVert \beta \rVert_{L^2(\wedge^kT^*\mathcal M)}
  \in
  \bigl(\mathcal N(d^k)\cap\mathcal N(\delta^{k-1})\bigr)\cap[\alpha] = \mathfrak{H}^k,
\]
which is the harmonic $k$-form of $\mathcal{M}$.

Moreover, the singular values of the exterior derivatives provide an alternative formula to compute the analytic torsion, which is another fundamental topological invariant of the manifold. For each $k=0,\dots,n$, let $\{\lambda_{k,j}\}_{j=1}^\infty$ be the positive eigenvalues of the Hodge--Laplacian $\Delta_k = \delta_k d_k + d_{k-1}\delta_{k-1}$. Consider the zeta function
\[
\zeta_k(z) \coloneqq \sum_{j=1}^\infty \lambda_{k,j}^{-z}.
\]
The zeta function is well-defined for
$\Re z$ sufficiently large and admits a meromorphic continuation to $\C$ that is holomorphic at $z=0$. For $z \in \C$ such that the right-hand side converges, 
\[
\frac{d}{dz}\zeta_k(z) = \sum_{j=1}^\infty -\lambda_{k,j}^{-z} \log \lambda_{k,j}.
\]
Thus, $\exp\bigl(-\zeta'_k(0)\bigr)$ assigns a finite value to the divergent product $\lambda_{k,1}\lambda_{k,2}\cdots$ via zeta regularization \cite{Hawking1977}. The analytic torsion is 
\[
T_{\mathcal M} \coloneqq \exp\Biggl(\sum_{k=1}^n\frac{(-1)^k k}{2}\zeta'_k(0)\Biggr).
\]

For $k = 0, \dots, n$, let $\{\sigma_{k,j}\}_{j=1}^\infty$ denote the singular values of $d^k$. By Hodge's theorem, the nonzero eigenvalues of the Hodge--Laplacian is exactly the multiset union $\{\sigma_{k,j}^2\}_{j=1}^\infty \cup \{\sigma_{k-1,j}^2\}_{j=1}^\infty$.

Therefore, for $z$ with sufficiently large real part,
\begin{equation}\label{eq:zeta}
  \zeta_k(z) = Z_k(z) + Z_{k-1}(z), \quad \text{ where } Z_k(z) \coloneqq \sum_{j=1}^\infty \sigma_{k,j}^{-2z}, \text{ for } k=0,\dots,n-1, \, Z_{-1} = Z_{n} \coloneqq 0.
\end{equation}
Since $Z_0=\zeta_0$, and $Z_k=\zeta_k-Z_{k-1}$  for $k=1,\dots,n-1$, each $Z_k$ admits a meromorphic continuation to $\C$ and is holomorphic at $z=0$. Hence \eqref{eq:zeta} continues to hold meromorphically, and we may differentiate it at $z=0$. Thus, 
\[
\log T_{\mathcal M} = \frac{1}{2} \sum_{k=0}^n (-1)^k k \left[ Z'_k(0) + Z'_{k-1}(0) \right].
\]
Expanding the telescoping sum, we obtain that 
\[
T_{\mathcal M} = \exp\Biggl(\sum_{k=0}^{n} \frac{(-1)^{k+1} }{2} Z'_k(0)\Biggr).
\]
In particular, the constant multiplier $k$ has been canceled via the telescoping sum and $\log T_{\mathcal M}$ reduces to an alternating series of $Z'_k(0)$.

\section{Singular value decomposition in numerical computations}\label{sec:eg-num}
In applied settings, one often encounters variants of differential operators arising from discretizations or two-sided actions. In this section, we discuss how singular value decomposition could provide insights into Petrov--Galerkin method, finite difference method, and two-sided actions on Hilbert--Schmidt operators.
\subsection{Petrov--Galerkin method}
The Petrov--Galerkin method is a method to approximate the solution of a PDE in weak formulation \cite[Section~2.4.3]{Reddy2004}. Given an unbounded operator $A: \mathcal{D}(A) \subseteq \mathcal{H} \to \mathcal{K}$ and $g \in \mathcal{K}$, the solution of the problem is $f \in \mathcal{D}(A)$ such that 
\[
  \langle Af, h\rangle = \langle g, h \rangle \quad \text{ for all } h \in \mathcal{K}. 
\]
The Petrov--Galerkin method proposes to solve a finite dimensional surrogate of this problem: Choose a $n$-dimensional subspace $\mathcal{H}_n \subseteq \mathcal{D}(A)$ and a $m$-dimensional subspace $\mathcal{K}_m \subseteq\mathcal K$ and find $f \in \mathcal{H}_n$ such that
\[
  \langle Af, h\rangle = \langle g, h \rangle \quad \text{ for all } h \in \mathcal{K}_m. 
\]
If we have $\{g_1,\dots,g_m\}$ a basis of $\mathcal{K}_m$ and $\{f_1,\dots, f_n\}$ a basis of $\mathcal{H}_n$, then the problem is simply to find the coefficients $c_k$, $k=1,\dots,n$ such that 
\[
\sum_{k=1}^n c_k \langle Af_k, g_l \rangle = \langle g, g_l\rangle, \quad l=1,\dots,m.
\]

As the Petrov--Galerkin method is mostly applied to differential operators with given boundary conditions, assume that $\mathsf{\Sigma}(A)$ is discrete. The singular value decomposition could be written as
\[
  Ax = \sum_{\sigma \in \mathsf{\Sigma}(A) \setminus \{0\}} \sum_{k=1}^{m_\sigma} \sigma \langle x, v_{\sigma,k} \rangle u_{\sigma,k}, \quad x \in \mathcal{D}(A).
\]
Letting $m=n$, the singular value decomposition provides the natural choice $\mathcal{H}_n = \operatorname{span}\{v_{\sigma_1,k_1},\dots, v_{\sigma_n,k_n}\}$ and $\mathcal{K}_n = \operatorname{span}\{u_{\sigma_1,k_1},\dots, u_{\sigma_n,k_n}\}$. In this case,
\[
\sum_{j=1}^n c_j \langle Av_{\sigma_j,k_j}, u_{\sigma_l,k_l} \rangle = c_l \sigma_l = \langle g, u_{\sigma_l,k_l}\rangle , \quad l=1,\dots,n,
\]
so $c_j = \langle g, u_{\sigma_j,k_j}\rangle/ \sigma_j$ for $j=1,\dots,n$.

\subsection{Finite difference operator}
Finite difference method is another classical technique to numerically solve differential equations by approximating the differential operators with finite difference operators. Based on Section~\ref{sec:grad-Rn}, we now show that the finite difference operator approximates the differential operator not only strongly, but also in the sense of the singular value decomposition. Let $0 \neq h \in \R^n$ be fixed. The finite difference operator is defined as the bounded operator 
\[
D_h \in \mathcal{B}(L^2(\R^n)), \quad D_h f(x) = \frac{f(x+h)-f(x)}{\lVert h \rVert}.
\]

Since
\begin{align*}
\int_{\R^n}
\frac{f(x+h)-f(x)}{\lVert h\rVert}
\overline{g(x)}\,dx
&=
\frac{1}{\lVert h\rVert}
\biggl(
\int_{\R^n}f(x+h)\overline{g(x)}\,dx
-
\int_{\R^n}f(x)\overline{g(x)}\,dx
\biggr) \\
&=
\frac{1}{\lVert h\rVert}
\biggl(
\int_{\R^n}f(x)\overline{g(x-h)}\,dx
-
\int_{\R^n}f(x)\overline{g(x)}\,dx
\biggr)
\end{align*}
its adjoint is the backward difference
$D_h^*g(x)=(g(x-h)-g(x))/\lVert h\rVert$ and
\[
D_h^*D_hf(x) = D_h^*\biggl(\frac{f(x+h)-f(x)}{\lVert h \rVert} \biggr) = \frac{2f(x) -f(x+h)-f(x-h)}{\lVert h \rVert^2}.
\]
Passing to the Fourier domain, we get the spectral decomposition
\begin{equation}\label{eq:fin-diff}
(\mathcal{F}D_h^*D_h\mathcal{F}^*g)(\xi)
=
\frac{2-e^{2\pi i\xi\cdot h}-e^{-2\pi i\xi\cdot h}}
{\lVert h\rVert^2}g(\xi)
=
\frac{4\sin^2(\pi\xi\cdot h)}
{\lVert h\rVert^2}g(\xi).
\end{equation}
Let $\sigma(\xi) = 2\lvert \sin(\pi\xi \cdot h)\rvert/\lVert h \rVert$. In particular, $\mathsf{\Sigma}(D_h) = [0, 2/\lVert h \rVert]$. By Theorem~\ref{thm:svd1}, the operator $U$ defined by
\[  Ug(x) = D_h \mathcal{F}^*T_{\sigma^{-1}}g(x) = \int_{\R^n} e^{2\pi i \xi \cdot x} \frac{(e^{2\pi i \xi \cdot h}-1)g(\xi)}{2\lvert \sin(\pi\xi \cdot h)\rvert} \, d\xi\]
is unitary, and we have obtained that $D_h = U T_\sigma \mathcal{F}$ is the singular value decomposition of the finite difference operator.

Furthermore, let $h >0$ now denote a step size; we may consider the finite-difference gradient defined by
\[
\nabla_hf = \begin{bmatrix}
    D_{he_1}f\\ \vdots \\ D_{he_n}f
\end{bmatrix}.
\]
By \eqref{eq:fin-diff}, we obtain the spectral decomposition of $\nabla_h^* \nabla_h$ as
\[
\mathcal{F} \nabla_h^* \nabla_h \mathcal{F}^* g(\xi) = \sum_{k=1}^n \mathcal{F} \Biggl( D_{h e_k}^* D_{h e_k} \mathcal{F}^* g \Biggr)(\xi) = \sum_{k=1}^n \frac{4\sin^2(\pi \xi_k h)}{h^2} g(\xi).
\]
Take 
\[
\sigma_h(\xi) = \frac{2}{\lvert h \rvert} \biggl( \sum_{k=1}^n \sin^2(\pi \xi_k h) \biggr)^{1/2}.
\]
By Theorem~\ref{thm:svd1}, the right singular operator of $\nabla_h$ is the Fourier transform $\mathcal{F}$, and the left singular operator is given component-wise by
\[
(U_h g)_k(x) = \int_{\R^n} e^{2\pi i x \cdot \xi} \frac{e^{2\pi i \xi_k h} - 1}{h \sigma_h(\xi)} g(\xi) \, d\xi.
\]

Notice that the finite-difference gradient and the continuous gradient both share the Fourier transform as their right singular operator. Furthermore, the singular value function of $\nabla_h$ also converges to $\nabla$ as
\[
\lim_{h \to 0} \sigma_h(\xi) = \lim_{h \to 0} \frac{2}{\lvert h \rvert} \biggl( \sum_{k=1}^n (\pi \xi_k h)^2 \biggr)^{1/2} = 2\pi \lvert \xi \rvert = \sigma(\xi).
\]
Lastly, since
\[
\lim_{h \to 0} \frac{e^{2\pi i \xi_k h} - 1}{h \sigma_h(\xi)} = \lim_{h \to 0} \frac{2\pi i \xi_k h + O(h^2)}{h (2\pi \lvert \xi \rvert)} = \frac{i\xi_k}{\lvert \xi \rvert},
\]
the left singular operator of $\nabla_h$ also converges strongly to the left singular operator of $\nabla$.

\subsection{Hilbert--Schmidt operators}
Let $\mathcal{H}_1$ and $\mathcal{H}_2$ be separable Hilbert spaces and let $\{e_k\}_{k\in I}$ be an orthonormal basis of $\mathcal{H}_1$. The space of Hilbert--Schmidt operators from $\mathcal{H}_1$ to $\mathcal{H}_2$ is
\[
\operatorname{HS}(\mathcal{H}_1,\mathcal{H}_2)
=\Bigl\{X\in\mathcal{L}(\mathcal{H}_1,\mathcal{H}_2):
\lVert X\rVert_{\operatorname{HS}}^2 \coloneqq
\sum_{k\in I} \lVert Xe_k\rVert_{\mathcal{H}_2}^2
<\infty \Bigr\}.
\]
It is a Hilbert space equipped with the inner product
\[
\langle X,Y\rangle_{\operatorname{HS}}= \sum_{k\in I}
\langle Xe_k,Ye_k\rangle_{\mathcal{H}_2}.
\]
Suppose $A,C:\mathcal{D}(A),\mathcal{D}(C) \subseteq \mathcal{H}_1 \to \mathcal{H}_2$ are closed, densely defined operators with discrete singular values. Write their condensed singular value decompositions respectively as $A=U_A\Sigma_AV_A^*$ and $C=U_C\Sigma_CV_C^*$, where
$U_A:\ell^2(I_A)\to\mathcal{H}_2$,
$V_A:\ell^2(I_A)\to\mathcal{H}_1$,
$U_C:\ell^2(I_C)\to\mathcal{H}_2$, and
$V_C:\ell^2(I_C)\to\mathcal{H}_1$
are isometries, and
\[
\Sigma_Ae_i=a_i e_i, \quad i\in I_A,
\quad \text{and} \quad
\Sigma_Ce_j=c_j e_j, \quad j\in I_C,
\]
with $a_i,c_j>0$. 

For $Y\in\operatorname{HS}(\ell^2(I_C),\ell^2(I_A))$, let $y_{ij}=\langle Ye_j,e_i\rangle$ for $i\in I_A$ and $j\in I_C$. Consider $\Sigma: \mathcal{D}(\Sigma) \to \operatorname{HS}(\ell^2(I_C),\ell^2(I_A))$ defined by
\[
\Sigma Y=\Sigma_AY\Sigma_C, \quad \text{for } Y \in
\mathcal{D}(\Sigma) \coloneqq
\Bigl\{ Y\in \operatorname{HS}(\ell^2(I_C),\ell^2(I_A)):
\sum_{i\in I_A} \sum_{j\in I_C} a_i^2c_j^2\lvert y_{ij}\rvert^2 <\infty \Bigr\}
\]
and isometries 
\begin{align*}
U:\operatorname{HS}(\ell^2(I_C),\ell^2(I_A))&\to
\operatorname{HS}(\mathcal{H}_2),
&Y\mapsto U_AYU_C^*, \\
V:\operatorname{HS}(\ell^2(I_C),\ell^2(I_A)) &\to
\operatorname{HS}(\mathcal{H}_1),
&Y\mapsto V_AYV_C^*.
\end{align*}

For all $X\in\operatorname{HS}(\mathcal{H}_1)$ such that $V_A^*XV_C\in\mathcal{D}(\Sigma)$ and
$h\in\mathcal{D}(C^*)$, 
\[
AXC^*h=U_A\Sigma_AV_A^*XV_C\Sigma_CU_C^*h=(U\Sigma V^*X)h.
\]
Since
$U\Sigma V^*X\in\operatorname{HS}(\mathcal{H}_2)$ and
$\mathcal{D}(C^*)$ is dense in $\mathcal{H}_2$, $U\Sigma V^*X$ is the unique Hilbert--Schmidt extension of $AXC^*$. Hence, $U\Sigma V^*$ is the condensed singular value decomposition of the two-sided operator $X \mapsto AXC^*$. In particular, its positive singular values are $a_ic_j$, $i\in I_A$, $j\in I_C$.

\section{Singular value decomposition in physics}\label{sec:eg-phy}
We now shift our focus to differential operators arising in physics. We begin with the textbook example of ladder operators for the quantum harmonic oscillator and then move to its generalizations in the context of supersymmetric quantum mechanics and Sturm--Liouville theory.

\subsection{Ladder operators}\label{sec:ladder}
The creation and annihilation operators are a pair of dual operators that play instrumental roles in the algebraic formulation of quantum mechanics. The annihilation operator $A$ is defined by 
\[
A: \{f \in H^1(\R): T_x f \in L^2(\R)\} \subseteq L^2(\R) \to L^2(\R), \quad Af(x) = \frac{1}{\sqrt{2}}\biggl(xf(x)+\frac{df}{dx}(x)\biggr).
\]
The creation operator is its dual 
\[
A^*: \mathcal{D}(A^*) = \mathcal{D}(A) \subseteq L^2(\R) \to L^2(\R), \quad A^*f(x) = \frac{1}{\sqrt{2}}\biggl(xf(x)-\frac{df}{dx}(x)\biggr).
\]
Since the Schwartz space is a subspace of $\mathcal{D}(A) = \mathcal{D}(A^*)$ and the Schwartz space is dense in $L^2(\R)$, $A$ and $A^*$ are densely defined. $A$ and $A^*$ are closed because they are adjoint operators to each other.
One of the fundamental applications of these ladder operators is that $A^*A$ is the Hamiltonian of the simple harmonic oscillator up to a constant shift. 

Consider the Hermite functions $\{v_k\}_{k=0}^\infty$, defined using the Hermite polynomials $H_k(x)$ as:
\begin{equation}\label{eq:hermite}
  H_k(x) = (-1)^k e^{x^2} \frac{d^k}{dx^k}\bigl(e^{-x^2}\bigr), \quad v_k(x) \coloneqq \frac{e^{-x^2/2}H_k(x)}{(\sqrt{\pi}2^k k!)^{1/2}}.
\end{equation}
By \cite[Theorem~11.4]{Hall2013}, these form an orthonormal basis of $L^2(\R)$. The nomenclature of ``ladder operator" comes from applying $A$ to the basis functions:
\begin{equation}\label{eq:creation}
Av_k = \sqrt{k}v_{k-1} \quad \text{for } k \ge 1, \quad \text{and } Av_0 = 0.
\end{equation}
In other words, by Lemma~\ref{lem:adjoint} and Corollary~\ref{cor:discrete}, we have obtained the singular value decompositions 
\begin{equation}\label{eq:svd-creation}
Af = \sum_{k=1}^\infty \sqrt{k}\langle f, v_k \rangle v_{k-1}, \quad A^*f = \sum_{k=0}^\infty \sqrt{k+1}\langle f, v_{k} \rangle v_{k+1}, \quad f \in \mathcal{D}(A) = \mathcal{D}(A^*).
\end{equation}
Written in multiplication operator form, this computation makes precise the convention of representing the operators by infinite matrices. Choosing $\{v_k\}_{k=0}^\infty$ as the basis, the annihilation and creation operators could be respectively represented by 
\[
A = \begin{pmatrix}
0 & \sqrt{1} & 0 & 0 & \dots \\
0 & 0 & \sqrt{2} & 0 & \dots \\
0 & 0 & 0 & \sqrt{3} & \dots \\
0 & 0 & 0 & 0 & \ddots \\
\vdots & \vdots & \vdots & \vdots & \ddots 
\end{pmatrix}
,\quad \text{and }
A^* = \begin{pmatrix}
0 & 0 & 0 & 0 & \dots \\
\sqrt{1} & 0 & 0 & 0 & \dots \\
0 & \sqrt{2} & 0 & 0 & \dots \\
0 & 0 & \sqrt{3} & 0 & \dots \\
\vdots & \vdots & \vdots & \ddots & \dots 
\end{pmatrix}.
\]
The singular value decomposition \eqref{eq:svd-creation} thus corresponds to the factorization 
\[
A = \begin{pmatrix}
1 & 0 & 0 & 0 & \dots \\
0 & 1 & 0 & 0 & \dots \\
0 & 0 & 1 & 0 & \dots \\
0 & 0 & 0 & 1 & \dots \\
\vdots & \vdots & \vdots & \vdots & \ddots
\end{pmatrix} \begin{pmatrix}
\sqrt{1} & 0 & 0 & 0 & \dots \\
0 & \sqrt{2} & 0 & 0 & \dots \\
0 & 0 & \sqrt{3} & 0 & \dots \\
0 & 0 & 0 & \sqrt{4} & \dots \\
\vdots & \vdots & \vdots & \vdots & \ddots
\end{pmatrix}\begin{pmatrix}
0 & 1 & 0 & 0 & \dots \\
0 & 0 & 1 & 0 & \dots \\
0 & 0 & 0 & 1 & \dots \\
0 & 0 & 0 & 0 & \dots \\
\vdots & \vdots & \vdots & \vdots & \ddots
\end{pmatrix}
\]
and analogously for $A^*$.

\subsection{Supersymmetric quantum mechanics}\label{sec:SUSY}

In fact, Section~\ref{sec:ladder} is a special case of a much more general phenomenon that arises in the study of supersymmetric quantum mechanics (SUSY QM) \cite{SUSY1,SUSY2}. In the Hamiltonian formulation of SUSY QM, a self-adjoint second-order Hamiltonian $\mathsf H$ is factorized, up to a constant shift, as $\mathsf H=A^*A$, an idea dating back to \cite{factorization}. The annihilation and creation operators give exactly this factorization for the simple harmonic oscillator. In general, the factorization reveals structural information about the energy levels and state spaces and produces a hierarchy of new Hamiltonians. We now show how the singular value decomposition of $A$ reveals the desired information of the SUSY Hamiltonian $\mathsf{H}$.

Let the Hamiltonian $\mathsf{H}_1: \mathcal{D}(\mathsf{H}_1)\subseteq L^2(\R) \to L^2(\R)$ be defined as the Friedrichs extension of the operator
\begin{equation}\label{eq:H1}
f \mapsto -\frac{d^2f}{dx^2}+ V_1f, \quad f \in C^\infty_c(\R),
\end{equation} where $V_1 \in C^\infty(\R)$ is the potential of the system such that $V_1(x) \to \infty$ as $\lvert x \rvert \to \infty$, which is densely defined and closed by construction. By \cite[Theorem~2.3.1]{BS1991}, $\mathsf{H}_1$ has a discrete spectrum $\mathsf{\Lambda}(\mathsf{H}_1) = \{\lambda_k(\mathsf{H}_1)\}_{k=0}^\infty$, arranged in increasing order, and each eigenspace is one-dimensional. For each eigenvalue $\lambda_k(\mathsf{H}_1)$, let $f_k$ denote, up to a sign, the normalized eigenfunction. Up to a constant shift, we assume that the ground state energy is zero, i.e., $\lambda_0(\mathsf{H}_1) = 0$. By \cite[Theorem~3.5]{BS1991}, $f_0$ has no zeros. One immediately obtains that the potential satisfies $V_1 = f_0''/f_0$. Consider
\[
A \coloneqq \frac{d}{dx} + T_W, \quad \text{ where } W \coloneqq -\frac{f_0'}{f_0}
\]
is the superpotential and the domain is $\mathcal{D}(A) = \{f \in L^2(\R): f' +Wf \in L^2(\R)\}$. This first-order operator is closed and densely defined. Integrating by parts gives that 
\[
A^* = -\frac{d}{dx} + T_W, \quad \mathcal{D}(A^*) =  \{f \in L^2(\R): -f' +Wf \in L^2(\R)\}.
\]

For $f \in C^\infty_c(\R)$,
\[A^*Af=
\left(-\frac{d}{dx}+T_W\right)
\left(\frac{d}{dx}+T_W\right)f=
-f''-W'f+W^2f=
-f''+\frac{f_0''}{f_0}f=
\mathsf H_1f.\]
Since the closed quadratic form \(f\mapsto\lVert Af\rVert_{L^2}^2\) is
the Friedrichs closure of the quadratic form of \eqref{eq:H1}, $\mathsf H_1=A^*A$.

A fundamental idea in SUSY QM is to construct a new Hamiltonian using $A$ and $A^*$:
\[
\mathsf{H}_2 = AA^* =  -\frac{d^2}{dx^2} + T_{V_2}, \quad \text{ where } V_2 = W^2 + W'
\]
is the supersymmetric partner potential. The following proposition summarizes some of the most fundamental observations in the Hamiltonian formulation of SUSY QM \cite[Section~2]{SUSY1}, which simplifies in the language of singular value decomposition.

\begin{proposition}\label{prop:SUSY}
Let $\mathsf{H}_1$, $\mathsf{H}_2$, $A$ be as defined above. Then 
\begin{enumerate}[\normalfont(i)]
    \item \label{itm:SUSY1} $A$ has countably many singular values $\mathsf{\Sigma}(A)=\{0\}\cup \{\sigma_k = \sqrt{\lambda_k(\mathsf{H}_1)}\}_{k=1}^\infty$;
    \item \label{itm:SUSY2}$\mathsf{\Lambda}(\mathsf{H}_2)  = \{\lambda_{k}(\mathsf{H}_2) = \sigma_{k+1}^2\}_{k=0}^\infty$;
    \item \label{itm:SUSY3} For $k=0,1,\dots$, the eigenfunctions of $\lambda_k(\mathsf{H}_2)$ are constant multiples of $g_k = \sigma_{k+1}^{-1} Af_{k+1}$;
    \item \label{itm:SUSY4} For $k=0,1,\dots$, $f_{k+1} = \sigma_{k+1}^{-1} A^*g_k$.
\end{enumerate}
\end{proposition}

\begin{proof}
Without loss of generality, we assume $\lVert f_0\rVert_{L^2(\R)}=1$. By Corollary~\ref{cor:discrete}, \ref{itm:SUSY1} follows as
\[
A = \sum_{k=0}^\infty \sigma_{k+1} \langle \,\cdot, f_{k+1}\rangle g_k, \quad \sigma_k = \sqrt{\lambda_k(\mathsf{H}_1)} \text{ and } g_k = \sigma_{k+1}^{-1} Af_{k+1}.
\]
Suppose there exists $g \in L^2(\R)$ such that $A^*g = 0$.
Then by the product rule,
\[
g' = Wg = -\frac{f_0'}{f_0} g
\Longleftrightarrow
(gf_0)' = 0.
\]
Thus, $g = C/f_0$ for some constant $C$. If $C\neq0$, then for every $R>0$, Cauchy--Schwarz inequality gives
\[
(2R)^2= \biggl( \int_{-R}^R \lvert f_0(x)\rvert \lvert f_0(x)\rvert^{-1} \,dx\biggr)^2
\le \biggl(\int_{-R}^R \lvert f_0(x)\rvert^2 \,dx\biggr) \biggl(\int_{-R}^R \lvert f_0(x)\rvert^{-2} \,dx\biggr)
\le \int_{-R}^R \lvert f_0(x)\rvert^{-2}\,dx.
\]
Letting $R\to\infty$ gives that
$1/f_0\notin L^2(\R)$, and hence $g\notin L^2(\R)$.
Therefore, $C=0$ and $\mathcal{N}(A^*)=\{0\}$. Since $A$ is closed, this implies that $\overline{\range(A)} = L^2(\R)$, and $\{g_k\}_{k=0}^\infty$ is an orthonormal basis of $L^2(\R)$.
Consequently, 
\[
\mathsf{H}_2 = AA^* = \sum_{k=0}^\infty \sigma_{k+1}^2 \langle \,\cdot, g_k\rangle g_k = \sum_{k=0}^\infty \lambda_{k+1}(\mathsf{H}_1) \langle \,\cdot, g_k\rangle g_k.
\]
is the spectral decomposition of $\mathsf{H}_2$, from which \ref{itm:SUSY2} and \ref{itm:SUSY3} follow. 

Lastly, by Lemma~\ref{lem:adjoint}, 
\[
    A^*g_j = \sum_{k=1}^\infty \sigma_k \langle g_j, g_{k-1}\rangle f_k = \sigma_{j+1}f_{j+1}. \qedhere
\]
\end{proof}

Therefore, the singular value decomposition of $A$ contains all the desired information of the augmented SUSY Hamiltonian
\[
\begin{pmatrix}
    \mathsf{H}_1 &0 \\
    0 &\mathsf{H}_2
\end{pmatrix} = \begin{pmatrix}
    0 &A^*\\ 0 &0
\end{pmatrix}
\begin{pmatrix}
    0 &0\\ A &0
\end{pmatrix}
+
\begin{pmatrix}
    0 &0\\ A &0
\end{pmatrix}
\begin{pmatrix}
    0 &A^*\\ 0 &0
\end{pmatrix} = \begin{pmatrix}
    0 &A^*\\ A &0
\end{pmatrix}^2.
\]
In particular, the left and right singular vectors of $A$ capture exactly the bosonic and fermionic energy states of the system. 

\subsection{Sturm--Liouville theory}\label{sec:Sturm--Liouville}

The Sturm--Liouville problem is a second-order ODE closely related to separable PDEs and quantum mechanics \cite{AHP2005}. Let $(a,b) \subseteq \R$. Consider the differential operator
\begin{equation}\label{eq:Sturm--Liouville}
  f \mapsto -\frac{1}{w}\frac{d}{dx}\biggl(p\frac{df}{dx}\biggr) + \frac{q}{w}f, \quad  \text{where } f\in C^\infty([a,b]) \; \text{satisfies } \begin{cases*} c_1 f(a) + c_2 f'(a) = 0, \\ c_3 f(b) + c_4 f'(b) = 0,
  \end{cases*}
\end{equation}
where $0 < p, w \in C^2([a,b])$, $q \in C([a,b];\R)$, $c_1,c_2 \in \R$ are not both zero, and $c_3,c_4 \in \R$ are not both zero. By \cite[Section~4]{Everitt2005}, \eqref{eq:Sturm--Liouville} extends to a densely defined closed self-adjoint operator $L: \mathcal{D}(L) \subseteq L^2([a,b],w(x)dx) \to L^2([a,b],w(x)dx)$. The Sturm--Liouville problem is the eigenvalue problem of this operator. 

$L$ has a discrete spectrum $\mathsf{\Lambda}(L) = \{\lambda_k(L)\}_{k=0}^\infty$, arranged in increasing order, and each eigenspace is one-dimensional \cite[Theorem~D]{Hinton2005}. For each eigenvalue $\lambda_k(L)$, let $f_k$ denote, up to a sign, the normalized eigenfunction. Up to a constant shift, we may again assume that $\lambda_0(L) = 0$. By \cite[Theorem~D]{Hinton2005} again, $f_0$ has no zeros in $(a,b)$. The apparent similarity to Section~\ref{sec:SUSY} is not coincidental: the bridge is the Liouville transform. Let $x_0 \in [a,b]$ and consider the change of variable 
\[
t(x) = x_0 + \int_{x_0}^x \sqrt{\frac{w(y)}{p(y)}} \, dy.
\]
Since $w$ and $p$ are strictly positive, the change of variable is a bijection between $(a,b)$ and $(\alpha,\beta)$ where
\[
\alpha = \lim_{x \to a^+} t(x), \quad \beta = \lim_{x \to b^-} t(x).
\]
The Liouville transform is the operator 
\[
U: L^2([a,b],w(x)dx) \to L^2([\alpha,\beta],dt), \quad Uf(t) = r(x(t))f(x(t)), \quad \text{ where } r(x) = \bigl(p(x)w(x)\bigr)^{1/4}.
\]
Since
\[
\lVert Uf \rVert_{L^2(dt)}^2 = \int_{\alpha}^\beta \bigl\lvert p(x(t))w(x(t))\bigr\rvert^{1/2} \lvert f(x(t)) \rvert^2 \,dt = \int_a^b \bigl\lvert p(x)w(x)\bigr\rvert^{1/2} \lvert f(x) \rvert^2 \sqrt{\frac{w(x)}{p(x)}}\,dx = \lVert f\rVert^2_{L^2(w\,dx)},
\]
$U$ is an isometry. Moreover, it is unitary as its inverse is
\[
U^{-1}u(x) = U^*u(x) =  \frac{1}{r(x)}u(t(x)) , \quad u(t) \in L^2((\alpha,\beta),dt).
\]
A direct calculation verifies that for the densely defined self-adjoint operator $\mathsf{H} \coloneqq ULU^*$,
\[
  \mathsf{H} = -\frac{d^2}{dt^2}+V(t) , \quad \text{ where }
V(t)=\frac{q(x(t))}{w(x(t))}+\frac{1}{r}\frac{d^2 r}{dt^2}.
\]
In particular, $\mathsf{H}$ has the same spectrum as $L$ and $\mathsf{H}Uf_0 = 0$ and $Uf_0$ has no zeros on $(\alpha,\beta)$. By the same argument as in Section~\ref{sec:SUSY}, 
\[
\mathsf{H} = \biggl(-\frac{d}{dt} + T_W\biggr)\biggl(\frac{d}{dt} + T_W\biggr), \enspace \text{ where } W(t) = -\sqrt{\frac{p(x)}{w(x)}} \left( \frac{f_0'(x)}{f_0(x)} + \frac{1}{4}\frac{p'(x)w(x) + p(x)w'(x)}{p(x)w(x)} \right) \bigg|_{x=x(t)}.
\]
Conjugating by the Liouville transform yields that 
\begin{equation}\label{eq:Sturm--Liouville2}
L = B^*B, \quad B \coloneqq U^* \Bigl(\frac{d}{dt}+T_W \Bigr) U =  \sqrt{\frac{p}{w}}\biggl(\frac{d}{dx}- \frac{f_0'}{f_0} \biggr).
\end{equation}
We take the maximal domain such that $B$ is a closed, densely defined operator. By \cite[Section~5]{Everitt2005}, the Sturm--Liouville operators are well-defined if the closed interval is replaced by a half-line. However, in this case, \cite[Theorem~D]{Hinton2005} no longer holds and $L$ might have a continuous spectrum. Nevertheless, if $L$ has a lowest eigenvalue $0$, by \cite[Theorem~3.5]{Simon2005}, it has an eigenfunction $f_0$ with no zeros. In that case, since the Liouville transform is well-defined if $[a,b]$ is replaced by a half-line, \eqref{eq:Sturm--Liouville2} holds for $L$ defined on $L^2((0,\infty),w(x)\,dx)$. Many important special functions arise as singular vectors of the operator $B$. 

\begin{example}[Laguerre polynomials]
Let $\alpha > -1$ and consider the domain $(0, \infty)$ with $p(x) = x^{\alpha+1} e^{-x}$, $w(x) = x^\alpha e^{-x}$, and $q(x) = 0$. The Sturm--Liouville operator is defined by extending
\[
  Lf = -\frac{1}{x^\alpha e^{-x}}\frac{d}{dx}\biggl(x^{\alpha+1} e^{-x}\frac{df}{dx}\biggr) = -x\frac{d^2f}{dx^2} - (\alpha + 1 - x)\frac{df}{dx}.
\]
Trivially, $f_0 = 1$ is an eigenfunction for $0$. Therefore, on $L^2((0, \infty), x^\alpha e^{-x}dx)$,
\[
  B = \sqrt{\frac{p(x)}{w(x)}}\biggl(\frac{d}{dx} - \frac{f_0'(x)}{f_0(x)} \biggr) = \sqrt{x}\frac{d}{dx}.
\]
Note that the expression for $B$ is independent of $\alpha$. By \cite[Section~27]{Everitt2005}, the normalized generalized Laguerre polynomials 
\[
  \sqrt{\frac{k!}{\Gamma(k+\alpha+1)}} L_k^{(\alpha)}(x) = \sqrt{\frac{k!}{\Gamma(k+\alpha+1)}} \frac{x^{-\alpha}e^x}{k!}\frac{d^k}{dx^k}(e^{-x}x^{k+\alpha}), \quad k = 0, 1, 2, \dots,
\]
form an orthonormal basis in $L^2((0, \infty), x^\alpha e^{-x}dx)$ and are the eigenfunctions of $L$ corresponding to the eigenvalues $\lambda_k = k$. Consequently, they are the right singular vectors of $B$ with singular values $\sigma_k = \sqrt{k}$. 

Since $L_k^{(\alpha)}{'}(x) = -L_{k-1}^{(\alpha+1)}(x)$, for $k=1,2,\dots$, the corresponding left singular vector is
\[
  \frac{1}{\sigma_k} B \sqrt{\frac{k!}{\Gamma(k+\alpha+1)}} L_k^{(\alpha)}(x) = -\sqrt{\frac{(k-1)! x}{\Gamma(k+\alpha+1)}} L_{k-1}^{(\alpha+1)}(x).
\]
\end{example}

\begin{example}[Jacobi polynomials]
Let $\alpha, \beta > -1$ and consider the domain $(-1, 1)$ with $p(x) = (1-x)^{\alpha+1}(1+x)^{\beta+1}$, $w(x) = (1-x)^\alpha (1+x)^\beta$, and $q(x) = 0$. 
The Sturm--Liouville operator is defined by extending
\[
  Lf = -\frac{1}{(1-x)^\alpha (1+x)^\beta}\frac{d}{dx}\biggl((1-x)^{\alpha+1}(1+x)^{\beta+1}\frac{df}{dx}\biggr).
\]
Trivially, $f_0(x) = 1$ is an eigenfunction for $0$. Therefore, on $L^2((-1, 1), (1-x)^\alpha (1+x)^\beta dx)$,
\[
  B = \sqrt{\frac{p(x)}{w(x)}}\biggl(\frac{d}{dx} - \frac{f_0'(x)}{f_0(x)} \biggr) = \sqrt{1-x^2}\frac{d}{dx}.
\]
Note that the expression for $B$ is independent of both $\alpha$ and $\beta$. Recall that the Jacobi polynomials 
\[
 P_k^{(\alpha, \beta)}(x) = \frac{(-1)^k}{2^k k!} (1-x)^{-\alpha} (1+x)^{-\beta} \frac{d^k}{dx^k} \Bigl[ (1-x)^{\alpha+k} (1+x)^{\beta+k} \Bigr]
\]
form an orthogonal basis in $L^2((-1, 1), (1-x)^\alpha (1+x)^\beta dx)$. We may normalize them by $h_k(x) \coloneqq \bigl( N_k^{(\alpha,\beta)} \bigr)^{-1/2} P_k^{(\alpha, \beta)}(x)$ where
\[
  N_k^{(\alpha, \beta)} = \int_{-1}^1 \bigl( P_k^{(\alpha, \beta)}(x) \bigr)^2 (1-x)^\alpha (1+x)^\beta \,dx = \frac{2^{\alpha+\beta+1}}{2k+\alpha+\beta+1} \frac{\Gamma(k+\alpha+1)\Gamma(k+\beta+1)}{k!\Gamma(k+\alpha+\beta+1)}.
\]
In fact, they are the eigenvalues of $L$ corresponding to the eigenvalues $\lambda_k = k(k+\alpha+\beta+1)$ \cite[Section~23]{Everitt2005}. Consequently, they are the right singular vectors of $B$ with singular values $\sigma_k = \sqrt{k(k+\alpha+\beta+1)}$.

Since 
\[
  \frac{d}{dx} P_k^{(\alpha, \beta)}(x) = \frac{1}{2}(k+\alpha+\beta+1) P_{k-1}^{(\alpha+1, \beta+1)}(x),
\] the left singular vectors are
\begin{align*}
  g_k(x) &= \frac{1}{\sigma_k} B h_k(x) = \frac{1}{\sqrt{k(k+\alpha+\beta+1)}} \sqrt{1-x^2} \bigl( N_k^{(\alpha, \beta)} \bigr)^{-1/2} \frac{1}{2}(k+\alpha+\beta+1) P_{k-1}^{(\alpha+1, \beta+1)}(x) \\
  &= \frac{1}{2} \sqrt{\frac{k+\alpha+\beta+1}{k}} \bigl( N_k^{(\alpha, \beta)} \bigr)^{-1/2} \sqrt{1-x^2} P_{k-1}^{(\alpha+1, \beta+1)}(x).
\end{align*}

In particular, if $\alpha = \beta = 0$, the Jacobi polynomials reduce to the Legendre polynomials.
\end{example}

\section{Singular value decomposition in integral operator and density estimation}
We now turn to an integral operator on $(0,\infty)$ whose right and left singular operators are the sine and cosine transforms. We then use this decomposition to construct a nonparametric density estimator and analyze the worst-case risk.

\subsection{Integral operator}\label{sec:integral}
Consider the integral operator $A$ over $L^2(0,\infty)$ with kernel $K(x,y)=\1_{\{y\ge x\}}$. It is straightforward to verify that it is the unbounded operator
\[
(Af)(x) \coloneqq \int_x^\infty f(y) \,dy, \quad \mathcal{D}(A) = \Big\{f\in L^2(0,\infty): x\mapsto \int_x^\infty f(y)\,dy\in L^2(0,\infty)\Big\}.
\]
Since $C_c(0,\infty)\subseteq \mathcal{D}(A)$, the domain $\mathcal{D}(A)$ is dense in $L^2(0,\infty)$. Applying Fubini's theorem, the adjoint of $A$ is given by
\[
(A^*g)(y)=\int_0^y g(x)\,dx, \quad \mathcal{D}(A^*)=\Bigl\{g\in L^2(0,\infty): y\mapsto \int_0^y g(x)\,dx \in L^2(0,\infty)\Bigr\}.
\]
Consequently, for $f\in\mathcal{D}(A^*A)$,
\[
(A^*Af)(x)
=
\int_0^x (Af)(t)\,dt
=
\int_0^x\int_t^\infty f(y)\,dy\,dt
=
\int_0^\infty \min\{x,y\}f(y)\,dy,
\]
i.e., $A^*A$ is the integral operator with kernel $\min\{x,y\}$.

Recall the sine transform on $L^2(0,\infty)$ is the unitary and self-adjoint operator \cite[Theorem~52]{Titchmarsh1948} defined by
\[
(\mathcal{S}f)(z)\coloneqq\sqrt{\frac{2}{\pi}}\int_0^\infty \sin(zy)f(y)\,dy.
\]
Consider the set $\mathcal{S}^{-1}\big(C_c^\infty(0,\infty)\big) = \mathcal{S}\big(C_c^\infty(0,\infty)\big)$, which is a core for the self-adjoint operator $\mathcal{S}T_{1/z^2}\mathcal{S}$ via a standard argument.  For
$f \in \mathcal{S}\big(C_c^\infty(0,\infty)\big)$,
\[
	\mathcal{S}T_{1/z^2}\mathcal{S}f(x) =\int_0^\infty\sin(zx)\frac{1}{z^2}\frac{2}{\pi} \Big[ \int_{0}^\infty \sin(zy)f(y)\,d y\Big] d z=\int_0^\infty K(x,y) f(y)\,d y
\]
where \cite[Equation~3.741.3]{Gradshteyn2007}
\[
K(x,y)
=
\frac{2}{\pi}\int_0^\infty \frac{\sin(zx)\sin(zy)}{z^2}\,dz=\min\{x,y\}.
\]
The integrals here converge absolutely since $f$ is a Schwartz function. Since $f \in \mathcal{S}^{-1}\big(C_c^\infty(0,\infty)\big)$,
\[
\sqrt{\pi/2}\,(\mathcal{S}f)'(0) = \sqrt{\frac{2}{\pi}}\int_0^\infty y\cos(zy)f(y)\,dy \bigg\vert_{z=0} = \int_0^\infty y f(y)\,dy = 0.
\]
In particular, $y \mapsto \int_0^y (Af)(t)\,dt$ vanishes at infinity and lies in $L^2(0,\infty)$, so $\mathcal{S}\big(C_c^\infty(0,\infty)\big) \subseteq \mathcal{D}(A^*A)$. As $A^*A$ and $\mathcal{S}T_{1/z^2}\mathcal{S}$ are self-adjoint and agree on the core $\mathcal{S}\big(C_c^\infty(0,\infty)\big)$, $A^*A = \mathcal{S} T_{1/z^2}\mathcal{S}$.  

We now proceed to find the singular value decomposition of $A$. Since $\mathcal{N}(A) = \{0\}$, Theorem~\ref{thm:svd1} gives that $A = U T_{1/z} \mathcal{S}$ for some unitary operator $U$, which we claim equals to the cosine transform defined by 
\[
(\mathcal{C}f)(x) \coloneqq \sqrt{\frac{2}{\pi}}\int_0^\infty \cos(xz)f(z)\,dz.
\]
It suffices to check the $\mathcal{C}f = A\mathcal{S}T_zf$ for all $f \in C_c(0,\infty)$. Notice that 
\begin{align*}
	\frac{d}{dx}(A \mathcal{S} T_z f)(x) &=-(\mathcal{S}T_zf)(x)
=
-\sqrt{\frac{2}{\pi}}\int_0^\infty \sin(xz)\,z f(z)\,dz 
=
\sqrt{\frac{2}{\pi}}\int_0^\infty \frac{d}{dx}\cos(xz)\,f(z)\,dz \\&
=
\frac{d}{dx}\Big(\sqrt{\frac{2}{\pi}}\int_0^\infty \cos(xz)f(z)\,dz\Big) =\frac{d}{dx}(\mathcal{C}f)(x).
\end{align*}

Therefore, integrating back with respect to $x$, we obtain $(Uf)(x)=(\mathcal{C}f)(x)+c$ for some constant $c$. Since 
$c = Uf - \mathcal{C}f \in L^2(0,\infty)$, $c= 0$. The singular value decomposition of $A$ is hence $A = \mathcal{C}T_{1/z}\mathcal{S}$. 

\subsection{Nonparametric density estimation}\label{sec:density}
Let $X_1, \dots, X_n$ be independent and identically distributed nonnegative random variables with density $f\in L^2(0,\infty)$. Our goal in this section is to use the singular value decomposition computed in Section~\ref{sec:integral} to construct an estimator of $f$ and establish matching bounds.

Recall that the survival function of $f$ is defined by 
\[
G(x)\coloneqq \mathbb{P}(X>x) = \int_x^\infty f(y)\,dy = Af(x)\quad \text{for } x \ge 0.
\]

In practice, we do not have access to $G(x)$ but instead the empirical survival function  
\[
G_n(x)\coloneqq \frac{1}{n} \sum_{k=1}^n \mathbbm{1}_{(x,\infty)}(X_k) \quad \text{for } x \geq 0.
\]
By Section~\ref{sec:integral} and Proposition~\ref{prs:mp-inv}, the Moore--Penrose inverse of $A$ is the unbounded operator $A^\dagger = \mathcal{S} T_{z}\mathcal{C}$. Notice that in general $G_n \notin \mathcal{D}(A^\dagger)$, so we cannot directly define an estimator as $A^\dagger G_n$. To compensate for the issue, we introduce the cutoff estimator 
\[
\widehat f_\alpha\coloneqq \mathcal{S} T_{z\1_{[0,\alpha]}} \mathcal{C} G_n.
\]
in the spirit of Proposition~\ref{prop:approx-bounded}. 

\begin{proposition}
    Let $X_1,\dots,X_n$ be iid random variables with density $f$ and $\E X_1<\infty$, then there is a  sequence $(\alpha_n)_{n\ge 0}$ such that $\alpha_n\to \infty$ and
    \[
    \mathbb{E} \lVert \widehat f_{\alpha_n} - f \rVert^2  \to 0.
    \]
\end{proposition}

\begin{proof}
Notice that
\[
\E\lVert \widehat f_\alpha -f \rVert_2^2  = \E\lVert A^\dagger_\alpha(G_n-G) +(A^\dagger_\alpha G-f)\rVert_2^2 \le 2\lVert A^\dagger_\alpha\rVert^2  \E\lVert G_n-G\rVert_2^2 + 2\lVert A^\dagger_\alpha G-f\rVert_2^2.
\]
For the first term, $\lVert A^\dagger_\alpha\rVert = \lVert \mathcal{S}T_{z\1_{[0,\alpha]}}\mathcal{C}\rVert= \lVert T_{z\1_{[0,\alpha]}}\rVert=\alpha$. Then by Tonelli's theorem, 
\begin{align*}
	\lVert A^\dagger_\alpha\rVert^2\E\lVert G_n-G\rVert_2^2 & =\alpha^2 \int_0^\infty \E \lvert G_n(t)-G(t)\rvert^2\,dt = \alpha^2\int_0^\infty \Var(G_n(t))\,dt \\&= \frac{\alpha^2}{n} \int_0^\infty G(t)(1-G(t))\,dt \le \frac{\alpha^2}{n} \int_0^\infty G(t)\,dt = \frac{\alpha^2}{n}\E X_1.
\end{align*}
Hence
\[
\E\lVert\widehat f_\alpha -f\rVert_2^2 \le \frac{2\alpha^2}{n}\E X_1+ 2\lVert A^\dagger_\alpha G-f\rVert_2^2.
\]
Select $\alpha_n$ such that $\alpha_n\to \infty$ and $\alpha_n^2/n\to 0$, then Proposition~\ref{prop:approx-bounded} yields that $\E\lVert\widehat f_{\alpha_n} -f\rVert_2^2 \to 0$.
\end{proof}

We now proceed to derive the closed-form formula for $\widehat{f}_\alpha$. Since
\[
 (\mathcal{C}G_n)(z)
= \sqrt{\frac{2}{\pi}} \int_0^\infty G_n(y)\cos(zy)\,dy
= \sqrt{\frac{2}{\pi}}\frac1n \sum_{k=1}^n \int_0^{X_k}\cos(zy)\,dy
= \sqrt{\frac{2}{\pi}}\frac1n \sum_{k=1}^n \frac{\sin(zX_k)}{z},
\]
we obtain
\[
\widehat f_\alpha(y)
= \sqrt{\frac{2}{\pi}} \int_0^\alpha \sin(yz) z (\mathcal{C}G_n)(z)\,dz
= \frac{2}{\pi n}\sum_{k=1}^n \int_0^\alpha \sin(yz)\sin(zX_k)\,dz .
\]
Applying the identity $2\sin(yz)\sin(zX_k)=\cos\bigl(z(y-X_k)\bigr)-\cos\bigl(z(y+X_k)\bigr)$, we obtain
\[
\widehat f_\alpha(y)
= \frac{1}{\pi n}\sum_{k=1}^n
\Biggl(
\frac{\sin\bigl(\alpha(y-X_k)\bigr)}{y-X_k}
-
\frac{\sin\bigl(\alpha(y+X_k)\bigr)}{y+X_k}
\Biggr) = \frac{\alpha}{\pi n}\sum_{k=1}^n \Bigl(
\operatorname{sinc}\bigl(\alpha(y-X_k)\bigr)
-
\operatorname{sinc}\bigl(\alpha(y+X_k)\bigr) \Bigr)
.
\]
This formula shows that $\widehat f_\alpha$ is a boundary-corrected sinc-kernel estimator induced by the singular value decomposition of the integral operator $A$. The cutoff parameter $\alpha$ plays the role of a regularization parameter: Large values of $\alpha$ reduce the bias introduced by truncation, while small values suppress the instability caused by the unbounded inverse $A^\dagger$.

The risk of the estimator $\widehat{f}_\alpha$ is $R_n(\alpha,f)\coloneqq \E_f \lVert \widehat f_\alpha-f\rVert^2_{L^2(0,\infty)}$. Fix $\beta>0$. Consider the Sobolev space
\begin{equation}
    \label{eq:sobolev-space}
    \mathcal{H}(\beta) \coloneqq \Big\{ f\in L^2(0,\infty): \lVert f\rVert^2_{\mathcal{H(\beta)}} < \infty\Big\}, \quad \lVert f\rVert^2_{\mathcal{H(\beta)}}\coloneqq \int_0^\infty (1+z^2)^\beta  |(\mathcal{S}f)(z)|^2\,dz.
\end{equation}
Due to the structure of singular value decomposition, $\mathcal{H}(\beta)$ differs from the textbook Sobolev ball in a single respect: Smoothness is measured through the
right singular operator $\mathcal{S}$ rather than the Fourier transform \cite[p.~25]{Tsybakov2009}. It encodes smoothness as the decay of the coefficients of $f$ right singular coordinates of $A$, which is one of the standard approaches of defining smoothness classes in statistical inverse problems \cite[Section~2.2]{Cavalier2008}. 

In the literature of nonparametric estimation, the risk of a class of estimators is studied via the minimax bound over the probability densities in Sobolev ball \cite{Devroye1983},  \cite[p.~25 and Definition~1.12]{Tsybakov2009}
\[
\mathcal{B}(\beta,\gamma) \coloneqq \Big\{ f\in \mathcal{H}(\beta) : \lVert f\rVert^2_{\mathcal{H(\beta)}} \leq \gamma, f\ge 0,\, \int_0^\infty f=1\Big\}.
\]
In the rest of the section, we show that for every fixed $\beta>0$, there exists $\gamma^*>0$ such that
\[
\inf_{\alpha>0}\,\sup_{f\in\mathcal{B}(\beta,\gamma^*)}
R_n(\alpha,f)\asymp n^{-2\beta/(2\beta+1)}
\qquad\text{as }n\to\infty.
\]
Thus, the optimal worst-case risk  within $\{\widehat f_\alpha:\alpha>0\}$ has rate $n^{-2\beta/(2\beta+1)}$, matching the classical rate \cite{Stone1980}, \cite[Chapters~1--2]{Tsybakov2009}.

\begin{theorem}\label{thm:density-minimax-upperbound}
Fix $\beta>0$ and $\gamma>0$. Then
\[
\inf_{\alpha>0}\,\sup_{f\in\mathcal{B}(\beta,\gamma)}
R_n(\alpha,f)\lesssim n^{-2\beta/(2\beta+1)}
\qquad\text{as }n\to\infty.
\]
\end{theorem}
\begin{proof}
Define $\theta(z)\coloneqq (\mathcal{S}f)(z)$ and
\[
\widehat \theta_n(z)\coloneqq z(\mathcal{C}G_n)(z) = \sqrt{\frac{2}{\pi}}\frac{1}{n} \sum_{k=1}^n\sin(zX_k).
\]
In particular, $\widehat f_\alpha=A_\alpha^\dagger G_n=\mathcal{S}T_{z\1_{[0,\alpha]}} \mathcal{C}G_n=\mathcal{S}(\1_{[0,\alpha]}\widehat \theta_n)$. Note that $\widehat\theta_n$ is an unbiased estimator of $\theta$ since
\[
\E_f \widehat\theta_n(z) = \E_f\biggl[\sqrt{\frac{2}{\pi}}\frac{1}{n} \sum_{k=1}^n\sin(zX_k) \biggr] = \sqrt{\frac{2}{\pi}} \E _f \sin(zX_1) = \sqrt{\frac{2}{\pi}} \int_0^\infty  f(x)\sin(zx)\,dx = (\mathcal{S}f)(z) = \theta(z).
\]
Therefore,
\[
\E_f \lvert \widehat\theta_n(z)-\theta(z) \rvert^2 = \Var_f \widehat\theta_n(z) = \Var_f \Big( \sqrt{\frac{2}{\pi}}\frac{1}{n} \sum_{k=1}^n\sin(zX_k)\Big) = \frac{2}{\pi n}\Var_f \big(\sin(zX_1)\big).
\]
Since $\mathcal{S}$ is unitary,
\begin{align*}
	\lVert \widehat f_\alpha-f\rVert^2_{L^2(0,\infty)} &=  \lVert  \mathcal{S}(\1_{[0,\alpha]} \widehat\theta_n)-\mathcal{S}\theta\rVert^2_{L^2(0,\infty)} = \lVert  \1_{[0,\alpha]} \widehat\theta_n-\theta\rVert^2_{L^2(0,\infty)} \\
    &= \lVert  \1_{[0,\alpha]} (\widehat\theta_n-\theta)-\1_{(\alpha,\infty)} \theta\rVert^2_{L^2(0,\infty)} = \int_0^\alpha |\widehat\theta_n(z)-\theta(z)|^2\,dz + \int_\alpha^\infty |\theta(z)|^2\,dz,
\end{align*}
we obtain via Tonelli's theorem that
\begin{equation}\label{eq:density-decomp}
\begin{aligned}
    	 R_n(\alpha,f) &=\E_f  \lVert \widehat f_\alpha-f\rVert^2_{L^2(0,\infty)} = \int_0^\alpha \E_f|\widehat\theta_n(z)-\theta(z)|^2\,dz + \int_\alpha^\infty |\theta(z)|^2\,dz\\
     &=\int_\alpha^\infty |(\mathcal{S}f)(z)|^2\,dz + \frac{2}{\pi n} \int_0^\alpha \Var_f \big(\sin(zX_1)\big)\,dz \leq \int_\alpha^\infty |(\mathcal{S}f)(z)|^2 \,dz + \frac{2\alpha}{\pi n},
\end{aligned}
\end{equation}
where the last inequality follows from $\Var_f\big(\sin(zX_1)\big) \le \E_f\big[\sin^2(zX_1)\big] \le 1$.
For any $f\in \mathcal{B}(\beta,\gamma)$, we have
\begin{align*}
	\int_\alpha^\infty |(\mathcal{S}f)(z)|^2 \,dz &= \int_\alpha^\infty (1+z^2)^{-\beta} (1+z^2)^\beta|(\mathcal{S}f)(z)|^2 \,dz  \\&\le (1+\alpha^2)^{-\beta}  \int_\alpha^\infty  (1+z^2)^\beta|(\mathcal{S}f)(z)|^2 \,dz \le (1+\alpha^2)^{-\beta} \gamma. 
\end{align*}
Therefore, 
\[
\sup_{f\in\mathcal{B}(\beta,\gamma)} R_n(\alpha,f) \le  (1+\alpha^2)^{-\beta} \gamma + \frac{2\alpha}{\pi n}.
\]
Taking $\alpha_n\asymp n^{1/(2\beta+1)}$, we have
\[
\inf_{\alpha>0}\,\sup_{f\in\mathcal{B}(\beta,\gamma)} R_n(\alpha,f)  \lesssim n^{-\frac{2\beta}{2\beta+1}}.\qedhere
\]
\end{proof}

We now proceed to show the lower bound.
\begin{theorem}\label{thm:density-minimax-lowerbound}
For every fixed $\beta>0$, there exists $\gamma^*>0$ such that
\[
\inf_{\alpha>0}\,\sup_{f\in\mathcal{B}(\beta,\gamma^*)}
R_n(\alpha,f)\gtrsim n^{-2\beta/(2\beta+1)}
\qquad\text{as }n\to\infty.
\]
\end{theorem}
We break down the proof into several lemmas. 

\begin{lemma}\label{lem:density-lowerbound-1}
There exist \(\gamma_1>0\), \(f\in\mathcal B(\beta,\gamma_1)\), and
$c_1,c_2>0$ such that for all $\alpha\ge 1$,
\[
R_n(\alpha, f)\ge c_1\frac{\alpha}{n} -c_2\frac{1}{n}.
\]
\end{lemma}

\begin{proof}
Choose a nonnegative function \(f\in C_c^\infty((1,2))\) such that
\(\int_0^\infty f=1\), extend it by \(0\) outside \((1,2)\), and let $\gamma_1\coloneqq\lVert f\rVert_{\mathcal H(\beta)}^2$. By \eqref{eq:density-decomp}, we obtain that
\[
R_n(\alpha, f) \ge \frac{2}{n\pi} \int_0^\alpha  \E_{f}[\sin^2(zX_1)]\,dz - \frac{1}{n}\int_0^\alpha|\theta(z)|^2\,dz .
\]
For the first term, we apply Tonelli's theorem to obtain
\[
\int_0^\alpha  \E_{f}[\sin^2(zX_1)]\,dz =\E_{f} \int_0^\alpha  [\sin(zX_1)]^2\,dz  = \frac{\alpha}{2} - \E_{f}\biggl[\frac{\sin(2\alpha X_1)}{4X_1}\biggr].
\]
Since $X_1\in [1,2]$ almost surely, 
\[
\biggl \lvert \E_{f}\biggl[\frac{\sin(2\alpha X_1)}{4X_1}\biggr]\biggr \rvert \le \frac{1}{4}\E_{f}\biggl[ \frac{1}{X_1}\biggr] \le \frac{1}{4}, \quad \text{so } \int_0^\alpha  \E_{f}[\sin^2(zX_1)]\,dz\ge \frac{\alpha}{2} -\frac{1}{4}.
\]
Since $\int_0^\alpha \lvert \theta(z) \rvert^2\,dz \le  \lVert f\rVert^2_{L^2(0,\infty)}$, we obtain that
\[
R_n(\alpha, f) \ge \frac{1}{n} \Big(\frac{\alpha}{\pi} -\frac{1}{2\pi}  - \lVert f\rVert^2_{L^2(0,\infty)} \Big).
\]
It suffices to let $c_1 = \frac{1}{\pi}$ and $c_2=\frac{1}{2\pi}+\lVert f\rVert^2_{L^2(0,\infty)}$.
\end{proof}
\begin{lemma}\label{lem:density-lowerbound-2}
For $\beta>0$, there exists $\gamma^* > 0$, $c_3>0$, and a sequence of densities $\{f_k\}_{k=1}^\infty\subset \mathcal{B}(\beta,\gamma^*)$, such that for sufficiently large $k$,
\[
\int_{k/2}^\infty|(\mathcal{S}f_k)(z)|^2\, dz\ge c_3 k^{-2\beta}.
\]
\end{lemma}
\begin{proof}
Let $\rho$, $g\in \mathcal{H}(\beta)$ satisfy
\begin{enumerate}[\normalfont(i)]
\item $\rho\vert_{(1,2)}$, $g\vert_{(1,2)} \in C_c^\infty((1,2))$  and $\rho\lvert_{(0,1)\cup(2,\infty)} = g\lvert_{(0,1)\cup(2,\infty)}=0$;
\item $\rho$ is a density;
\item $\rho\ge c_0$ in $\supp g$ for some $c_0>0$;
\item $g\not\equiv 0$ and $\int_0^\infty g = 0$.
\end{enumerate}
For $k\ge 1$, let $c_k\coloneqq \int_0^\infty g(x)\cos(kx)\,dx$ and $h_k(x)\coloneqq g(x) \cos(kx) -c_k \rho(x)$. By construction, $\int_0^\infty h_k(x)\,dx = 0$. Furthermore, for $x\in\supp g$, 
\[
\lvert h_k(x) \rvert  \le \lvert g(x)\rvert + \lvert c_k\rvert \lvert \rho(x) \rvert \le \lVert g\rVert_\infty + \lvert c_k \rvert \lVert \rho\rVert_\infty 
\]
By the Riemann--Lebesgue lemma $\lvert c_k \rvert\to 0$, so for sufficiently large $k$, there exists $M >0$ such that $|h_k(x)|\le M$ for all $x\in \supp g$. Fix $0<\delta<c_0/2M$ and let $f_k\coloneqq \rho+\delta k^{-\beta} h_k$. By construction, $\int_0^\infty f_k=1$ and for all $x \in \supp g$,
\[
f_k(x) \ge \rho(x) -\delta k^{-\beta}  \lvert h_k(x)\rvert \ge c_0- \delta M > 0.
\]
For $x\notin \supp g$, $f_k(x)=\rho(x)(1-\delta k^{-\beta} c_k)\ge 0$ for sufficiently large $k$ by Riemann--Lebesgue lemma again. Therefore, $f_k$ is a density function.

We now proceed to show that $f_k \in \mathcal{B}(\beta,\gamma^*)$ for some $\gamma^* > 0$. Following the literature in probability theory, let $\mathcal{F}$ denote the renormalized Fourier transform
\[
\mathcal{F}\psi(z)=\frac{1}{\sqrt{2\pi}}\int_{\R}\psi(x)\,e^{izx}\,dx .
\]
For a function $\phi$ on $(0,\infty)$, let
$\tilde\phi\coloneqq \phi(x)\1_{\{x\ge 0\}}-\phi(-x)\1_{\{x<0\}}$ be its odd extension to $\R$. Since $\tilde\phi(x)\cos(zx)$ is odd and $\tilde\phi(x)\sin(zx)$ is even, we obtain
\[
\mathcal{F}\tilde\phi(z)
=\frac{i}{\sqrt{2\pi}}\int_{\R}\tilde\phi(x)\sin(zx)\,dx
=\frac{2i}{\sqrt{2\pi}}\int_0^\infty \phi(x)\sin(zx)\,dx
=i\,(\mathcal{S}\phi)(z),
\]
Applying it to $g_0(x) =g(x)\cos(kx)$, whose odd extension is
$\tilde g_0(x) =\tilde g(x)\cos(kx)$, we obtain that
\[
\lvert (\mathcal{S}g_0)(\lvert z \rvert)\rvert^2 =  \lvert (\mathcal{F} \tilde g_0)(z)\rvert^2 \le \frac{1}{2} \lvert (\mathcal{F}\tilde g)(z-k)\rvert^2 + \frac{1}{2}\lvert (\mathcal{F}\tilde g)(z+k)\rvert^2
\]
Therefore
\begin{align*}
\lVert g_0 \rVert_{\mathcal{H}(\beta)}^2 &= \int_0^\infty(1+z^2)^\beta \lvert (\mathcal{S}g_0)(z) \rvert^2\,dz\\
&\le \frac{1}{4}\int_\R (1+z^2)^\beta \lvert(\mathcal{F}\tilde g)(z-k)\rvert^2\,dz + \frac{1}{4}\int_\R(1+z^2)^\beta \lvert(\mathcal{F}\tilde g)(z+k)\rvert^2\,dz
\end{align*}
For first the first term, we apply the change of variable $u=z-k$ and Peetre's inequality \cite[Lemma~1.18]{saintraymond1991} to obtain
\begin{align*}
 \int_\R (1+z^2)^\beta & \lvert (\mathcal{F}\tilde g)(z-k)\rvert^2\,dz = \int_\R (1+(u+k)^2)^\beta \lvert(\mathcal{F}\tilde g)(u)\rvert^2\,du \\&\le 2^\beta (1+k^2)^\beta \int_\R (1+u^2)^\beta \lvert(\mathcal{F}\tilde g)(u)\rvert^2\,du \lesssim k^{2\beta}.
\end{align*}
Using the change of variable $u=z+k$ to the second term, we obtain $\lVert g_0\rVert_{\mathcal{H}(\beta)}\lesssim k^\beta$. Since $h_k=g_0-c_k\rho$, the triangle inequality gives 
\[
\lVert h_k\rVert_{\mathcal{H}(\beta)}\le \lVert g_0\rVert_{\mathcal{H}(\beta)}+\lvert c_k\rvert \lVert \rho\rVert_{\mathcal{H}(\beta)}\lesssim k^\beta.
\]
Therefor, there exists $\gamma^* > 0$ such that $f_k\in \mathcal{B}(\beta,\gamma^*)$ for sufficiently large $k$ as
\[
\lVert f_k\rVert_{\mathcal{H}(\beta)}\le \lVert \rho\rVert_{\mathcal{H}(\beta)} + \delta k^{-\beta} \lVert h_k\rVert_{\mathcal{H}(\beta)}=O(1).
\]

Since $g\not\equiv 0$, $\tilde g\not\equiv 0$ and $\mathcal{F}\tilde g\not\equiv 0$. Therefore there is $r>0$ s.t. $\int_{\lvert z \rvert\le r} \lvert(\mathcal{F} \tilde g)(z)\rvert^2\,dz>0$. Let $k>2r$ and $I_k=[k-r,k+r] \subset [k/2, \infty)$. Applying the inequality $\lvert a-b \rvert^2\ge \frac{1}{2}\lvert a \rvert^2- \lvert b \rvert^2$, and let $u=z-k$ and $v=z+k$, we obtain 
\begin{align*}
\int_{k/2}^\infty |(\mathcal{S}g_0)(z)|^2\,dz &\geq\int_{I_k} \lvert (\mathcal{S}g_0)(z)\rvert^2\,dz \ge \frac{1}{8} \int_{I_k} \lvert (\mathcal{F}\tilde g)(z-k)\rvert^2 \,dz-\frac{1}{4} \int_{I_k} \lvert (\mathcal{F}\tilde g)(z+k)\rvert^2\,dz \\
&=\frac{1}{8} \int_{\lvert u \rvert\le r} \lvert(\mathcal{F}\tilde g)(u)\rvert^2 \,du -\frac{1}{4} \int_{\lvert v-2k \rvert\le r} \lvert (\mathcal{F}\tilde g)(v)\rvert^2\,dv.
\end{align*}
Since $\mathcal{F}\tilde g\in L^2(\mathbb{R})$, $\int_{|v-2k|\le r}|(\mathcal{F}\tilde g)(v)|^2\,dv\to0$ as $k\to\infty$. For sufficiently large $k$ it is at most $b_0\coloneqq \frac{1}{4} \int_{|u|\le r}|(\mathcal{F}\tilde g)(u)|^2\,du$, hence 
\[
\int_{k/2}^\infty |(\mathcal{S}g_0)(z)|^2\,dz \ge \frac{b_0}{4}>0.
\]
Because $\lvert c_k \rvert\to0$, for all sufficiently large $k$ we have $\lvert c_k \rvert^2\int_0^\infty\rho(x)^2\,dx\le b_0/16$. Therefore,
\begin{align*}
\int_{k/2}^\infty \lvert (\mathcal{S} h_k)(z)\rvert^2\,dz &\ge \frac{1}{2} \int_{k/2}^\infty \lvert (\mathcal{S}g_0)(z) \rvert^2\,dz- \lvert c_k\rvert^2 \int_{k/2}^\infty \lvert (\mathcal{S}\rho)(z) \rvert^2\,dz \\
&\ge \frac{1}{2} \int_{k/2}^\infty \lvert (\mathcal{S}g_0)(z) \rvert^2\,dz- \lvert c_k \rvert^2 \int_{0}^\infty \rho(x)^2\,dx  \ge \frac{b_0}{16}>0.
\end{align*}
Therefore,
\[
  \int_{k/2}^\infty |(\mathcal{S}f_k)(z)|^2\,dz  \ge \frac{1}{2} \delta^2k^{-2\beta}  \int_{k/2}^\infty |(\mathcal{S}h_k)(z)|^2\,dz - \int_{k/2}^\infty |(\mathcal{S}\rho)(z)|^2\,dz \ge c_3k^{-2\beta}.\qedhere
\]
\end{proof}

\begin{corollary}\label{cor:density-lowerbound-3}
For every $\beta>0$, there exist constants $\gamma^*>0$, $\alpha_0\ge1$, and $c^*,C^*>0$ independent of
$n$ and $\alpha$, such that for every $n\ge1$ and
$\alpha\ge\alpha_0$,
\[
\sup_{f\in\mathcal{B}(\beta,\gamma^*)}R_n(\alpha,f)
\ge c^*\alpha^{-2\beta}+c^*\frac{\alpha}{n}-\frac{C^*}{n}.
\]
\end{corollary}
\begin{proof}
Up to manipulating the radius, we assume without loss of generality the densities in both Lemma~\ref{lem:density-lowerbound-1} and Lemma~\ref{lem:density-lowerbound-2} belong to the same ball. By
Lemma~\ref{lem:density-lowerbound-2}, for sufficiently large $\alpha$, $k=\lceil2\alpha\rceil$. Since $k/2\ge\alpha$ and $k\le3\alpha$, \eqref{eq:density-decomp}
gives
\[
\sup_{f\in\mathcal{B}(\beta,\gamma^*)}R_n(\alpha,f)
\ge R_n(\alpha,f_k)
\ge\int_{k/2}^\infty|(\mathcal{S}f_k)(z)|^2\,dz
\ge c_3k^{-2\beta}
\ge c_3\,3^{-2\beta}\alpha^{-2\beta}.
\]
On the other hand, Lemma~\ref{lem:density-lowerbound-1} gives
\[
\sup_{f\in\mathcal{B}(\beta,\gamma^*)}R_n(\alpha,f)
\ge c_1\frac{\alpha}{n}-\frac{c_2}{n}.
\]
Taking the maximum of these two bounds yields
\[
\sup_{f\in\mathcal{B}(\beta,\gamma^*)}R_n(\alpha,f)
\ge\frac{c_3\,3^{-2\beta}}{2}\alpha^{-2\beta}
  +\frac{c_1}{2}\frac{\alpha}{n}-\frac{c_2}{2n}.
\]
The result follows with
$c^*\coloneqq\tfrac12\min\{c_3\,3^{-2\beta},c_1\}$ and
$C^*\coloneqq c_2/2$.
\end{proof}

\begin{proof}[Proof of Theorem~\ref{thm:density-minimax-lowerbound}]
Let $\gamma^*$, $\alpha_0$, $c^*$, and $C^*$ be given by
Corollary~\ref{cor:density-lowerbound-3}, and $p\coloneqq2\beta/(2\beta+1)\in(0,1)$. For every $\alpha\ge\alpha_0$, set $t\coloneqq\alpha n^{-1/(2\beta+1)}$. Then
\[
\alpha^{-2\beta}+\frac{\alpha}{n}
=n^{-p}\bigl(t^{-2\beta}+t\bigr)\ge n^{-p},
\]
because $t^{-2\beta}\ge1$ when $0<t\le1$, while $t\ge1$
when $t\ge1$. By Corollary~\ref{cor:density-lowerbound-3},
\[
\sup_{f\in\mathcal{B}(\beta,\gamma^*)}R_n(\alpha,f)
\ge c^*n^{-p}-\frac{C^*}{n}
\qquad\text{for every }\alpha\ge\alpha_0.
\]
Since $p<1$, for all sufficiently large $n$,
\[
\inf_{\alpha\ge\alpha_0}
\sup_{f\in\mathcal{B}(\beta,\gamma^*)}R_n(\alpha,f)
\ge\frac{c^*}{2}n^{-p}.
\]

For $0<\alpha\le\alpha_0$, choose $k_*$ sufficiently large that $k_*/2\ge\alpha_0$ and Lemma~\ref{lem:density-lowerbound-2} applies to $f_{k_*}$.
By \eqref{eq:density-decomp},
\[
  \sup_{f\in\mathcal{B}(\beta,\gamma^*)}R_n(\alpha,f)
  \ge R_n(\alpha,f_{k_*}) \ge\int_\alpha^\infty|(\mathcal{S}f_{k_*})(z)|^2\,dz \ge\int_{k_*/2}^\infty|(\mathcal{S}f_{k_*})(z)|^2\,dz \ge c_3k_*^{-2\beta}>0.
\]
Combining the two ranges, for all sufficiently large $n$,
\[
\inf_{\alpha>0}\sup_{f\in\mathcal{B}(\beta,\gamma^*)}
R_n(\alpha,f)
\ge\min\{c_3k_*^{-2\beta},c^*/2\}\,n^{-p}. \qedhere
\]
\end{proof}

\section{Singular value decomposition of Black--Scholes operator}\label{sec:finance}

Consider the Black--Scholes model \cite{Black1973, Merton1973} from mathematical finance. The market consists of a risk-free asset with interest rate $r > 0$ and a risky asset whose price follows a geometric Brownian motion with volatility $\sigma > 0$. A European option written on the risky asset pays $h(S)$ at the maturity $T$, where $S$ is the asset price at maturity; the function $h$ is called the \emph{payoff}. By a hedging argument, the price $V(S,t)$ of the option, when the asset price is $S$ at calendar time $t\le T$, solves the Black--Scholes equation
\begin{equation}\label{eq:bs-pde}
  \partial_t V + \frac{1}{2}\sigma^2 S^2\,\partial_S^2 V + rS\,\partial_S V - rV = 0,
  \qquad V(S,T) = h(S).
\end{equation}
Writing $\tau = T-t$ for the \emph{time to maturity} and
$C(S,\tau) \coloneqq V(S,T-\tau)$, we rewrite \eqref{eq:bs-pde} as the forward Cauchy problem
\[\partial_\tau C = A_{\mathrm{BS}}\,C,\qquad C(\,\cdot\,,0)=h,
\qquad\text{where}\quad
A_{\mathrm{BS}} = \frac{1}{2}\sigma^2 S^2\,\partial_S^2 + rS\,\partial_S - rI.\]
Since $S\partial_S$ satisfies $S\partial_S\,S^\beta = \beta S^\beta$ for every $\beta\in\C$, and the product rule gives 
\begin{equation}\label{eq:bs-euler}
  A_{\mathrm{BS}}
  = \frac{1}{2}\sigma^2\bigl[(S\partial_S)^2 - S\partial_S\bigr] + r\,S\partial_S - rI
  = \frac{1}{2}\sigma^2 (S\partial_S)^2 + \bigl(r-\frac{1}{2}\sigma^2\bigr)S\partial_S - rI.
\end{equation}
As one may notice, $A_{\mathrm{BS}}$ is not yet an unbounded operator since no Hilbert space or dense domain has been fixed. We now proceed to construct a family of closed, densely defined operators such that each $A^{(\alpha)}$ agrees with $A_{\mathrm{BS}}$ on $C_c^\infty((0,\infty))$. For $\alpha\in\R$, let
\[
  \mathcal H_\alpha \coloneqq L^2\bigl((0,\infty),\,S^{2\alpha-1}\,dS\bigr),
  \quad
  (Q_\alpha f)(x) \coloneqq e^{\alpha x}f(e^x),\quad S=e^x,
\]
where $\alpha$ is referred to as the exponential \emph{damping} parameter \cite[p.~69]{CarrMadan1999}.

\begin{lemma}\label{lem:Talpha-unitary}
For every $\alpha\in\R$, $Q_\alpha\colon \mathcal H_\alpha\to L^2(\R)$ is a unitary operator with $(Q_\alpha^{-1}g)(S)=S^{-\alpha}g(\log S)$ and
\begin{equation}\label{eq:euler-conj}
  Q_\alpha\,(S\partial_S)\,Q_\alpha^{-1} = \partial_x - \alpha .
\end{equation} 
\end{lemma}

\begin{proof}
For $f\in\mathcal H_\alpha$, the change of variable $S=e^x$ gives
\[
  \lVert Q_\alpha f\rVert_{L^2(\R)}^2
  = \int_\R e^{2\alpha x}\lvert f(e^x)\rvert^2\,dx
  = \int_0^\infty S^{2\alpha}\lvert f(S)\rvert^2\,\frac{dS}{S}
  = \int_0^\infty \lvert f(S)\rvert^2 S^{2\alpha-1}\,dS
  = \lVert f\rVert_{\mathcal H_\alpha}^2,
\]
so $Q_\alpha$ is isometric. Moreover, $Q_\alpha$ is invertible since for $g\in L^2(\R)$, $f(S) = S^{-\alpha}g(\log S) \in \mathcal H_\alpha$ and $Q_\alpha f = g$. Thus $Q_\alpha$ is unitary. For~\eqref{eq:euler-conj}, let $g\in L^2(\R)$ be smooth and set $f=Q_\alpha^{-1}g$. Then
\[
  (S\partial_S f)(S)
  = S\,\frac{d}{dS}\bigl[S^{-\alpha}g(\log S)\bigr]
  = -\alpha S^{-\alpha}g(\log S) + S^{-\alpha}g'(\log S),
\]
and applying $Q_\alpha$,
\[
  \bigl(Q_\alpha S\partial_S Q_\alpha^{-1} g\bigr)(x)
  = e^{\alpha x}\,e^{-\alpha x}\bigl[g'(x)-\alpha g(x)\bigr]
  = (\partial_x-\alpha)g(x).\qedhere
\]
\end{proof}

By~\eqref{eq:euler-conj}, consider
\[
  B_\alpha \coloneqq Q_\alpha A_{\mathrm{BS}}Q_\alpha^{-1} = \frac{1}{2}\sigma^2\partial_x^2
  + \bigl(r-\sigma^2(\alpha+\frac{1}{2})\bigr)\partial_x
  + \Bigl[\frac{1}{2}\sigma^2(\alpha^2+\alpha)-r(\alpha+1)\Bigr]I .
\]
For convenience in computation, we normalize the Fourier transform on $L^2(\R)$ here so that 
\begin{equation}\label{eq:fourier-conv}
  (\mathcal F g)(k) = \frac{1}{\sqrt{2\pi}}\int_\R g(x)e^{-ikx}\,dx,
  \qquad \mathcal F\,\partial_x\,\mathcal F^{-1} = T_{ik}.
\end{equation}
In particular, $\mathcal F B_\alpha\mathcal F^{-1}=T_{\rho_\alpha}$, where  $\rho_\alpha$ is the \emph{symbol} of $B_\alpha$~\cite[Section~8.3, p.~271]{Hormander1} 
\begin{equation}\label{eq:rho-alpha}
  \rho_\alpha(k)
  = -\frac{1}{2}\sigma^2 k^2
  + i\bigl(r-\sigma^2(\alpha+\frac{1}{2})\bigr)k
  + c_\alpha,
  \qquad
  c_\alpha \coloneqq (\alpha+1)\bigl(\frac{1}{2}\sigma^2\alpha - r\bigr).
\end{equation}
A direct verification gives 
\begin{equation}\label{eq:contour-shift}
  \rho_\alpha(k) = \rho_0(k+i\alpha),
  \qquad
  \rho_0(k) = -\frac{1}{2}\sigma^2 k^2 + i\bigl(r-\frac{1}{2}\sigma^2\bigr)k - r.
\end{equation}

\begin{lemma}\label{lem:bs-domain}
For every $\alpha\in\R$, $\mathcal D(T_{\rho_\alpha}) = \mathcal F\bigl(H^2(\R)\bigr)$.
\end{lemma}

\begin{proof}
Since $\rho_\alpha$ is a quadratic polynomial with leading coefficient
$-\frac{1}{2}\sigma^2\neq 0$, we have $\lvert\rho_\alpha(k)\rvert\sim\frac{1}{2}\sigma^2 k^2$
as $\lvert k\rvert\to\infty$. Consequently, there is a constant $C>0$ with
\[
  \lvert\rho_\alpha(k)\rvert\le C(1+k^2)\quad\text{for all }k\in\R,
\]
and constants $c,R>0$ with
\[
  \lvert\rho_\alpha(k)\rvert\ge c(1+k^2)\quad\text{for }\lvert k\rvert\ge R.
\]
For $g\in L^2(\R)$, the bounds give 
\[
  \rho_\alpha g\in L^2\bigl(\{\lvert k\rvert\ge R\}\bigr)
  \iff
  (1+k^2)g\in L^2\bigl(\{\lvert k\rvert\ge R\}\bigr).
\]
Since both $\rho_\alpha$ and $1+k^2$ are bounded on $\{\lvert k\rvert<R\}$, $\rho_\alpha g$, $(1+k^2)g \in L^2\bigl(\{\lvert k\rvert<R\}\bigr)$. Therefore, $\rho_\alpha g\in L^2(\R)$ if and only if $(1+k^2)g\in L^2(\R)$, i.e., 
\[\mathcal D(T_{\rho_\alpha})=\{g\in L^2(\R):(1+k^2)g\in L^2(\R)\}=\mathcal F\bigl(H^2(\R)\bigr). \qedhere \]
\end{proof}
Consider
\begin{equation}\label{eq:bs-alpha}
  A^{(\alpha)} \coloneqq Q_\alpha^{-1}\mathcal F^{-1}T_{\rho_\alpha}\mathcal F\,Q_\alpha, \quad \mathcal{D}(A^{(\alpha)}) = Q_\alpha^{-1}\mathcal{F}^{-1} (\mathcal{D}(T_{\rho_\alpha})) = Q_\alpha^{-1}H^2(\R).
\end{equation}
Since $T_{\rho_\alpha}$ is closed and densely defined, $A^{(\alpha)}$ is a closed, densely defined operator. Notice that $C_c^\infty\bigl((0,\infty)\bigr) \subseteq \mathcal D(A^{(\alpha)})$ since, for $f\in C_c^\infty\bigl((0,\infty)\bigr)$, $\mathcal F Q_\alpha f$ is a Schwartz function and hence $\rho_\alpha\,\mathcal F Q_\alpha f\in L^2(\R)$. In this case, 
\[
  A^{(\alpha)}f
  = Q_\alpha^{-1}\mathcal F^{-1}\bigl(\rho_\alpha\,\mathcal F Q_\alpha f\bigr)
  = Q_\alpha^{-1}\mathcal F^{-1}\mathcal F\bigl(B_\alpha Q_\alpha f\bigr)
  = Q_\alpha^{-1}\bigl(B_\alpha Q_\alpha f\bigr)
  = A_{\mathrm{BS}}f .
\]

For $\alpha\in\R$, let the \emph{weighted Mellin transform} be $\mathcal M_\alpha \coloneqq \mathcal FQ_\alpha:
\mathcal H_\alpha\to L^2(\R,dk)$. For
$f\in C_c^\infty(0,\infty)$, the change of variables $S=e^x$ in \eqref{eq:fourier-conv} gives 
\[
(\mathcal M_\alpha f)(k)
=
\frac{1}{\sqrt{2\pi}}
\int_0^\infty S^{\alpha-ik-1}f(S)\,dS, \quad k\in\R.
\]

Since $Q_\alpha$ and $\mathcal F$ are unitary, $\mathcal M_\alpha$ is unitary. The nomenclature is due to its relationship to the Mellin transform \cite[Theorem~71, pp.~94--95]{Titchmarsh1986}, which is the multiplicative analogue of the Fourier transform: Recall that for $f\in C_c^\infty(0,\infty)$ the Mellin transform is
\[
  (\mathcal{M} f)(z)
  =\frac{1}{\sqrt{2\pi}}\int_0^\infty S^{z-1}f(S)\,dS, \quad z\in\C,
\]
so $(\mathcal M_\alpha f)(k)=(\mathcal{M} f)(\alpha-ik)$. In other words, $\mathcal M_\alpha$ is the Mellin transform $\mathcal{M}$ along the line $\Re z=\alpha$.

By definition of $\mathcal M_\alpha$,
\begin{equation}\label{eq:Aalpha-mult}
  A^{(\alpha)} = Q_\alpha^{-1}\mathcal F^{-1}T_{\rho_\alpha}\mathcal F Q_\alpha
  = \mathcal M_\alpha^{*}T_{\rho_\alpha}\mathcal M_\alpha,
\end{equation}
Since the polynomial $\rho_\alpha$ has finitely many real zeros, $\lvert\rho_\alpha\rvert>0$ almost everywhere. The function $\rho_\alpha/\lvert\rho_\alpha\rvert$ is well-defined almost everywhere, and we define it to be $1$ on $\{\rho_\alpha=0\}$. Substituting
$T_{\rho_\alpha}=T_{\rho_\alpha/\lvert\rho_\alpha\rvert}T_{\lvert\rho_\alpha\rvert}$, and substituting into~\eqref{eq:Aalpha-mult}, yields the singular value decomposition
\begin{equation}\label{eq:mellin-svd}
  A^{(\alpha)} = U_\alpha T_{\lvert\rho_\alpha\rvert} V_\alpha^{*},
  \qquad \text{where }
  U_\alpha = \mathcal M_\alpha^{*} T_{\rho_\alpha/\lvert\rho_\alpha\rvert} \text{ and } V_\alpha^{*} = \mathcal M_\alpha.
\end{equation}
Since $T_{\rho_\alpha/\lvert\rho_\alpha\rvert}$ is unitary, $U_\alpha$ is unitary as composition of unitary operators. 

Consider the undamped case of $\alpha=0$, $\mathcal H_0= L^2\bigl ((0,\infty), dS/S\bigr)$. In this case, the risk-free rate is the smallest singular value

\begin{proposition}[The risk-free rate is the smallest singular value]\label{prop:risk-free}
The operator $A^{(0)}: \mathcal{D}(A^{(0)}) \to \mathcal{H}_0 $ is a bijection with $\mathsf{\Sigma}\bigl(A^{(0)}\bigr) = [r,\infty)$ and $\bigl\lVert\bigl(A^{(0)}\bigr)^{-1}\bigr\rVert_{\mathcal L(\mathcal H_0)} = \frac1r$.
\end{proposition}

\begin{proof}
By~\eqref{eq:contour-shift},
\begin{equation}\label{eq:fact}
      \lvert\rho_0(k)\rvert^2
  = \bigl(\frac{1}{2}\sigma^2 k^2+r\bigr)^2 + \bigl(r-\frac{1}{2}\sigma^2\bigr)^2k^2
  = (k^2+1)\Bigl(\frac{1}{4}\sigma^4 k^2 + r^2\Bigr),
\end{equation}
the last equality being an elementary identity in $k$. Since both $(k^2+1)$ and $(\frac{1}{4}\sigma^4 k^2+r^2)$ are continuous, strictly
increasing in $k^2$, and take the values $1$ and $r^2$ at $k=0$, $\lvert\rho_0(k)\rvert^2$ increases continuously from $r^2$ at $k=0$ to $+\infty$ as
$\lvert k\rvert\to\infty$. Therefore, $\lvert\rho_0(k)\rvert^2$ is a continuous surjection of $\R$ onto $[r^2,\infty)$, and its essential range is $[r^2,\infty)$, so $\mathsf{\Sigma}\bigl(A^{(0)}\bigr) = [r,\infty)$. Since $\lvert\rho_0(k)\rvert \ge r>0$ for all $k$,  $\lVert T_{1/\lvert\rho_0\rvert}\rVert=\sup_k \lvert\rho_0(k)\rvert^{-1}
=\bigl(\inf_k \lvert\rho_0(k)\rvert\bigr)^{-1}=1/r$, giving the operator norm of $(A^{(0)})^{-1}$.
\end{proof}

\section{Conclusion}

To the best of our knowledge, this article provides the first systematic study of the singular value decomposition of unbounded operators. The examples throughout the article demonstrate that the singular value decomposition is a useful tool across analysis, geometry, numerical computation, physics, statistics, mathematical finance, and beyond. We hope that the work in this article lays the groundwork for further applications of singular value decomposition to unbounded operators, just as it has long done for matrices and compact operators.

\subsection*{Acknowledgments} This work is partially supported by the Vannevar Bush Faculty Fellowship ONR N000142312863. 

RTW would like to thank Alan Edelman for pointing us towards special functions, resulting in the section on Sturm--Liouville operators.
\bibliographystyle{abbrv}
\bibliography{unbounded}

\end{document}